\documentclass[a4paper]{article}
\usepackage{a4wide}
\usepackage{authblk}

\usepackage{graphicx}
\usepackage{tikz}
\usetikzlibrary{fit,patterns,decorations.pathmorphing,decorations.pathreplacing,calc,arrows}

\usepackage{amsmath,amssymb,amsthm}
\usepackage[inline]{enumitem}

\usepackage{thmtools}

\usepackage{hyperref}
\usepackage[capitalize,nameinlink,sort]{cleveref}

\usepackage{etoolbox}%

\usepackage{mleftright}
\mleftright

\usepackage[T1]{fontenc}
\usepackage{newpxtext}
\usepackage{euler}

\usepackage[textsize=scriptsize]{todonotes}

\DeclareMathOperator{\Fac}{Fac}
\DeclareMathOperator{\N}{\mathbb{N}}
\DeclareMathOperator{\Z}{\mathbb{Z}}

\newcommand{\infw}[1]{\mathbf{#1}}
\newcommand{\fc}[1]{\mathsf{p}_{#1}}
\newcommand{\bc}[2]{
	\expandafter\ifstrequal\expandafter{#2}{1}{%
		 \mathsf{a}_{#1} }{%
		 \mathsf{b}_{#1}^{(#2)} }
	}
\DeclareMathOperator{\pref}{pref}
\DeclareMathOperator{\suff}{suff}
\DeclareMathOperator{\am}{\mathcal{A}_m}
\DeclareMathOperator{\ams}{\mathcal{A}^*_m}
\DeclareMathOperator{\sm}{\sigma_{m}}
\newcommand{\smn}[1]{\sigma_m^{#1}}
\newcommand{\mi}[1]{\overline{#1}}

\declaretheorem[numberwithin=section]{theorem}
\declaretheorem[sibling=theorem]{lemma,corollary,proposition,conjecture}
\declaretheorem[sibling=theorem,style=definition]{example,definition,remark}

\declaretheorem[refname={Claim,Claims},Refname={Claim,Claims}]{claim}

\declaretheoremstyle[
    headfont=\normalfont\itshape, 
    bodyfont = \normalfont,
    qed=$\blacksquare$, 
    headpunct={:}]{claimproofstyle} 
\declaretheorem[name={Proof of claim}, style=claimproofstyle, unnumbered]{claimproof}

\crefname{equation}{}{}

\title{Computing the $k$-binomial complexity of generalized Thue--Morse words}

\author[1]{M. Golafshan}
\author[1]{M. Rigo\thanks{The first two authors are supported by the FNRS Research grant T.196.23 (PDR)}}
\author[2]{M. A. Whiteland\thanks{Part of the work was performed while affiliated with Univeristy of Liège and supported by the FNRS Research grant 1.B.466.21F}}

\affil[1]{Department of Mathematics, University of Li\`ege, Li\`ege, Belgium\\\texttt{\{mgolafshan,m.rigo\}@uliege.be}}
\affil[2]{Department of Computer Science, Loughborough University, Epinal Way, LE11 3TU Loughborough, Leicestershire, United Kingdom\\\texttt{m.a.whiteland@lboro.ac.uk}}

\date{}

\begin{document}
\maketitle            

\begin{abstract}
   Two finite words are $k$-binomially equivalent if each subword (i.e., subsequence) of
  length at most $k$ occurs the same number of times in both
  words. 
  The $k$-binomial complexity of an infinite word is a function that maps the integer  $n\geqslant 0$ to the number of $k$-binomial equivalence classes represented by its factors of length $n$.  

The Thue--Morse (TM) word and its generalization to larger
  alphabets are ubiquitous in mathematics due to their rich combinatorial properties.
  This work addresses the $k$-binomial
  complexities of generalized TM words.
  Prior research by Lejeune, Leroy, and Rigo determined the $k$-binomial complexities of the $2$-letter TM word. For larger alphabets,
  work by Lü, Chen, Wen, and Wu determined the $2$-binomial
  complexity for $m$-letter TM words, for arbitrary $m$, but the exact behavior for $k\geqslant 3$
  remained unresolved. They conjectured that the $k$-binomial complexity function of the $m$-letter TM word is eventually periodic with period $m^k$.
  
  We resolve the conjecture positively by deriving explicit formulae for the $k$-binomial complexity functions for any generalized TM word. We do this by characterizing $k$-binomial equivalence among factors of generalized TM words.
  This comprehensive analysis not only solves the open conjecture, but also develops tools such as abelian Rauzy graphs.
\end{abstract}

\tableofcontents

\section{Introduction}

The Thue--Morse infinite word (or sequence) $\infw{t}_2=011010011001\cdots$ is the fixed point of the morphism $\sigma_2:0\mapsto 01, 1\mapsto 10$ starting with $0$. It was originally constructed by A.~Thue in the context of avoidable patterns. It does not contain any overlap of the form $auaua$ where $a\in\{0,1\}$ and $u\in\{0,1\}^*$. This word was later rediscovered by M.~Morse while studying differential geometry and geodesics on surfaces of negative curvature \cite{Morse1921}. The study of non-repetitive structures is fundamental in combinatorics. See references \cite{MR4046776,MR3032928}
for further details.
 The Thue--Morse word has found applications across a wide range of fields including mathematics, physics, economics, and computer science \cite{JTNB_2015__27_2_375_0,ubiquitous}. In number theory, the word is linked to the Prouhet--Tarry--Escott problem \cite{MR104622}. Additionally,  L.~Mérai and A.~Winterhof
 have analyzed its pseudo-random characteristics; see e.g., \cite{Merai}. The Thue--Morse word also emerges in physics as an example of an aperiodic structure that exhibits a singular continuous contribution to the diffraction pattern \cite{WOLNY2000313,PhysRevB.43.1034}. This  property is  significant in  the study of quasi-crystals and materials with non-periodic atomic arrangements \cite{SAHEL20171} or fractal geometry \cite{Kennard}. In economics or game theory, the Thue--Morse word has been proposed  to ensure fairness in sequential tournament competitions between two agents \cite{Palacios}.

The Thue--Morse word arises in a wide range of unexpected contexts due to its remarkable combinatorial properties.
For instance, consider the study of arithmetic complexity of an infinite word $\infw{w}=w_0w_1w_2\cdots$. This function maps $n$ to  the number of subwords of size $n$ that appear in~$\infw{w}$ in an arithmetic progression, i.e., 
\[n\mapsto \#\{ w_tw_{t+r}\cdots w_{t+(n-1)r}\mid \, t\geqslant 0, r\geqslant 1\}.\]

Let $m\geqslant 2$ be an integer and $\am=\{0,\ldots,m-1\}$ be the alphabet
identified with the additive group $\Z/(m\Z)$. Hereafter, all operations on letters are considered modulo~$m$, and notation  $\pmod{m}$ 
will be omitted.
Avgustinovich et al.~showed that, under some mild assumptions, 
 the fixed point of a {\em symmetric morphism} over $\am$ achieves a maximal arithmetic complexity $m^n$. Such a symmetric morphism $\varphi:\ams\to\ams$ is defined as follows. If $\varphi(0)$ is the finite word $x_0\cdots x_\ell$ over~$\am$, then for $i>0$, $\varphi(i)=(x_0+i)\cdots (x_\ell+i)$,  with all sums taken  modulo~$m$.

This article deals with a natural generalization of the Thue--Morse word over an alphabet of size $m\geqslant 2$. Our primary goal is to identify and count its subwords. It directly relates to the notion of binomial complexity. 
We
consider the symmetric morphism
$\sm:\ams\to\ams$, defined by
\begin{align*}
  \sm:i\mapsto i (i+1) \cdots (i+m-1).
\end{align*}  
With our convention along the paper, integers out of the range $\{0,\ldots,m-1\}$ are reduced modulo~$m$.
The images
$\sm(i)$ correspond to cyclic shifts of the word $012\cdots (m-1)$. For
instance, $\sigma_2$ is the classical Thue--Morse morphism.  Our focus is on the infinite words
$\infw{t}_m:=\lim_{j\to\infty}\smn{j}(0)$. For example, we have
\begin{align*}
\infw{t}_3=012 120 201 120 201 012 201 012 120\cdots.
\end{align*}
Throughout this paper, infinite words are denoted using boldface symbols.
The Thue--Morse word~$\infw{t}_2$ and its generalizations~$\infw{t}_m$ play a prominent role in combinatorics on words~\cite{ubiquitous}. It serves as an example of an $m$-automatic sequence, where each letter is mapped by the morphism  $\sm$ to an image of uniform  length~$m$.
Thus,
 $\sm$ 
 is said to be {\em $m$-uniform}.  The $j^{\text{th}}$ term of $\infw{t}_m$ is equal to the $m$-ary sum-of-digits of $j\geqslant 0$, reduced modulo~$m$. Further results on subwords of $\infw{t}_m$ in arithmetic progressions can be found in \cite{Parshina}.

In this paper, we distinguish between a {\em factor} and a {\em subword} of a word $w=a_1a_2\cdots a_\ell$. 
A factor consists of consecutive symbols
$a_ia_{i+1}\cdots a_{i+n-1}$,  whereas a subword is a subsequence $a_{j_1}\cdots a_{j_n}$, with $1\leqslant j_1<\cdots <j_n\leqslant \ell$. Every factor is a subword, but the converse does not always hold. 
The set of factors of an infinite word 
$\infw{w}$
(respectively, factors of length~$n$) 
is denoted by $\Fac(\infw{w})$
(respectively, $\Fac_n(\infw{w})$).
We denote the length of a finite word~$x$
by~$|x|$,
and
 the number of occurrences of a letter~\(a\)
 in~\(x\)
 by~$|x|_a$.
 For general references on binomial coefficients of words and binomial equivalence, see~\cite{Lothaire,RigoBook,RigoRelations,RigoSalimov}.

\begin{definition}
  Let $u$ and $w$ be words over a finite alphabet~$\mathcal{A}$.  The \emph{binomial
    coefficient} 
    \(\binom{u}{w}\)
    is the number of occurrences of
     $w$  as a subword of $u$.
      Writing $u = a_1\cdots a_n$, where $a_i \in \mathcal{A}$ for all $i$,
      it is defined as
\[
  \binom{u}{w}=\#\left\{ i_1<i_2<\cdots < i_{|w|} \mid \,
    a_{i_1}a_{i_2}\cdots a_{i_{|w|}}=w\right\}.
\]
\end{definition}
Note that the same notation is used for the binomial coefficients of words and integers, as the context prevents any ambiguity (the binomial coefficient of unary words naturally coincides with the integer version: $\binom{a^n}{a^k} = \binom{n}{k}$).

\begin{definition}[\cite{RigoSalimov}] 
Two words $u, v\in \mathcal{A}^*$ are said to be \emph{$k$-binomially equivalent}, and we write $u \sim_k v$, if
\[
  \binom{u}{x} = \binom{v}{x}, \quad \forall\, x\in \mathcal{A}^{\leqslant k}.
\]
If 
\(u\) and
\(v\) are not 
$k$-binomially equivalent,
we write $u\not\sim_k v$.
\end{definition}

A word $u$  is a permutation of the letters in $v$ if and
only if $u \sim_1 v$. 
This relation is known as the \emph{abelian equivalence}.

\begin{definition}
  Let $k\geqslant 1$ be an integer. The \emph{$k$-binomial complexity function} $\bc{\infw{w}}{k} \colon \N \to \N$ for an infinite word $\infw{w}$ is defined as
  \[
    \bc{\infw{w}}{k}: n \mapsto \#\left(\Fac_n(\infw{w})/{\sim_k}\right).
  \]
 
\end{definition}

For
\(k =1\),
the
$k$-binomial complexity
is nothing else but the
 {\em abelian complexity function}, denoted by $\bc{\infw{w}}{1}(n)$.

For instance, M. Andrieu and L. Vivion have recently shown that the $k$-binomial complexity function  is well-suited for studying hypercubic billiard words \cite{Andrieu}. 
These words encode the sequence of faces successively hit by a billiard ball in a 
 $d$-dimensional unit cube.
 The ball moves in straight lines until it encounters a face, then bounces elastically according to the law of reflection.
A notable property is that removing a symbol from a $d$-dimensional billiard word results in a $(d-1)$-dimensional billiard word. 
Consequently, the projected factors of the
$(d-1)$-dimensional
word are subwords of the
$d$-dimensional word.

The connections between binomial complexity and Parikh-collinear morphisms are studied in~\cite{RSW}.

\begin{definition}\label{def:Parikh-collinear}
  Let $\Psi:\mathcal{B}^*\to\N^{\# \mathcal{B}}$, defined as $w\mapsto \left(|w|_{b_1},\ldots,|w|_{b_m}\right)$ be the Parikh map for a totally ordered alphabet~$\mathcal{B}=\{b_1<\cdots<b_m\}$.  A morphism $\varphi\colon \mathcal{A}^* \to \mathcal{B}^*$ is said to be \emph{Parikh-collinear} if, for all letters $a,b \in \mathcal{A}$, there exist constants $r_{a,b}, s_{a,b} \in\mathbb{N}$ such that
  $r_{a,b} \Psi\left(\varphi(b)\right) = s_{a,b} \Psi\left(\varphi(a)\right)$.
  If
  $r_{a,b}=s_{a,b}$
  for all $a,b\in \mathcal{A}$, the morphism is called \emph{Parikh-constant}.
\end{definition}

\begin{proposition}[{\cite[Cor.~3.6]{RSW}}]\label{bkbound}
  Let $\infw{w}$ denote a fixed point of a Parikh-collinear morphism.  For
  any $k \geqslant 1$, there exists a constant $C_k \in \N$  satisfying
  \(\bc{\infw{w}}{k}(n)\leqslant C_{k}\) for all \(n \in \N\).
\end{proposition}

It is worth noting that the above proposition was previously stated for Parikh-constant fixed points in
\cite{RigoSalimov}.


\subsection{Previously known results on generalized Thue--Morse words}
It is well-known that the factor complexity of any automatic word, including the generalized Thue--Morse words, is in $\mathcal{O}(n)$.
The usual factor complexity function of $\infw{t}_m$ is known exactly via results of Starosta \cite{Starosta}:
\begin{theorem}\label{thm:starosta}
For any $m \geq 1$, we have $\fc{\infw{t}_m}(0)=1$, $\fc{\infw{t}_m}(1) = m$, and
\[
\fc{\infw{t}_m}(n) = \begin{cases}
				m^2(n - 1) - m(n - 2) 		& \text{if }2\leqslant n \leqslant m;\\
				m^2(n - 1) - m^{k+1} + m^k 	& \text{if } m^{k} + 1 \leqslant n \leqslant 2m^{k}-m^{k-1},\ k\geq 1;\\
				m^2(n-1)-m^{k+1}+m^k + m\ell	& \text{if } n = 2m^k-m^{k-1} +1 + \ell,\\
				& \text{ with } 0 \leqslant \ell < m^{k+1} - 2m^k + m^{k-1},\ k\geq 1.
\end{cases}				
\]
\end{theorem}
The abelian complexity of $\infw{t}_m$ is known to be ultimately periodic with period $m$,
 as established by Chen and Wen \cite{ChenWen2019}.
For example, $\left(\bc{\infw{t}_2}{1}(n)\right)_{n\geqslant 0}=(1,2,3,2,3,\ldots)$ and $\left(\bc{\infw{t}_3}{1}(n)\right)_{n\geqslant 0}=(1,3,6,7,6,6,7,6,\ldots)$.
Moreover, the period takes either two or three distinct values, depending on the parity of $m$, as described in the following result.

\begin{theorem}[{\cite{ChenWen2019}}]\label{thm:abelian_complexity}
Let $m \geqslant 2$ and  $n\geqslant m$. Let $\nu=n\pmod{m}$.

\begin{itemize}
    \item 
    If $m$ is odd, then we have
    \[
  \bc{\infw{t}_m}{1}(n)=\#\left(\Fac_{n}(\infw{t}_m)/\!\sim_1\right)
  =\begin{cases}
            \frac14m(m^2-1)+1, & \text{ if }\nu=0;\\
            \frac14m(m-1)^2+m, & \text{ otherwise.}
    \end{cases}
    \]
    \item 
    If $m$ is even, then we have
     \[
   \bc{\infw{t}_m}{1}(n)
  = \begin{cases}
            \frac14m^3+1, & \text{ if } \nu=0;\\
            \frac14m(m-1)^2+\frac54m, & \text{ if } \nu\neq 0 \text{ is even};\\
            \frac14m^2(m-2)+m, & \text{ if }\nu\text{ is odd}.\\
          \end{cases}
\]
\end{itemize}
\end{theorem}

It is important to note that the abelian complexity function of a word generated by a Parikh-collinear morphism is not always eventually periodic
\cite{RigoSW2023automaticity}.
Furthermore, \cite{RigoSW2024automatic} shows that the abelian complexity function of such a word is automatic in the sense defined by Allouche and Shallit \cite{AS}.

According to  \cref{bkbound} the $k$-binomial complexity of $\infw{t}_m$ is bounded by a constant (that depends on~$k$). 
Explicit expressions of the functions $\bc{\infw{t}_2}{k}$ have been established:

\begin{theorem}[{\cite[Thm.~6]{LLR}}]\label{thm:kbin2} Let $k\geqslant 1$. For every length $n\geqslant 2^k$, 
the
$k$-binomial complexity
\(
\bc{\infw{t}_2}{k}(n)
\)
is given by
  \[
  \bc{\infw{t}_2}{k}(n)=3\cdot 2^k+\left\{\begin{array}{ll}                                                                                                                                                       -3, & \text{ if }n\equiv 0 \pmod{2^k};\\ 
      -4, & \text{ otherwise}.\\
                                          \end{array}\right.
                                      \]
If $n<2^k$,  the $k$-binomial complexity $\bc{\infw{t}_2}{k}(n)$ is equal to the factor complexity $\mathrm{p}_{_{\infw{t}_m}}(n)$.
\end{theorem}

Let us also mention that infinite recurrent words, where all factors appear infinitely often, sharing the same
 $j$-binomial complexity as the Thue–Morse word $\infw{t}_2$, 
 for all
 \(
 j\leqslant k
 \),
 have been characterized in~\cite{RSW}.

 The authors of~\cite{LLR} conclude  that  ``\ldots the expression of a formula describing the $k$-binomial complexity of $\infw{t}_m$ ($m>2$) seems to be more intricate. Therefore, a sharp description of the constants related to a given Parikh-constant morphism appears to be challenging''.
 
Indeed, the difficulty in obtaining such an expression already becomes apparent with the $2$-binomial complexity. 
In~
\cite{ChenWen2024},
 Lü, Chen, Wen, and Wu derived a closed formula for the $2$-binomial complexity of $\infw{t}_m$.

\begin{theorem}[{\cite[Thm.~2]{ChenWen2024}}]\label{thm:2bin_complexity}
  For every length $n\geqslant m^2$ and alphabet size $m\geqslant 3$,
   the $2$-binomial complexity 
   \(
    \bc{\infw{t}_m}{2}(n)
   \)
   is given by
  \[
    \bc{\infw{t}_m}{2}(n) =\left\{\begin{array}{ll}
        \bc{\infw{t}_m}{1}(n/m)+m(m-1)(m(m-1)+1), & \text{ if }n\equiv 0\pmod{m};\\ \rule{0pt}{2.5ex}
                                    m^4-2m^3+2m^2,& \text{ otherwise}.\\
                                  \end{array}\right.
  \]
\end{theorem}

The authors of~\cite{ChenWen2024} 
propose the conjecture that,
 for all $k\geqslant 3$, the $k$-binomial complexity of the generalized Thue–Morse word $\infw{t}_m$ is ultimately periodic. Precisely,

\begin{conjecture}[{\cite[Conj.~1]{ChenWen2024}}]\label{conj:1}
For every $k\geqslant 3$, the $k$-binomial complexity $\bc{\infw{t}_m}{k}$ of the generalized Thue--Morse word is ultimately periodic with period $m^k$.
\end{conjecture}

 In this paper, we confirm this conjecture by getting the exact expression for the $k$-binomial complexity of $\infw{t}_m$ for alphabet of any size~$m$.


\subsection{Main results}

Let $k\geqslant 2$ and $m\geqslant 2$. The behavior of $\bc{\infw{t}_m}{k}(n)$ depends on the length~$n$ of the factors and is fully characterized by the following three results.

\begin{restatable}{theorem}{shortlengths}\label{thm:main1short}
The shortest pair of distinct factors that are
 $k$-binomially  equivalent have a  length of $2m^{k-1}$. In particular, for any length~$n<2m^{k-1}$, the $k$-binomial complexity $\bc{\infw{t}_m}{k}(n)$ coincides with the factor complexity $\mathrm{p}_{_{\infw{t}_m}}(n)$.
\end{restatable}

Recall \cref{thm:starosta} for an explicit expression for $\mathrm{p}_{_{\infw{t}_m}}(n)$.

\begin{theorem}\label{thm:inter}
  Let~$n\in [2m^{k-1},2m^k)$. 
\begin{enumerate}
    \item
    If 
    \(
    n=\nu\, m^{k-1}
    \)  for some $\nu\in\{2,\ldots,2m-1\}$, then 
\[\bc{\infw{t}_m}{k}(\nu\, m^{k-1})=(m^{k-1}-1) \#E_{m}(\nu)+\bc{\infw{t}_m}{1}(\nu).\]
    \item 
     If 
    \(n=
   \nu\, m^{k-1}+\mu
    \) for some $\nu\in\{2,\ldots,2m-1\}$ and $0<\mu<m^{k-1}$, then 
\begin{align*}
\bc{\infw{t}_m}{k}(\nu\, m^{k-1}+\mu) = (\mu-1)\#E_{m}(\nu+1)  + (m^{k-1}-\mu-1)\# E_{m}(\nu) + \# Y_m(\nu)
\end{align*}
\end{enumerate}
where
    \[
  \#E_{m}(\nu)=\begin{cases}
                      m(1+\nu m-\nu),&\text{ if }\nu<m;\\
                      m^3-m^2+m,&\text{ otherwise}\\
                    \end{cases}\]

and
    \[\#Y_{m}(\nu)=\begin{cases}
                      2m(1+\nu m-\nu)-m\nu(\nu-1),&\text{ if }\nu<m;\\
                      m^3-m^2+2m,&\text{ otherwise.}\\
                    \end{cases}
                \]
\end{theorem}

\begin{theorem}\label{thm:main}
  For every length~$n\geqslant 2m^k$, if $\lambda=n \pmod{m^k}$ and $\lambda=\nu m^{k-1}+\mu$ with $\nu<m$ and $\mu<m^{k-1}$, we have
  \[
    \bc{\infw{t}_m}{k}(n)=(m^{k-1}-1)(m^3-m^2+ m)+
    \begin{cases}
             \bc{\infw{t}_m}{1}(m+\nu), & \text{ if }\mu =0;\\ \rule{0pt}{2.5ex}
             m, & \text{ otherwise}.\\
           \end{cases}
       \]
       In particular, $\left(\bc{\infw{t}_m}{k}(n)\right)_{n\geqslant 2m^k}$ is periodic with period $m^k$.
 \end{theorem}

 Combining the above two theorems, 
 we conclude that the periodic part of 
\(
\bc{\infw{t}_m}{k}(n)
\)
begins at~$m^k$ and therefore answer positively to~\cref{conj:1}.

\begin{corollary}
   The sequence $\left(\bc{\infw{t}_m}{k}(n)\right)_{n\geqslant m^k}$ is periodic with period $m^k$.
 \end{corollary}

\begin{example}
    \cref{fig:complexm3}
     illustrates 
     the $2$- and $3$-binomial complexities of $\infw{t}_3$.
     For short lengths, as described by \cref{thm:main1short},
     the factor complexity is shown using a black dashed line, while values from \cref{thm:inter} 
     are depicted in yellow.
     For larger lengths, values given by \cref{thm:main} are shown in purple and blue, with one period over $[2m^k,3m^k)$ highlighted in purple.
     
     \begin{figure}[h!t]
  \centering
  \includegraphics[width=11cm]{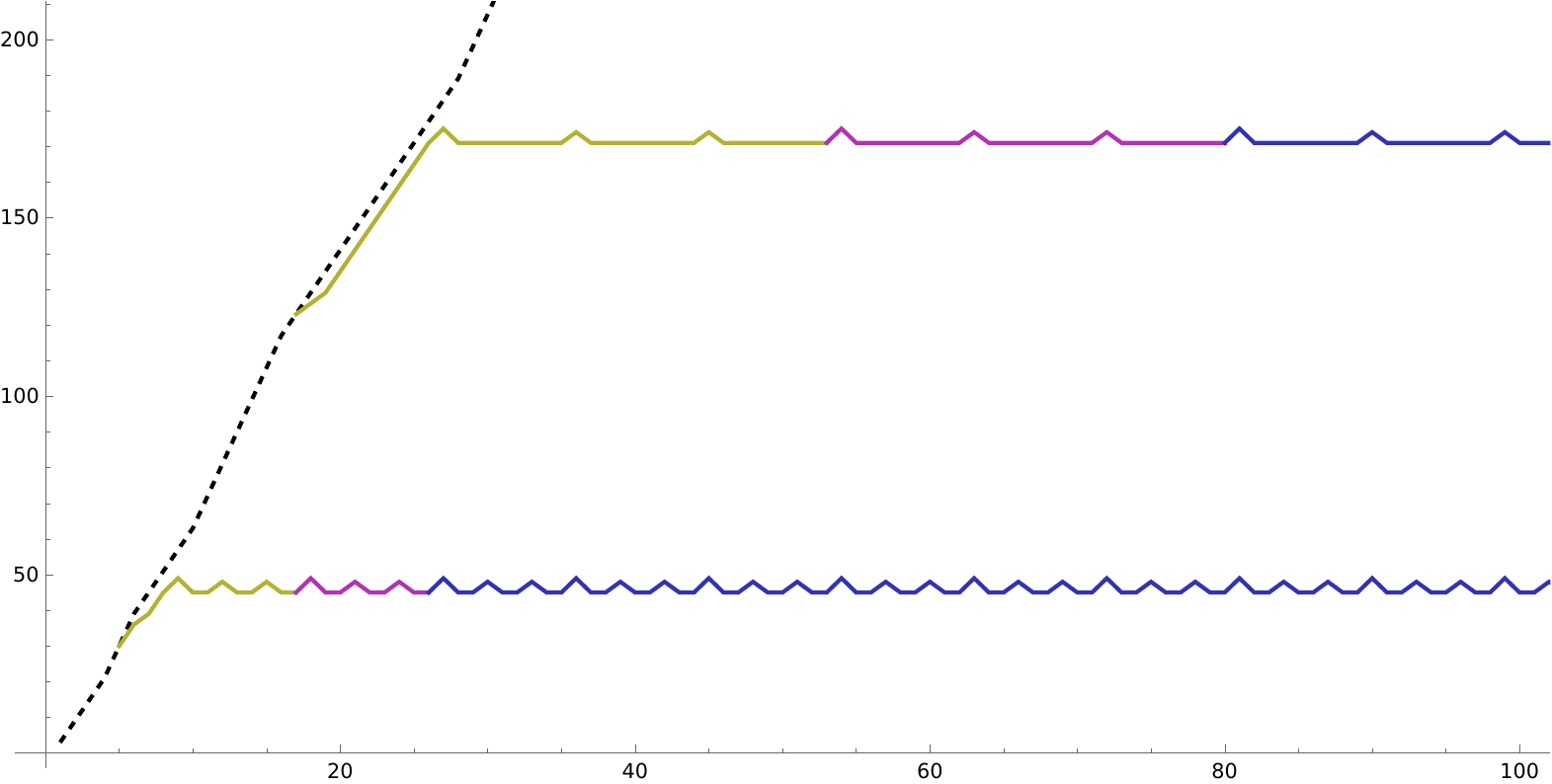}
  \caption{The first few values of the factor complexity (dashed), $2$-, and $3$-binomial complexities of $\infw{t}_3$.}
  \label{fig:complexm3}
\end{figure}

 For $m=3$ and $k=2, \ldots,6$,
\cref{tabt3}
provides the period of the
 $k$-binomial complexity 
 of $\infw{t}_3$,
 where exponents denote repetitions.
 \begin{table}[h!tbp]
 \small{
 \[
   \begin{array}{l}
(49,45^{2},48,45^{2},48,45^{2}); \ 
(175,171^{8},174,171^{8},174,171^{8});\  
(553,549^{26},552,549^{26},552,549^{26}); \\
(1687,1683^{80},1686,1683^{80},1686,1683^{80}); \
(5089,5085^{242},5088,5085^{242},5088,5085^{242})
 \end{array}
\]
}
\caption{The period of $\bc{\infw{t}_3}{k}$ for $k=2,\ldots,6$.}\label{tabt3}
\end{table}

\end{example}

Let us highlight that  \cref{thm:main} simultaneously generalizes the results from \cite{LLR} and \cite{ChenWen2024}. Furthermore, for $k=2$, our formula reduces to \cref{thm:2bin_complexity}. 
We also compute the values of $\bc{\infw{t}_m}{2}(n)$ for the short lengths $n<m^2$.
For $m=2$, \cref{thm:main} provides the following result. For every length $n\geqslant 2^k$, we have:
  \[
    \bc{\infw{t}_2}{k}(n)=3\cdot 2^k+\left\{\begin{array}{ll}
      -6 + \bc{\infw{t}_m}{1}(2), & \text{ if }n\equiv 0 \pmod{2^k};\\  
      -6 + \bc{\infw{t}_m}{1}(3), & \text{ otherwise}.  
                                          \end{array}\right.
\]
This result corresponds to~\cref{thm:kbin2},
with the shortest factors being handled by~\cref{thm:main1short}.


\section{
 Key Points of Our Proof Strategy
}\label{sec:strategy}

The developments presented are relatively intricate. Therefore, we found it useful to schematically outline the main steps of the proof. We hope this provides the reader with a general understanding about the structure of the paper, allowing each section to be read almost independently of the others. This, we believe, makes the paper easier to follow.

\begin{definition}\label{def:factorization}
Let $j\geqslant 1$ and $U$ be a factor of $\infw{t}_m$.
A factorization of the form $U = x\smn{j}(u)y$ is referred to as a
\emph{$\smn{j}$-factorization}  if there exists a factor $aub$ of $\infw{t}_m$, where $a,b \in \am \cup \{\varepsilon\}$.
In this factorization,
 $x$ (respectively,  $y$) must be a proper suffix (respectively, prefix) of $\smn{j}(a)$ (respectively,  $\smn{j}(b)$). Here, 
 $\varepsilon$ is regarded as both a proper prefix and a proper suffix of itself.
\end{definition}

In the literature, 
 the terms
 {\em interpretation in} $\infw{t}_m$ and {\em ancestor}
 are also used. See, for instance, 
 \cite{Frid1998}.

\cref{thm:main} addresses long enough factors.
As discussed in~\cref{sec:rec}, any factor~$U\in\Fac(\infw{t}_m)$ of length $\geqslant 2m^k$  has a unique $\smn{k}$-factorization of the form $p_{_{U}}\smn{k}(u)s_{_{U}}$.
In particular, notice that   
\(|p_{_{U}}|, |s_{_{U}}| <m^k\).
Thus, we can associate each such factor~$U$ with a unique pair  $(p_{_{U}},s_{_{U}})$, 
leading to the following definition.

\begin{definition}\label{def:equivk}
The equivalence relation on~$\mathcal{A}_m^{<m^k}\times \mathcal{A}_m^{<m^k}$ is defined by $(p_1,s_1)\equiv_k (p_2,s_2)$  if there exist $x,y,p,q,r,t\in \ams$ satisfying $|x|,|y|<m^{k-1}$ and 
\begin{eqnarray*}
(p_1,s_1)&=&\left(x \smn{k-1}(p),\smn{k-1}(q) y\right),\\
(p_2,s_2)&=&\left(x \smn{k-1}(r),\smn{k-1}(t) y\right),  
\end{eqnarray*}
and one of the following conditions holds
\begin{itemize}
    \item $pq\sim_1 rt$,
    \item $pq\sim_1 rt\sm(0)$, 
    \item $pq\sm(0)\sim_1 rt$.
\end{itemize}
\end{definition}

We will show the following result in~\cref{sec:discern}.

\begin{restatable}{proposition}{bothdir}\label{pro:both_dir}
    Let $x,y\in\ams$ and $k\geqslant 1$. Then, $x\sim_1 y$ 
    holds
    if and only if $\smn{k}(x)\sim_{k+1}\smn{k}(y)$.
\end{restatable}

To achieve this result, a key challenge was identifying a suitable subword $z$ of length~$k+1$ such that  $x\not\sim_1 y$, implies $\binom{x}{z}\neq\binom{y}{z}$.  \cref{sec:discern} focuses on providing the necessary computations to distinguish non-equivalent factors.

It can easily be shown that if $U,V\in\Fac(\infw{t}_m)$ are factors of length at least $2m^k$  and $(p_{_{U}},s_{_{U}})\equiv_k (p_{_{V}},s_{_{V}})$, then $U\sim_k V$. 
See~\cref{pro:imply}. 
Moreover, the converse of this property is also valid.
However, further developments, as outlined below, are necessary to prove this result.

Assuming, for now, that $(p_{_{U}},s_{_{U}})\equiv_k (p_{_{V}},s_{_{V}})$
if and only if $U\sim_k V$,  proving  \cref{thm:main},  reduces to counting the number  \[\#\, \left\{(p_{_{U}},s_{_{V}})\mid \, U\in \Fac_n(\infw{t}_m)\right\}/\!\!\equiv_k\] of such equivalence classes for $n\geqslant 2m^k$.
This forms the core of~\cref{sec:4} and is given by~\cref{thm:main_equiv}, whose statement is similar to \cref{thm:main}.

To prove that $U\sim_k V$ implies $(p_{_{U}},s_{_{U}})\equiv_k (p_{_{V}},s_{_{V}})$, we first obtain the generalization of \cite[Thm.~2]{ChenWen2024}  originally stated for  $2$-binomial equivalence. This result is then extended to all $k\geqslant 2$.

\begin{restatable}{proposition}{conclusionfinalgeneralization}\label{prop:conclusion-final-generalization}
Let $k\geqslant 2$.
For any two factors $U$ and $V$ of $\infw{t}_m$,
 the relation
$U \sim_k V$
holds if and only if there exist 
 $\smn{k-1}$-factorizations $U = p_{_{U}}\smn{k-1}(u) s_{_{U}}$ and $V = p_{_{V}} \smn{k-1}(v) s_{_{V}}$, such that $p_{_{U}} = p_{_{V}}$, $s_{_{U}} = s_{_{V}}$, and $u \sim_1 v$.
\end{restatable}

We proceed by induction on $k$.
 The base case for $k=2$ is essentially  \cite[Thm.~2]{ChenWen2024}.
 However, our result slightly improves upon that of Chen et al. by not requiring any assumptions about the lengths of
 $U$ and $V$ in the factorizations.

Using
\cref{prop:conclusion-final-generalization},
 we can easily deduce the following result, thereby concluding this part.

\begin{restatable}{proposition}{propconverse}\label{prop:converse}
 Let $k\geqslant 2$. Let $U$ and $V$ be factors of $\infw{t}_m$ with the same length~$\geqslant 2m^k$ such that
 \begin{align*}
      U=p_{_{U}}\smn{k-1}\left(\alpha_{{u}}\sm(u)\beta_{{u}}\right)s_{_{U}},
      \quad
      \text{and}
      \quad
       V=p_{_{V}}\smn{k-1}\left(\alpha_{{v}}\sm(v)\beta_{{v}}\right)s_{_{V}},
 \end{align*}
  where $|p_{_{U}}|,|s_{_{U}}|,|p_{_{V}}|,|s_{_{V}}|<m^{k-1}$
  and $|\alpha_{{u}}|,|\beta_{{u}}|,|\alpha_{{v}}|,|\beta_{{v}}|<m$.
 If
 \(
 U\sim_{k} V
 \),
 then
  \[
     \left(p_{_{U}}\smn{k-1}(\alpha_{{u}}),\smn{k-1}(\beta_{{u}})s_{_{U}}\right)\equiv_{k}\left(p_{_{V}}\smn{k-1}(\alpha_{{v}}),\smn{k-1}(\beta_{{v}})s_{_{V}}\right).
  \]
\end{restatable}

We now focus on factors of length $n\in[2m^{k-1},2m^k)$.
The proof of \cref{thm:inter} relies on analyzing the so-called abelian Rauzy graphs.

\begin{definition}\label{def:abr}
For an infinite word, the {\em abelian Rauzy graph} of order $\ell\geqslant 1$ 
is defined with vertices corresponding to the abelian equivalence classes of factors of length~$\ell$ (or equivalently, to their Parikh vectors). 
The edges of the graph are defined as follows. Let $a,b$ be letters.
If $aUb$ is a factor of length $\ell+1$, there exists a directed edge from $\Psi(aU)$ to $\Psi(Ub)$  labeled $(a,b)$.
\end{definition}

We denote the abelian Rauzy graph of order
$\ell$
of
 $\infw{t}_m$
 by
 $G_{m,\ell}$.
 The number of vertices in
  $G_{m,\ell}$
  is clearly 
  $\bc{\infw{t}_{m}}{1}(\ell)$.
 For all $\ell\geqslant 1$, we define the  following sets:
\begin{eqnarray*}
  Y_{m,R}(\ell)&:=&\left\{\left(\Psi(U),a\right)\mid\, a\in\am, \, Ua\in\Fac_{\ell+1}(\infw{t}_m)\right\},\\
  Y_{m,L}(\ell)&:=&\left\{\left(a,\Psi(U)\right)\mid\, a\in\am,  \, aU\in\Fac_{\ell+1}(\infw{t}_m)\right\},\\
  Y_m(\ell)&:=&Y_{m,R}(\ell)\cup Y_{m,L}(\ell).
\end{eqnarray*}
Since  
$\infw{t}_m=\smn{k-1}(\infw{t}_m)$,  it is quite straightforward to adapt \cite[Prop.~5.5]{RSW}. The idea behind the following formula is that to get $\bc{\infw{t}_m}{k}(j\, m^{k-1}+r)$, one has to count the distinct $\smn{k-1}$-factorizations up to the equivalence relation given by \cref{prop:conclusion-final-generalization}. 
  \begin{proposition}\label{pro:5.5}
Let $k\geqslant 2$.
We let $E_{m}(j)$ denote the set of edges in the abelian Rauzy graph $G_{m,j}$.
For all $j\geqslant 2$ and $0<r<m^{k-1}$, 
the following holds
\[\bc{\infw{t}_m}{k}\left(j\, m^{k-1}\right)=\left(m^{k-1}-1\right) \,\#E_{m}(j)+\bc{\infw{t}_m}{1}(j),\]
and
\[\bc{\infw{t}_m}{k}\left(j\, m^{k-1}+r\right)=(r-1)\, \#E_{m}(j+1) +(m^{k-1}-r-1)\,\#E_{m}(j)+\#Y_m(j).\]
\end{proposition}

The reader may notice that the formula leading to \cref{thm:inter}
requires the values of the abelian complexity for short factors.
However, \cref{thm:abelian_complexity} provides these values only for $j\geqslant m$,
 leaving the case
$j<m$
unaddressed.
Therefore,  in~\cref{sec:abco}, we describe the missing values of $\bc{\infw{t}_m}{1}(j)$ for $j<m$.
In
\cref{sec:arg},
 we proceed to a detailed analysis of the structure of the abelian Rauzy graph of order~$j$.
We are thus able to determine explicit expressions for $\#E_{m}(j)$ and $\#Y_m(j)$.


\section{
Compilation of Preliminary Results
}\label{sec:collecting}

For the sake of completeness, we recall some basic properties of binomial coefficients~\cite{Lothaire,RigoSalimov},
which are implicitly applied throughout this paper.

\begin{lemma}\label{lem:binomial2}
Let $x,y,z$ be three words over the alphabet $\mathcal{A}$. The following relation holds
\[
  \binom{xy}{z}=\sum_{\substack{u,v\in \mathcal{A}^*\\ uv=z}} \binom{x}{u}\binom{y}{v}.
\]
More generally, let $x_1,\ldots,x_\ell$, $z \in \mathcal{A}^*$
and 
$\ell \geqslant 1$. 
 Then, the following relation holds
\[
  \binom{x_1 \cdots x_{\ell}}{z}=\sum_{\substack{e_1,\ldots,e_{\ell}\in \mathcal{A}^*\\ e_1\cdots e_\ell=z}} \, \prod_{i=1}^{\ell}\binom{x_i}{e_i}.
\]

\end{lemma}

\begin{lemma}[Cancellation property]\label{lem:cancel}
Let  $u,v,w$ be three words. 
 The following equivalences hold
 \begin{itemize}
     \item $v \sim_k w$ if and only if $u v \sim_k u w$;

     \item $v \sim_k w$ if and only if $v u \sim_k w u$.
 \end{itemize}
\end{lemma}

We present a few straightforward observations regarding generalized Thue--Morse words.
See, for instance, \cite{Seebold}.

\begin{proposition}[{\cite[Thm.~1]{AlloucheS2000sums}}]\label{prop:GTM-overlap-free}
For any $m\geqslant 2$, the word $\infw{t}_m$ is overlap-free.
\end{proposition}

\begin{lemma}\label{lem:subwords}
Let
$i,j\in\am$.
 If $i<j$ (respectively,  $i>j$),  the word $ij$ appears exactly once as a subword in $m-j+i$ (respectively, $i-j$) of the images $\sm(0), \sm(1), \ldots,\sm(m-1)$.
 Furthermore, the word $ii$  does not occur as a subword in any of these images. Conversely, the $\binom{m}{2}$ 
 distinct
 $2$-subwords appearing in $\sm(j)$ are given by  $(j+t) (j+t+r)$,
 for $t=0,\ldots,m-2$ and $r=1,\ldots,m-t-1$.
\end{lemma}

Let 
 $\tau_m\colon \ams \to \ams$ be the cyclic morphism
 where  each letter $a \in \am$  is mapped to
 $a+1$.
Because the compositions
$\sm \circ \tau_m$ and $\tau_m \circ \sm$ are equal, the following lemma holds.

\begin{lemma}[Folklore]\label{lem:permut}
  For all $n\geqslant 1$, the set $\Fac_n(\infw{t}_m)$ is closed under $\tau_m$.
\end{lemma}

The following result, proven in 
~\cite[Lem.~2]{ChenWen2019},
uses the concept of boundary sequence introduced in~\cite{GuoWen2022}.

\begin{lemma}\label{lem:boundaryseq}
  For all letters  $a,b\in\am$ and all integer $n\geqslant 0$, there exists a factor of $\infw{t}_m$ in the form $awb$, where $|w|=n$. In particular, $\Fac_2(\infw{t}_m)=\mathcal{A}_m^2$. 
\end{lemma}

Since $\sm$ is Parikh-constant, the following result holds.

\begin{proposition}\label{prop:phik} 
  Assume $k\geqslant 1$. For all $u,v\in \ams$, the following hold
  \begin{itemize}
  \item[(i)] If $u\sim_k v$, then $\sm(u)\sim_{k+1}\sm(v)$.
  \item[(ii)] If $u\sim_1 v$, then $\smn{k}(u)\sim_{k+1}\smn{k}(v)$.
  \item[(iii)] If $|u|=|v|$, then $ \smn{k}(u)\sim_{k}\smn{k}(v)$.
      \end{itemize}
\end{proposition}

\begin{proof}
  The first two statements are direct consequences of ~\cite[Prop.~3.9]{RSW}, which applies to any Parikh-collinear morphism.
For all letters $i,j\in \am$, 
 it holds that
   $\sm(i)\sim_1\sm(j)$.
    Hence, if two words $u$ and $v$ have the same length, then $\sm(u)\sim_1\sm(v)$. So statement (iii) follows directly from statement (ii). 
    Therefore, (iii) holds true for any Parikh-constant morphism.
\end{proof}


\section{Ability to Discern 
\texorpdfstring{$k$-Binomially Non-Equivalent Factors}{k-Binomially Non-Equivalent Factors}}
\label{sec:discern}

The purpose of this section is to express differences of the form $\binom{\smn{k}(u)}{x}-\binom{\smn{k}(v)}{x}$  for suitable subwords $x$.  We additionally compute $\binom{\smn{k}(u)}{x}-\binom{\smn{k}(u)}{y}$ for an appropriate choice of $x$ and $y$.

Recall the convention that $\am=\Z/(m\Z)$, meaning any 
 $i\in\Z$ is replaced with $(i\bmod{m})$.
For example, a letter like $(-1)$  is identified as $m-1$.
For convenience, if $a\in\mathbb{N}$, we let $\overline{a}$ denote $-a$.
As an example, with $m=4$, the expression $2(-3)4(-1)=2\overline{3}0\overline{1}$ is indeed $2103$. 
In particular, the word $0 \mi{1} \cdots \mi{k}$ which has length~$k+1$, is a prefix of the periodic word $(0 \mi{1}\, \mi{2} \cdots 1)^\omega$.

In the following statement,  the letter~$0$  does not have any particular role.  By \cref{lem:permut},
one can instead consider $\smn{k}(i)$ and the subword $i (i-1)\cdots (i-k)$. This kind of result is particularly useful for proving that two factors are not $(k+1)$-binomially equivalent.

\begin{proposition}\label{prop:-1}
  Let $m\geqslant 2$ and $k\geqslant 1$. Then for all $j\in\mathcal{A}_m \setminus \{0\}$,
the following holds
  \[
    \binom{\smn{k}(0)}{0 \mi{1} \cdots \mi{k}}-\binom{\smn{k}(j)}{0 \mi{1} \cdots \mi{k}}=m^{\binom{k}{2}}.
  \]
    In particular, the  coefficients $\binom{\smn{k}(j)}{0 \mi{1} \cdots \mi{k}}$  are identical for all  $j\neq 0$.
\end{proposition}

As an example, for the classical Thue--Morse morphism, where $m=2$,
it follows that
 $\mi{1}=1$.  We have: 
 \[
 \binom{\sigma_2^{2n}(0)}{(01)^n0}-\binom{\sigma_2^{2n}(1)}{(01)^n0}=2^{n(2n-1)} 
 \]
 and
\[
  \binom{\sigma_2^{2n+1}(0)}{(01)^{n+1}}-\binom{\sigma_2^{2n+1}(1)}{(01)^{n+1}}=2^{n(2n+1)}.
\]

 \begin{proof}
     We proceed by induction on $k$. 
     For the base case $k=1$, \cref{lem:subwords} shows that the subword $0\mi{1}$ occurs exactly once in $\sm(0)$ and does not appear in any other $\sm(j)$ for $j\neq 0$.
    Assume that the statement holds for some $k\geqslant 1$.
    We now prove it for $k+1$.

The word $u=\smn{k+1}(0)$ can be factorized into $m$ consecutive words, each of length~$m^k$ (referred to as \emph{$m^k$-blocks}),
as follows: $u=\smn{k}(0)\smn{k}(1)\cdots \smn{k}(\mi{1})$. Similarly, the word $v=\smn{k+1}(j)$ is a cyclic permutation of the $m^k$-blocks of $u$,
given by
\[
  v=\smn{k}(j) \cdots \smn{k}(\mi{1})\smn{k}(0)\cdots \smn{k}(j-1).
\]
Our task is to count (or at least compare, as we are only interested in the difference) the occurrences of subwords $w=0 \mi{1}\cdots \mi{k}\,\mi{k+1}$ of length $k+2$  in  $u$ and $v$.

First, the number of occurrences fully contained within a single $m^k$-block is identical in $u$ and $v$ because 
they have the same
 $m^k$-blocks.

 Next, we count the occurrences of $w$ that are split across more than one $m^k$-block. 
These occurrences can be categorized into two cases:
    \begin{itemize}
    \item[I)] $w$ is split across at least two blocks, with no more than $k$ letters of $w$ appearing in each $m^k$-block.
     \cref{prop:phik}  ensures that $\smn{k}(i)\sim_{k}\smn{k}(i')$ for all letters $i$ and $i'$.
     So  $u$ and $v$
      contain the same number of these types of occurrences.
      \item[II)] $w$ is split across at least two blocks, with $k+1$ letters of $w$ appearing within a single  $m^k$-block.
    \end{itemize} 
    A difference arises only when $k+1$ letters of $w$ appear within a single $m^k$-block, while its first or last letter belongs to a different $m^k$-block.
By induction hypothesis, $\binom{\smn{k}(i)}{0 \mi{1} \cdots \mi{k}}=\binom{\smn{k}(i')}{0 \mi{1} \cdots \mi{k}}$ for any $i,i'\neq 0$.
 Similarly, $\binom{\smn{k}(i)}{\mi{1} \cdots \mi{k+1}}=\binom{\smn{k}(i')}{\mi{1} \cdots \mi{k+1}}$ for $i,i'\neq \mi{1}$.  So to get different contributions, we only focus where the blocks $\smn{k}(0)$ and $\smn{k}(1)$ occur in $u$ and $v$.
    
Let us first consider $\smn{k}(0)$.
 It appears at the beginning of $u$ and it contains the subword
 $0 \mi{1}\cdots \mi{k}$
 exactly
 $\binom{\smn{k}(0)}{0 \mi{1} \cdots \mi{k}}$ times.
 Moreover $\mi{k+1}$ occurs once in every of the subsequent $(m-1)m^{k-1}$ blocks of length~$m$ within $\smn{k}(1)\cdots \smn{k}(\mi{1})$. 
 However, the first $m^k$-block in $v$ is $\smn{k}(j)$, where the subword $0 \mi{1}\cdots \mi{k}$ appears only $\binom{\smn{k}(j)}{0 \mi{1} \cdots \mi{k}}$ times.
By induction hypothesis, the resulting difference is
\[
  m^{\binom{k}{2}} (m-1) m^{k-1}.
\]
A similar reasoning applies to  $\smn{k}(\mi{1})$,  which appears as the suffix of $u$  and contains the subword 
$\mi{1}\, \mi{2}\cdots \mi{k+1}$
exactly
$\binom{\smn{k}(\mi{1})}{\mi{1}\, \mi{2} \cdots \mi{k+1}}$ times.
Moreover, $0$ occurs exactly once in each of the preceding $(m-1)m^{k-1}$ blocks of length~$m$ within  $\smn{k}(0)\cdots \smn{k}(\mi{2})$.
Using \cref{lem:permut} and the induction hypothesis, the resulting difference is once again $m^{\binom{k}{2}} (m-1) m^{k-1}$. 

We still have to take into account the contributions of 
 $\smn{k}(0)$ and $\smn{k}(\mi{1})$ within $v$.
 The word $v$ begins with $m-1-j$ 
 blocks of length $m^k$ followed by $\smn{k}(\mi{1})\smn{k}(0)$,
 and ends with $j-1$ blocks of length $m^k$. 
 We have to count the number of $0$'s appearing before $\smn{k}(\mi{1})$ and the $\mi{k+1}$'s appearing after $\smn{k}(0)$.
 There are $(m-1-j)m^{k-1}$ such $0$'s and $(m-3-j)m^{k-1}$ such  $\mi{k+1}$'s.
 By comparing with the blocks occurring in the corresponding position in $u$,
we obtain the following difference
     \[
    \left(\binom{\smn{k}(\mi{j+1})}{\mi{1}\, \mi{2}\cdots \mi{k}}-\binom{\smn{k}(\mi{1})}{\mi{1}\, \mi{2} \cdots \mi{k}}\right) (m-1-j)m^{k-1} + \left(\binom{\smn{k}(\mi{j})}{0 \mi{1} \cdots \mi{k}}-\binom{\smn{k}(0)}{0 \mi{1} \cdots \mi{k}}\right) (j-1)m^{k-1}.
\]
By induction hypothesis, we find that both terms in parentheses are equal to $-m^{\binom{k}{2}}$. 
Therefore, the difference is $-m^{\binom{k}{2}}(m-2)m^{k-1}$. 

Combining the results from the three preceding discussions, we get a total difference of
\[
    2m^{\binom{k}{2}} (m-1) m^{k-1}-m^{\binom{k}{2}}(m-2)m^{k-1}=m^{\binom{k+1}{2}}
\]
matching the expected result.
 \end{proof}

\begin{corollary}\label{cor:notequiv}
  Let $u,v\in\ams$ with the same length. Then,
  \[
    \binom{\smn{k}(u)}{0 \mi{1} \cdots \mi{k}}-\binom{\smn{k}(v)}{0 \mi{1} \cdots \mi{k}}=\left(|u|_0-|v|_0\right)\, m^{\binom{k}{2}}.
  \]
  In particular, if $u\not\sim_1v$, then $\smn{k}(u)\not\sim_{k+1} \smn{k}(v)$.
\end{corollary}

\begin{proof}
There exist words $p,u'$, and $v'$ such that $u\sim_1 p u'$ and $v\sim_1 p v'$,
where
$u'$ and $v'$ share no common letters, and $|u'|=|v'|$.
Let 
$\Psi(u)=(s_1,\ldots,s_m)$ and $\Psi(v)=(t_1,\ldots,t_m)$.
Then,  $p$ is a word such that
$\Psi(p)=\left(\min\{s_1,t_1\},\ldots,\min\{s_m,t_m\}\right)$. 
By  \cref{prop:phik}, $\smn{k}(u)\sim_{k+1} \smn{k}(pu')$. 
Therefore,
\[
  \binom{\smn{k}(u)}{0 \mi{1} \cdots \mi{k}}=\binom{\smn{k}(p u')}{0 \mi{1} \cdots \mi{k}}
  =\sum_{\substack{x,y\in\ams \\ xy=0 \mi{1} \cdots \mi{k}}} \binom{\smn{k}(p)}{x} \binom{\smn{k}(u')}{y}.
\]
Thus,
\[
  \binom{\smn{k}(u)}{0 \mi{1} \cdots \mi{k}}-\binom{\smn{k}(v)}{0 \mi{1} \cdots \mi{k}}
  =\sum_{\substack{x,y\in\ams \\ xy=0 \mi{1} \cdots \mi{k}}} \binom{\smn{k}(p)}{x} \left(\binom{\smn{k}(u')}{y}-\binom{\smn{k}(v')}{y}\right).
\]
Using  \cref{prop:phik}
again,
$\smn{k}(u')\sim_k\smn{k}(v')$.
Therefore, if 
$|y|\leqslant k$,
 we have
\[
  \binom{\smn{k}(u')}{y}-\binom{\smn{k}(v')}{y}=0.
\]
Hence, we conclude
\[
  \binom{\smn{k}(u)}{0 \mi{1} \cdots \mi{k}}-\binom{\smn{k}(v)}{0 \mi{1} \cdots \mi{k}}=
  \binom{\smn{k}(u')}{0 \mi{1} \cdots \mi{k}}-\binom{\smn{k}(v')}{0 \mi{1} \cdots \mi{k}}.
\]
As shown in the proof of  \cref{prop:-1}, since $\smn{k}(i)\sim_k\smn{k}(j)$  for all $i,j\in\am$, a non-zero difference arises only if a subword $0 \mi{1} \cdots \mi{k}$ appears entirely within an $m^k$-block.
More precisely, if $u'=a_1\cdots a_r$ and $v'=b_1\cdots b_r$ where $a_i$'s and $b_j$'s are letters, the difference can be expressed as
\[
  \binom{\smn{k}(u)}{0 \mi{1} \cdots \mi{k}}-\binom{\smn{k}(v)}{0 \mi{1} \cdots \mi{k}}=\sum_{i=1}^r \binom{\smn{k}(a_i)}{0 \mi{1} \cdots \mi{k}}-\sum_{i=1}^r \binom{\smn{k}(b_i)}{0 \mi{1} \cdots \mi{k}}
\]
Using \cref{prop:-1}, 
it follows that
\[
  \binom{\smn{k}(u)}{0 \mi{1} \cdots \mi{k}}-\binom{\smn{k}(v)}{0 \mi{1} \cdots \mi{k}}=(|u'|_0-|v'|_0)\, m^{\binom{k}{2}}.
\]

In the particular case where  $u$ and $v$ are not abelian equivalent, the words $u'$ and $v'$ must be non-empty. W.l.o.g., we assume that 
$0$ appears in $u'$ (and does not appear in $v'$).
The conclusion then follows.
\end{proof}

 By combining \cref{prop:phik,cor:notequiv}, we obtain \cref{pro:both_dir},
 which is restated below.
\bothdir*

\cref{prop:-1}  dealt with subwords of length~$k+1$ occurring in $m^k$-blocks.
The next statement focuses on subwords of length at most $k$ that appear in the image of a word under $\smn{k}$.  This result will play a key role in the proof of \cref{lem:bigdiff}.

\begin{lemma}\label{lem:smart} Let $\ell\leqslant k$. For all $j$, the following holds
  \[
    \binom{\smn{k}(u)}{0 \mi{1} \cdots \mi{\ell-1}}
    =\binom{\smn{k}(u)}{\mi{j} \cdots \mi{j+\ell-1}}
  \]
\end{lemma}

\begin{proof}
  Let $u=a_1\cdots a_t$, where $a_i\in \am$. First of all, we note that trivially
  \[
  \binom{\smn{k}(a_1\cdots a_t)}{\mi{j} \cdots \mi{j+\ell-1}}
    = \binom{\tau_m^{j}(\smn{k}(a_1 \cdots a_t))}{\tau_m^{j}(\mi{j} \cdots \mi{j+\ell-1})},
  \]
  as the subwords occur at the same positions in the respective words.
  Furthermore, we have $\tau_m^{j}(\mi{j} \cdots \mi{j+\ell-1}) = 0 \mi{1} \cdots \mi{\ell-1}$.
  Finally, since $\sm \circ \tau_m = \tau_m \circ \sm$, it follows that
  \[
    \binom{\smn{k}(a_1\cdots a_t)}{\mi{j} \cdots \mi{j+\ell-1}}
    = \binom{\tau_m^{j}(\smn{k}(a_1 \cdots a_t))}{0 \mi{1} \cdots \mi{\ell-1}}
    = \binom{\smn{k}(\tau_m^{j}(a_1 \cdots a_t))}{0 \mi{1} \cdots \mi{\ell-1}} = \binom{\smn{k}((a_1+j)\cdots (a_t+j))}{0 \mi{1} \cdots \mi{\ell-1}}.
  \]
  Furthermore, by
  \cref{prop:phik}(iii),
  we know that
  $
  \smn{k}(a_1\cdots a_t) \sim_k \smn{k}\left((a_1+j)\cdots (a_t+j)\right).
  $ Hence, the desired result.
\end{proof}

The next lemma is presented in its full generality. For the sake of presentation, the proof is given in \cref{sec:appbigfiff}.

  \begin{lemma}\label{lem:bigdiff}
Let $k\geqslant 2$. 
Suppose
$u,u',\gamma,\gamma',\delta,\delta'\in\ams$ 
are words such that
 $\gamma\delta\sim_1\gamma'\delta'$ and $|u|=|u'|$. Then, the difference
\[
  \binom{\smn{k-1}(\gamma \sm(u) \delta)}{0 \mi{1}\cdots \mi{k}}-
  \binom{\smn{k-1}(\gamma' \sm(u') \delta')}{0 \mi{1}\cdots \mi{k}}
\]
is given by 
\[
\begin{aligned}
    &m^{\binom{k}{2}} \biggl[ |u|_0 - |u'|_0 + |u| \, \left(|\gamma|_0 - |\gamma'|_0 + |\delta|_{\mi{1}} - |\delta'|_{\mi{1}}\right) \biggr] \\
    &\quad + m^{\binom{k}{2}-1} \sum_{b \in \am} \left( \binom{\gamma\delta}{b\mi{1}} - \binom{\gamma'\delta'}{b\mi{1}} + \binom{\gamma\delta}{0b} - \binom{\gamma'\delta'}{0b} \right).
\end{aligned}
\]
 
\end{lemma}


\section{Recognizability and Structure of Factors}\label{sec:rec}

First, we recall a recognizability property stating that any long enough factor~$U\in\Fac(\infw{t}_m)$ has a unique  $\smn{k}$-factorization 
of the form
$p_{_{U}}\smn{k}(u)s_{_{U}}$,
where
$p_{_{U}}$
and
$s_{_{U}}$
are blocks of length less than $m^k$.
Next, we examine the structure of those pairs 
$\left(p_{_{U}}, s_{_{U}}\right)$
 in detail and show that they are subject to strong constraints. This will allow us to carry out precise counting in \cref{sec:4}.

We summarize some well-known concepts and results (see, for instance, \cite{Balkova2012,Frid1998}).
A morphism~$\varphi$ is called {\em marked} if, for every pair of distinct letters, their images under~$\varphi$ differ  in both the first and last letters.
 A morphism~$\varphi\colon \mathcal{A}^*\to \mathcal{A}^*$ is said to be {\em primitive} if there exists an integer~$n$ such that, for all $a\in \mathcal{A}$, 
 the word
 $\varphi^n(a)$ contains all letters of $\mathcal{A}$.
 
\begin{remark}
  Let $\varphi:\mathcal{A}^*\to \mathcal{A}^*$ be a morphism, and 
  let $n\geqslant 1$ be an integer. If $\varphi$ is marked (respectively, primitive, $\ell$-uniform), then $\varphi^n$ has the same properties, meaning
  $\varphi^n$
  is marked
  (respectively, primitive, $\ell^n$-uniform).
\end{remark}

Note that, for all $k\geqslant 1$, the $k^{\text{th}}$ power of our morphism of interest $\sm$ is such that $\smn{k}(i)$ begins with $i$ and ends with $i-k$. 
Therefore, the morphism~$\smn{k}$ is marked.

Let $\infw{x}$ be a fixed point of a morphism $\varphi$ over $\mathcal{A}$.
A factor $w$ of $\infw{x}$ 
is said to contain a
{\em synchronization point}
$(w_1,w_2)$ if $w=w_1w_2$ and, for all $v_1,v_2\in \mathcal{A}^*$, $s\in\Fac(\infw{x})$ such that $\varphi(s)=v_1w_1w_2v_2$, there exist $s_1,s_2\in\Fac(\infw{x})$ such that $s=s_1s_2$, $\varphi(s_1)=v_1w_1$, and $\varphi(s_2)=w_2v_2$. A factor $w$ that contains a synchronization point is said to be 
{\em circular}.

\begin{proposition}\label{pro:gen-circular}
Let $\varphi$ be an $\ell$-uniform, primitive, marked morphism with $\infw{x}$ as one of its fixed points. If $u$ is a circular factor of
$\infw{x}$, then $u$ has a unique $\varphi$-factorization
 (in the sense of
\cref{def:factorization}).
\end{proposition}

\begin{proposition}\label{pro:circular}
 For all $k\geqslant 1$, the morphism~$\smn{k}$ is an $m^k$-uniform, primitive, marked morphism.
 Moreover, 
  every factor of its fixed point
  $\infw{t}_m$
  that has length at least
 $2m^k$  is circular.
\end{proposition}

\begin{example}
The factor $\sm(0)^2$ of $\infw{t}_m$ has  $m$ factorizations
$\sm(00)$ and
\[
\suff_{j}\left(\sm(j)\right) \cdot \sm(j) \cdot \pref_{m-j}\left(\sm(j)\right), \qquad j=1,\ldots,m-1.
\]
However, only one of these is a valid $\sm$-factorization, namely $\sm(00)$.
This is because 
 $j^3$ does not occur in $\infw{t}_m$ for any $j$ (cf.~\cref{prop:GTM-overlap-free}), implying that none of the other factorizations are valid $\sm$-factorizations.

The factor $\sm(0)01\cdots (m-2)$ which has a length of $2m-1$,
has two possible $\sm$-factorizations:
\[
\sm(0)\cdot \pref_{m-1}\left(\sm(0)\right) \quad \text{and} \quad \suff_{m-1}\left(\sm(m-1)\right) \cdot \sm(m-1).
\]
Recall from \cref{lem:boundaryseq} that $00$ and $(m-1)(m-1)$ are indeed factors of $\infw{t}_m$.
\end{example}

\begin{remark}\label{rem:all-factors-smk-factorization}
For any $k\geqslant 1$, it is obvious that all factors of length at least~$m^k-1$ in $\infw{t}_m$ have a $\smn{k}$-factorization, since the image of a letter has length~$m^k$.
To simplify the arguments in \cref{sec:characterizing}, we extend this observation to all factors.
Namely, for any $k\geqslant 1$, any factor $U$ of $\infw{t}_m$ has a $\smn{k}$-factorization.
We will prove this by induction on $k$.

For $k=1$, the only case to consider is when a factor 
$U$ appears properly within the image of a letter, i.e., $U = \ell \cdots \left(\ell + |U|-1\right)$ for some $\ell \in \am$ with $|U| \leqslant m-2$. Notice that
\[
\pref_j(U) = \suff_j\left(\sm(\ell + j)\right)
\quad \text{and} \quad
\suff_{|U|-j}(U) = \pref_{|U| - j}\left(\sm(\ell+j)\right).
\]
Since all squares $a^2$, where $a \in \am$, appear in $\infw{t}_m$,  it follows that
for each value of 
$j$,
where
$0 \leqslant j \leqslant |U|$,
the word
$U$ has  $|U|+1$  distinct $\sm$-factorizations 
of the form
\[
\suff_{j}(\ell+j) \cdot \sm(\varepsilon) \cdot \pref_{|U|-j}\left(\sm(\ell + j)\right).
\]

Now,
assume that $U$ has a $\smn{k}$-factorization
 of the form
$x \smn{k}(u)y$, where $x$ is a proper suffix of $\smn{k}(a)$ and $y$ is a proper prefix of $\smn{k}(b)$, and $aub$ is a factor of $\infw{t}_m$.
If $u = \varepsilon$,
then we have the $\smn{k+1}$-factorization $x\cdot \smn{k+1}(\varepsilon)\cdot y$. This is valid since $(a+1)b$ is a factor of
$\infw{t}_m$, $\smn{k}(a)$ is a suffix of
$\smn{k+1}(a+1)$, and $\smn{k}(b)$ is a prefix of $\smn{k+1}(b)$. 
Now, assume  $|u|\geqslant 1$, implying $|U| \geqslant m^k$. If $U$ does not appear properly within the $\smn{k+1}$-image of a letter, there is nothing to prove. Thus consider the case that $U$ appears, w.l.o.g., properly within $\smn{k+1}(0) = \smn{k}(0\cdots(m-1))$, which implies $|U| \leqslant m^{k+1}-2$.
We can express 
$U$ as
$U = x'\smn{k}(u')y'$, where $u' = \ell(\ell+1)\cdots (\ell + t)$ for some $\ell \geqslant 1$ and  $t < m-1 - \ell$, with $x'$ being a proper suffix of  $\smn{k}(\ell-1)$,
and $y'$ a proper prefix of $\smn{k}(\ell + t + 1)$. 
Here, we allow $t=-1$ to indicate that $u'$
is empty.
For instance, we obtain the
$\smn{k+1}$-factorization $x' \cdot \smn{k+1}(\varepsilon) \cdot \smn{k}(u')y'$, where $x'$, being a suffix of $\smn{k}(\ell-1)$, is a proper suffix of $\smn{k+1}(\ell)$, and $u'y'$ is a proper prefix of $\smn{k+1}(\ell)$.
As $\ell\ell$ is a factor of $\infw{t}_m$, the conclusion holds. If $x' = \varepsilon$, then we obtain the $\smn{k+1}$-factorization $\varepsilon \cdot \smn{k+1}(\varepsilon) \cdot \smn{k}(u')y'$. This concludes the proof.
\end{remark}

\begin{corollary}\label{cor:unique-factorization-bound}
For all factors $U\in\Fac(\infw{t}_m)$ of length $|U|\geqslant 2m^k$, there exists a unique $\smn{k}$-factorization:
\[
  U=p_{_{U}} \smn{k}(u) s_{_{U}}.
\]
In particular, the words $p_{_{U}}$, $s_{_{U}}$, and $u$ are unique.
\end{corollary}

\begin{proof}
This result follows directly from  \cref{pro:gen-circular,pro:circular}.
\end{proof}

\begin{example}
   Let $m=3$ and $k=2$.
   The word 
   \[
   U=1200121202011202010122010121,
   \]
   which has length~$28$, is a factor of $\infw{t}_3$.
    It can be factorized as: 
   \[
\sigma_3(1)\sigma_3^2(01)\sigma_3(20)1
   \]
where $p_{_{U}}=\sigma_3(1)$
and
$s_{_{U}}=\sigma_3(20)1$. 
\end{example}

Since the word $s_{_{U}}$ is a proper prefix of some $\smn{k}(j)$, it has a specific structure. Since $|s_{_{U}}|<m^k$, this length can be uniquely expressed using a base-$m$ expansion as
\[
  |s_{_{U}}|=\sum_{i=0}^{k-1} c_{k-i}\, m^i,\quad c_1,\ldots,c_k\in\{0,\ldots,m-1\}.
\]
By applying a similar greedy procedure to the word $s_{_{U}}$ (refer to~\cite{DumontThomas} for details on Dumont--Thomas numeration systems associated with a morphism, or~\cite{RigoBook}), we obtain the following unique decomposition
\begin{equation}
     \label{eq:decompsu}
     s_{_{U}}=\prod_{i=1}^k \smn{k-i}\left(v_{i}\right)
\end{equation}
where the words $v_i$ are defined as follows 
\[
  v_i=(j+\sum_{r=1}^{i-1}c_r) \, (j+\sum_{r=1}^{i-1}c_r+1) \cdots (j+\sum_{r=1}^{i}c_r-1).
\]
Notice that $|v_i|=c_i$, and $v_1\cdots v_k$ is a prefix of $\left(j(j+1)\cdots (j+m-1)\right)^\omega$.

\begin{example}
The base-$4$ expansion of $226$ is $3.4^3+2.4^2+2$. The prefix of $\sigma_4^4(0)$ with a length~$226$
is given by
\[
  \sigma_4^3(012)\sigma_4^2(30)12
\]
where $v_1=012$, $v_2=30$, $v_3=\varepsilon$, and $v_4=12$. Thus, $v_1\cdots v_4=0123012$. For instance, $\sigma_4^3(\underline{02})\sigma_4^2(30)12$ is not the prefix of any $\sigma_4^4(a)$,
as it involves applying
 $\sigma_4^3$ 
 to a block composed of non-consecutive letters.
\end{example}

\begin{remark}\label{rem:uniquesu}
Knowing the value of $j$ and 
the length
$|s_{_{U}}|$ uniquely determines the decomposition given in~\eqref{eq:decompsu}. Equivalently, for all $n\geqslant 1$ and letter~$a$, there exists a unique factor of the form $s_{_{U}}$, of length~$n$, that starts (respectively, ends) with the letter~$a$. 
\end{remark}

\begin{corollary}\label{cor:delpu} 
 The collect the following facts.
  \begin{itemize}
  \item[(i)] With the above notation, let $q$ (respectively, $r$) be the least (respectively, largest) integer such that 
 $c_q$ (or $c_r$) is non-zero.
Let $v_q=xy$ and $v_r=zh$, 
 such that
 $v_1\cdots v_k=xy v_{q+1}\cdots v_{r-1}zh$. Then,
    \[
      \smn{k-q}(y)\prod_{i=q+1}^{r-1} \smn{k-i}\left(v_{i}\right)\smn{k-r}(z)
    \]
is the proper prefix of the image of a letter under $\smn{k}$. 

\item[(ii)] If $c_1>0$ and at least one of $c_2,\ldots,c_k$ is non-zero, the only admissible deletion of letters from $v_1$, leading to a proper prefix of some $\smn{k}(a)$, is to suppress a prefix of $v_1$. Removing a proper suffix of $v_1$ or any ``internal'' factor would violate the constraint that $v_1\cdots v_k$ must be a prefix of the sequence $(j(j+1)\cdots (j+m-1))^\omega$.

\item[(iii)] If $c_1$ is the only non-zero coefficient, the only permissible deletion of letters from $v_1$,  resulting in a proper prefix of some $\smn{k}(a)$, is to suppress 
either
a prefix or a suffix of $v_1$.
\end{itemize}
\end{corollary}

A similar observation applies to $p_{_{U}}$,  which is the proper suffix of some $\smn{k}(j+1)$. 
The only difference lies in the fact that
$\sm(j+1)$ ends with $j$.

Since $|p_{_{U}}|<m^k$, this length can be uniquely expressed using a base-$m$ expansion as:
\[
  |p_{_{U}}|=\sum_{i=0}^{k-1} c_{i+1}\, m^i,\quad c_k,\ldots,c_1\in\{0,\ldots,m-1\}.
\]
By applying a similar greedy procedure to the word $p_{_{U}}$, we obtain the following decomposition:
\[
  p_{_{U}}=\prod_{i=1}^{k} \smn{i-1}\left(v_{i}\right)
\]
where the words $v_i$ 
are defined as follows
\[
v_i = \left(j-k+i-\sum_{r=i}^{k}c_r+1\right) \cdots \left(j-k+i-\sum_{r=i+1}^{k}c_r-1\right)\left(j-k+i-\sum_{r=i+1}^{k}c_r\right).
\]
Notice  that $|v_i|=c_i$.

\begin{example}
The base-$4$ representation of $226$ is $3.4^3+2.4^2+2$.
Here, the suffix of $\sigma_4^4(0)$  with a length of  $226$
is given by
\[
  23\sigma_4^2(23)\sigma_4^3(123)
\]
where $v_4=123$, $v_3=23$, $v_2=\varepsilon$, and $v_1=23$.
\end{example}

\begin{remark}\label{rem:uniquepu}
Similar to the previous case, knowing the value of $j$
and the length
$|p_{_{U}}|$ uniquely determines the decomposition. 
Equivalently, for all integers $n\geqslant 1$ and and any letter~$a$, there exists a unique factor of the form $p_{_{U}}$, of length~$n$,  that starts (respectively, ends) with the letter~$a$. 
\end{remark}

\begin{corollary}
We collect the following facts.

\begin{itemize}
\item[(i)] If $c_k>0$ and at least one of $c_1,\ldots,c_{k-1}$ is non-zero, the only admissible deletion of letters from $v_k$ resulting in a proper suffix of some $\smn{k}(a)$ is to suppress a suffix of $v_k$. Deleting a proper prefix of $v_k$ or some ``internal'' factor would not yield a valid suffix.

\item[(ii)] If $c_k$ is the only non-zero coefficient, the only admissible deletion of letters from $v_k$ leading to a proper suffix of some $\smn{k}(a)$, is to suppress 
either
a prefix or a suffix of $v_k$.
\end{itemize}
\end{corollary}


\section{Counting Classes of a New Equivalence Relation}\label{sec:4}

Since $\sm$ is Parikh-constant, to determine $k$-binomial equivalence of two factors primarily depends on their short prefixes and suffixes, rather than their central part composed of $m^k$-blocks.
Thus, it is meaningful to focus on these prefixes and suffixes for our analysis.
This section presents the core of our counting methods.

For the sake of presentation, let us recall \cref{def:equivk}. 
Let $(p_1,s_1),(p_2,s_2)\in\mathcal{A}_m^{<m^k}\times \mathcal{A}_m^{<m^k}$. We have $(p_1,s_1)\equiv_k (p_2,s_2)$ whenever there exist $x,y,p,q,r,t\in \ams$ with $|x|,|y|<m^{k-1}$ such that 
\begin{eqnarray*}
(p_1,s_1)&=&\left(x \smn{k-1}(p),\smn{k-1}(q) y\right),\\
(p_2,s_2)&=&\left(x \smn{k-1}(r),\smn{k-1}(t) y\right),  
\end{eqnarray*}
and one of the following conditions holds
\begin{itemize}
    \item $pq\sim_1 rt$,
    \item $pq\sim_1 rt\sm(0)$, 
    \item $pq\sm(0)\sim_1 rt$.
\end{itemize}
Notice that if $(p_1,s_1)\equiv_k (p_2,s_2)$, then
\[
|p_1s_1|=|p_2s_2| \quad \text{or} \quad \left|\, |p_1s_1|-|p_2s_2|\, \right|=m^k.
\]

\begin{proposition}\label{pro:imply}
  Let $k\geqslant 2$, and  $U,V\in \Fac(\infw{t}_m)$ of length at least $2m^k$.
  If $(p_{_{U}},s_{_{U}})\equiv_k (p_{_{V}},s_{_{V}})$, then $U\sim_k V$.
\end{proposition}

\begin{proof}
  Suppose first that $|p_{_{U}}s_{_{U}}|=|p_{_{V}}s_{_{V}}|$.  By definition, there exist $x,y,p,q,r,t,u,v\in A^*$ such that:
\begin{eqnarray*}
  U=p_{_{U}}\smn{k}(u)s_{_{U}}&=x \smn{k-1}(p)\smn{k}(u)\smn{k-1}(q) y&=x \smn{k-1}\left(p\sm(u)q\right)y\\
  V=p_{_{V}}\smn{k}(v)s_{_{V}}&=x \smn{k-1}(r)\smn{k}(v)\smn{k-1}(t) y&=x \smn{k-1}\left(r\sm(v)t\right)y    
  \end{eqnarray*}
  and $pq\sim_1 rt$. Since $|U|=|V|$, it follows that $|u|=|v|$ and $\sm(u)\sim_1\sm(v)$. 
  Thus,
  $p\sm(u)q\sim_1 r\sm(v)t$. 
  By \cref{prop:phik},
  we have
  \[
  \smn{k-1}\left(p\sm(u)q\right)\sim_k \smn{k-1}\left(r\sm(v)t\right).
  \]
  
  For the second case, suppose that $|p_{_{U}}s_{_{U}}|=|p_{_{V}}s_{_{V}}|+m^k$.  Using the same notation as above, we have $pq\sim_1 rt\sm(0)$ and $|v|=|u|+1$.
  Therefore
  \[
  r\sm(v)t\sim_1 r\sm(u)\sm(0)t\sim_1 p\sm(u)q
  \]
  and we reach the same conclusion.
\end{proof}

We have an immediate lower bound for the  $k$-binomial complexity of the generalized Thue--Morse word~$\infw{t}_m$.
Using
 \cref{thm:main_equiv},  we will get the value of 
 \(
 \#\left( \left\{(p_{_{U}},s_{_{U}})\mid \, U\in \Fac_n(\infw{t}_m)\right\}/\equiv_k\right)
 \).

\begin{corollary}
  For all $n\geqslant 2m^k$, 
  the
$k$-binomial complexity
$\bc{\infw{t}_m}{k}(n)$
 satisfies the inequality
 \[
 \bc{\infw{t}_m}{k}(n) \geqslant \#\left( \{(p_{_{U}},s_{_{U}})\mid \, u\in \Fac_n(\infw{t}_m)\}/\equiv_k\right).
 \]
\end{corollary}

Let $n\geqslant 2 m^k$. Define $n=\lambda \pmod{m^k}$, where $\lambda=\mu+\nu m^{k-1}$, with $\mu<m^{k-1}$ and $\nu<m$. 
We begin by defining a partition of the set of pairs.

\begin{definition}
Let $\ell\in\{0,\ldots,m^{k-1}\}$. 
Let
\[
  P_\ell^{(n)}:=\left\{(p_{_{U}},s_{_{U}}) \mid \, U\in \Fac_n(\infw{t}_m), \, |p_{_{U}}|\equiv \ell\pmod m^{k-1}\right\}.
\]
and similarly, 
\[
  S_{\ell'}^{(n)}:=\left\{(p_{_{U}},s_{_{U}}) \mid \, U\in \Fac_n(\infw{t}_m), \, |s_{_{U}}|\equiv \ell'\pmod m^{k-1}\right\}.
\]
\end{definition}

Note that
\[
  \bigcup_{\ell=0}^{m-1}P_\ell^{(n)}=\left\{(p_{_{U}},s_{_{U}}) \mid \, U\in \Fac_n(\infw{t}_m)\right\}=\bigcup_{\ell'=0}^{m-1}S_{\ell'}^{(n)}.
\]

Let $(p,s)\in P_\ell^{(n)}$. By Euclidean division, since $|p|,|s|<m^k$, we have
\[
  |p|=\ell +\alpha\, m^{k-1} \text{ and }\quad |s|=\ell' +\alpha'\, m^{k-1},
\]
for some $\alpha,\alpha'<m$ and $\ell'<m^{k-1}$.
We show that $n,\ell,\alpha$ completely determine $\ell'$ and $\alpha'$. In particular, for each $\ell$, there exists a unique $\ell'$ such that $P_\ell^{(n)}=S_{\ell'}^{(n)}$.

Since 
\[
\ell +\ell'+(\alpha+\alpha')\, m^{k-1}=|ps|\equiv n \pmod{m ^k},
\]
 we have   
 \[
 \ell +\ell'\equiv \mu\bmod{m^{k-1}}.
 \]
 Thus, either
 \begin{enumerate}
     \item $\ell\leqslant\mu$ and $\ell'=\mu-\ell$ or,
     \item $\ell>\mu$ and $\ell'=m^{k-1}+\mu-\ell$.
 \end{enumerate}

If $\ell\leqslant\mu$, then $\ell +\ell'=\mu$ and:
\[
|ps|=\mu+(\alpha+\alpha')\, m^{k-1}\equiv \mu+\nu m^{k-1}\pmod{m^k}.
\]
 If $\alpha\leqslant\nu$, then $\alpha'=\nu-\alpha$. Otherwise $\alpha>\nu$, then $\alpha'=\nu+m-\alpha$.

In the second case ($\ell>\mu$),
we have 
\[
\ell +\ell'=\mu+m^{k-1},
\quad
\text{and}
\quad
|ps|=\mu+(\alpha+\alpha'+1)\, m^{k-1}
\]
 If $\alpha\leqslant\nu-1$, then $\alpha'=\nu-\alpha-1$. Otherwise, $\alpha>\nu-1$, then $\alpha'=\nu+m-\alpha-1$. These observations are recorded in \cref{tab:summary}.

\begin{table}[h!tb]
  \centering
  \[
    \begin{array}{ll|ll}
      &\ell\leqslant\mu& &\ell>\mu\\
      \hline
      \alpha\leqslant\nu: & \ell'=\mu-\ell&  \alpha\leqslant\nu-1: & \ell'=m^{k-1}+\mu-\ell\\
      & \alpha'=\nu-\alpha & & \alpha'=\nu-\alpha-1\\
      i.e.,  & \alpha+\alpha'=\nu & i.e.,  & \alpha+\alpha'=\nu-1 \\
      \hline
      \alpha>\nu: & \ell'=\mu-\ell &  \alpha>\nu-1: & \ell'=m^{k-1}+\mu-\ell\\
      & \alpha'=\nu+m-\alpha & & \alpha'=\nu+m-\alpha-1 \\
      i.e.,  & \alpha+\alpha'=\nu+m & i.e.,  & \alpha+\alpha'=\nu+m-1 \\
      \hline
    \end{array}
  \]
  \caption{Summary for $(\ell',\alpha')$ for fixed $\mu,\nu$ and $\alpha$ varying.}
  \label{tab:summary}
\end{table}

\begin{example}
Let $m=3$ and $k=2$. If $n\equiv 4\pmod{9}$, then $\mu=1$ and $\nu=1$.
The set $P^{(n)}_0$ contains pairs $(p,s)$ such that $|p|=0,3,6$,
which is $0+\alpha\, 3$ for $\alpha=0,1,2$.
Since $\ell=0\leqslant 1=\mu$, we have $\ell'=1$. For $\alpha=0$
or
$1$,
which is less than or equal to
$\nu$, the corresponding values of $\alpha'$ are $1$ and
$0$,
respectively.
For $\alpha=2$
which is greater than
$\nu$, $\alpha'$
 is 
$\nu+3-\alpha=2$.
Thus, the lengths of $s$ corresponding to $|p|=0,3,6$ are $4,1,7$,
respectively.
Therefore, note that
 $P^{(n)}_0=S^{(n)}_1$.

The set $P^{(n)}_1$ contains pairs $(p,s)$ such that $|p|=1,4,7$, 
which is 
$1+\alpha\, 3$ for $\alpha=0,1,2$.
Since $\ell=1\leqslant 1=\mu$, 
we have
$\ell'=0$.
For $\alpha=0$
or
$1$,
which is less than or equal to
$\nu$, the corresponding values of $\alpha'$ are $1$ and
$0$,
respectively.
For $\alpha=2$,
which is greater than
$\nu$,
$\alpha'$
is 
$\nu+3-\alpha=2$. 
Thus, the lengths of 
$s$
 corresponding to $|p| =1,4,7$ are $3,0,6$,
 respectively.
 Therefore, note that $P^{(n)}_1=S^{(n)}_0$.

The set $P^{(n)}_2$ contains pairs $(p,s)$ such that $|p|=2,5,8$,
which is 
$2+\alpha\, 3$ for $\alpha=0,1,2$.
Since $\ell=2$
 is greater than
 $\mu=1$, we have
$\ell'=3+\mu-2=2$.
For $\alpha=0$, the corresponding value of $\alpha'$ is $\nu-1=0$.
For 
$\alpha=1$ and
$\alpha=2$,
both greater than
$\nu-1$, the corresponding values of 
 $\alpha'$ are $2$
 and $1$
 respectively.
  Thus, the lengths of 
$s$
 corresponding to $|p|=2,5,8$ are $2,8,5$,
  respectively.
 Finally, note that  $P^{(n)}_2=S^{(n)}_2$.
\end{example}

Note that if $\mu=0$, then $P^{(n)}_0=S^{(n)}_0$. If $\mu\neq 0$, then for $\ell=0$,
we have
$\ell'=\mu\neq 0$. In that case $P^{(n)}_0=S^{(n)}_\mu\neq S^{(n)}_0=P^{(n)}_\mu$. This observation gives an initial hint as to why the statement of \cref{thm:main_equiv} contains two cases.

Recall that the abelian complexity of $\infw{t}_m$ is well known (see \cref{thm:abelian_complexity}).

\begin{theorem}\label{thm:main_equiv}
  Let $n\geqslant 2 m^k$. If $\lambda=n \bmod{m^k}$ and $\lambda=\nu m^{k-1}+\mu$, where $\nu<m$ and $\mu<m^{k-1}$, then the value of 
  \[\#\left\{(p_{_{U}},s_{_{U}}) \mid \, U\in \Fac_n(\infw{t}_m)\right\}/\equiv_k\]
  is given by 
  \[
    (m^{k-1}-1)(m^3-m^2+ m)+\left\{\begin{array}{ll}
                                     \bc{\infw{t}_m}{k}(m+\nu), & \text{ if }\mu =0;\\
                                     m, & \text{ otherwise.}\\
                                   \end{array}\right.
                               \]
\end{theorem}

\begin{remark}
Note that for $k=2$, which was the case studied  in~\cite{ChenWen2024}, this expression matches the $2$-binomial complexity of $\infw{t}_m$.
Thus, we obtain the converse of \cref{pro:imply}: Let $U$ and $V$ be two factors of $\infw{t}_m$ of length at least $\ 2m^2$.
Then,
 $U\sim_2 V$ if and only if $(p_{_{U}},s_{_{U}})\equiv_2 (p_{_{V}},s_{_{V}})$. 
\end{remark}

\begin{proof}
$\bullet$ {\bf Case 1.a)} Let us consider $\mu\neq 0$ and $\ell\neq 0$.
Assume that $\ell\leqslant\mu$. Referring to the first column of \cref{tab:summary}, the elements of $P^{(n)}_\ell$ have the form 
given in \cref{tab:pnl}, 
where $x^j$ and $y^j$ are words and $r_i^j$, $t_i^j$ are letters.
\begin{table}[h!tbp]
  \[
    \begin{array}{c|r|l|c}
 \alpha&     p_{_{U}} & s_{_{U}} & \alpha'\\ \hline
    0&x^0 \smn{k-1} (\varepsilon) & \smn{k-1} (r^0_1\cdots r^0_{\nu-1} r^0_\nu) y^0 &\nu\\
    1&x^1 \smn{k-1} (t_1^1) & \smn{k-1} (r^1_1\cdots r^1_{\nu-1}) y^1 & \nu-1 \\
    \vdots&\vdots & \vdots & \vdots \\
    \nu-1&x^{\nu-1} \smn{k-1} (t^{\nu-1}_{\nu-1}\cdots t^{\nu-1}_1) & \smn{k-1} (r_1^{\nu-1}) y^{\nu-1} & 1 \\
    \nu&x^{\nu} \smn{k-1} (t^{\nu}_\nu\ t_{\nu-1}^{\nu}\cdots \ t^{\nu}_1\ ) & \smn{k-1} (\varepsilon) y^{\nu} & 0 \\
    \nu+1&x^{\nu+1} \smn{k-1} (t^{\nu+1}_{\nu+1} t_{\nu}^{\nu+1}\cdots t^{\nu+1}_1) & \smn{k-1} (r^{\nu+1}_1\cdots r^{\nu+1}_{\nu+1}\cdots r^{\nu+1}_{m-1}) y^{\nu+1} & m-1 \\
    \vdots & \vdots & \vdots & \vdots \\
    m-1&x^{m-1} \smn{k-1} (t^{m-1}_{m-1}\cdots t_{\nu}^{m-1}\cdots t^{m-1}_1) & \smn{k-1} (r^{m-1}_1\cdots r^{m-1}_{\nu+1}) y^{m-1} & \nu+1 \\
    \end{array}
  \]
  \caption{Words in $P^{(n)}_\ell$.}\label{tab:pnl}
\end{table}

Since we are dealing with proper suffixes or prefixes of the image of
a letter under $\smn{k}$, we also have
\[
  \forall j<m: \quad t^j_{i+1}=t^j_i-1 \text{ and } r^j_{i+1}=r^j_i+1.
\]
Since $\ell\neq 0$ (respectively, $\ell\neq \mu$), the words $x^j$ (respectively, $y^j$) are non-empty of length~$\ell$ (respectively, $\mu-\ell$).

Thanks to \cref{rem:uniquesu,rem:uniquepu}, there are at most $m^2$ words on each row of \cref{tab:pnl}: a prefix (respectively, \ suffix) of any given length is determined by its last (respectively, \ first) letter. Thanks to \cref{lem:boundaryseq}, there are exactly $m^2$ words on each row.

We now consider the quotient by $\equiv_k$. Since the words $r^0_1\cdots r^0_{\nu-1} r^0_\nu$ have length less than $m$ and are made of consecutive letters, if two such words have distinct first letter, then there are not abelian equivalent. Hence the $m^2$ words on this row are pairwise non-equivalent.

The same argument applies on the second row. Nevertheless, if $t_1^1=r_1^1-1$, then
\[
  (x^1 \smn{k-1} (t_1^1), \smn{k-1} (r^1_1\cdots r^1_{\nu-1}) y^1)\equiv_k
  (x^1 \smn{k-1} (\varepsilon), \smn{k-1} (t_1^1 r^1_1\cdots r^1_{\nu-1} ) y^1).
\]
If $t_1^1\neq r_1^1-1$, we cannot make such a move an keep equivalent pairs (we know from \eqref{eq:decompsu} that we must have consecutive letters in $t_1^1 r^1_1\cdots r^1_{\nu-1}$).  
So we find $m(m-1)$ new classes.

We have a similar counting in the first $\nu+1$ rows (we proceed downwards, comparing elements on a row with elements on previous rows). Take a word of the form
\[
  x^{j} \smn{k-1} (t^{j}_j \cdots\  t_s^{j}\cdots \ t^{j}_1\ )
\]
on the row $j\leqslant \nu$. Thanks to \cref{cor:delpu} (ii), we can only delete a suffix of 
$t^{j}_s \cdots \ t^{j}_1$ to keep a valid suffix of some $\smn{k}(a)$. If $t_1^j=r_1^j-1$, since the suffix is made of consecutive letters
\[
\begin{aligned}
    (x^{j} \smn{k-1} (t^{j}_j \cdots t^{j}_s \cdots t^{j}_1), \smn{k-1} (r^j_1 \cdots r^j_{\nu-r}) y^j) \\
    \equiv_k  (x^{j} \smn{k-1} (t^{j}_j \cdots t^{j}_{s+1}), \smn{k-1} (t^{j}_{s} \cdots t^{j}_1 r^j_1 \cdots r^j_{\nu-r}) y^j)
\end{aligned}
\]
for any $1\leqslant s\leqslant j$. We again find $m(m-1)$ new classes.

For the second part of the Table, take row $j\geqslant \nu+1$. The reasoning is again the same but this time, when $t_1^j=r_1^j-1$, take $s\geqslant j-\nu$, then $t^{j}_{s} \cdots t^{j}_1\  r^j_1\cdots r^j_{m+\nu-j}$ has length $m+\nu-j+s\geqslant m$. So it has a prefix which is a cyclic permutation of $0,1,\ldots,m-1$. Hence, so we find an equivalent pair
\[
  (x^{j} \smn{k-1} (t^{j}_j \cdots t^{j}_{s+1} ),  \smn{k-1} (r^j_{m-s}\cdots r^j_{m+\nu-j})  y^j)
\]
in the first part of the table. 

The case $\ell>\mu$ is treated similarly. As a conclusion, we have $m^2$ classes for the first row and $m(m-1)$ classes for each of the $m-1$ other rows for a total of $m^2+m(m-1)^2$ classes. 

We have considered so far $m^{k-1}-2$ sets $P^{(n)}_\ell$ each containing $m^2+m(m-1)^2$ classes.

$\bullet$ {\bf Case 1.b)} Let us consider $\mu\neq 0$ and focus on $P^{(n)}_0$ (similar discussion for $P^{(n)}_\mu$). The only difference in \cref{tab:pnl} is that there is no word $x^j$ (it is empty because $\ell=0$). The word $y^j$ remains non-empty (because $\mu\neq 0$). In the first row, we have $(\varepsilon,\smn{k-1} (r^0_1\cdots r^0_{\nu-1} r^0_\nu) y^0)$ so the number of classes is given by the number~$m$ of choices for $r^0_1$. Now come the extra discussion for $1\leqslant j\leqslant m-1$ due to the absence of~$x^j$.
In $\smn{k-1} (t^{j}_j\cdots \ t^{j}_1\ )$ 
to get equivalent pairs, we can as above move a suffix $t_s^{j}\cdots \ t^{j}_1$ to the second component whenever $t^j_1=r^j_1-1$ but also move a prefix $t^{j}_j \cdots t^j_{j-s+1}$ whenever $t^j_{j-s+1}=r^j_1-1$. Consequently, the word $t^{j}_j\cdots \ t^{j}_1$ should not contain $r^j_1-1$ which is equivalent to $t^{j}_j\in\{r^j_1,r^j_1+1,\ldots,r^j_1+m-j-1\}$ using the fact that the word is made of consecutive letters. So we have $m(m-j)$ choices. So the total is given by 
\[
m+\sum_{j=1}^{m-1} m(m-j)=\frac{1}{2} \left(m^3-m^2+2 m\right),
\]
 and this contribution is doubled to take the symmetric case of $P^{(n)}_\mu$.

As a conclusion, when $\mu\neq 0$, i.e.,  if $n\neq 0\pmod{m^k}$, then
\begin{eqnarray*}
  \#\left\{\left(p_{_{U}},s_{_{U}}\right) \mid \, U\in \Fac_n(\infw{t}_m)\right\}/\equiv_k&=&
                                                      (m^{k-1}-2)(m^2+m(m-1)^2)+m^3-m^2+2 m  \\
                                                  &=& (m^{k-1}-1)(m^3-m^2+ m)+ m.
\end{eqnarray*}

$\bullet$ Case 2) Let $\mu=0$. If $\ell\neq 0$, then from \cref{tab:summary} we get $\ell'=m^{k-1}-\ell\neq 0$. Then, we have the same discussion as in our first case. The $m^{k-1}-1$ sets $P_\ell^{(n)}$ for $\ell=1,\ldots,m-1$ contain $m^2+m(m-1)^2$ classes (we get the same main term in the expression). 

If $\ell=0$, then $\ell'=0$. Here, the particularity of the single set $P_0^{(n)}$ is that in \cref{tab:pnl} the words $x^j$ and $y^j$ are both empty. So we only consider pairs $(p_{_{U}},s_{_{U}})$ of the form $(\smn{k-1}(p'),\smn{k-1}(s'))$ with $|p'|,|s'|<m$ and $|p's'|=\nu$ or $m+\nu$. 
We will show that
\[
  \#(P_0^{(\nu\, m^{k-1})}/\equiv_k)= \#(\Fac_{2m+\nu}(\infw{t}_m)/\sim_1).
\]
Thanks to \cref{pro:circular}, any factor $x$ of length $2m+\nu$ has a unique factorization of the form
\[
  x=p_x\sigma(w)s_x \text{ with } |p_x|,|s_x|<m \text{ and } |w|\in\{1,2\}.
\]
 Thanks to \cref{lem:boundaryseq}, a pair $(p_{_{U}},s_{_{U}})=(\smn{k-1}(p'),\smn{k-1}(s'))$ belongs to $P_0^{(\nu\, m^{k-1})}$ if and only if $(p',s')$ is of the form $(p_x,s_x)$ for some $x$ in $\Fac_{2m+\nu}(\infw{t}_m)$.

 Let $x,y\in \Fac_{2m+\nu}(\infw{t}_m)$, and their corresponding factorizations $x=p_x\sigma(w)s_x$ and $y=p_y\sigma(w')s_y$. If $x\sim_1 y$ and $|p_xs_x|=|p_ys_y|$, then $|w|=|w'|$ and thus $\sm(w)\sim_1\sm(w')$. So $p_xs_x\sim_1 p_ys_y$ and we get $$(\smn{k-1}(p_x),\smn{k-1}(s_x))\equiv_k (\smn{k-1}(p_y),\smn{k-1}(s_y)).$$ If $x\sim_1 y$ but $|p_xs_x|\neq |p_ys_y|$, then the difference of their length is $m$. We may assume that $|p_xs_x|=|p_ys_y|+m$, so $|w|=1$ and $|w'|=2$. Since $\sm(a)$ is a circular permutation of $01\cdots (m-1)$, we deduce that $p_xs_x\sim_1 p_ys_y\sigma(0)$ and the same conclusion follows. The converse also holds, if $x,y\in \Fac_{2m+\nu}(\infw{t}_m)$ and $(\smn{k-1}(p_x),\smn{k-1}(s_x))\equiv_k (\smn{k-1}(p_y),\smn{k-1}(s_y))$, then considering both situations, one concludes that $x\sim_1 y$. It is known that for words of length at least $m$ the abelian complexity function is periodic of period $m$, see~\cite{ChenWen2019}. Hence,
 \[
 \#\left(\Fac_{2m+\nu}(\infw{t}_m)/\sim_1\right)=\#\left(\Fac_{m+\nu}(\infw{t}_m)/\sim_1\right).
 \]
\end{proof}


\section{Characterizing Binomial Equivalence in \texorpdfstring{$\infw{t}_m$}{tm}}\label{sec:characterizing}

In this section, we focus on characterizing $k$-binomial equivalence among factors of $\infw{t}_m$ through their $\smn{k-1}$-factorizations.
We recall the main result:

\conclusionfinalgeneralization*

We observe that this proposition extends \cite[Thm.~2]{ChenWen2024} by removing an additional assumption $|u|,|v| \geqslant 3$ and extending it to all $k\geqslant 2$.

To prove the main characterization, we shall present the following restricted version.

\begin{lemma}\label{lem:k-1-factorisation-first-last}
Let $k\geqslant 2$ and $U$ and $V$ be factors of~$\infw{t}_m$ for some $m\geqslant 2$. Assume further that $U$ and $V$ begin and end with distinct letters. Then $U \sim_k V$ if and only if there exist $\smn{k-1}$-factorizations
$U = \smn{k-1}(u)$ and $V = \smn{k-1}(v)$ such that $u \sim_1 v$.
\end{lemma}

Before diving into the proof of \cref{lem:k-1-factorisation-first-last}, let us observe how
\cref{prop:conclusion-final-generalization} follows from it. First, we obtain \cref{thm:main1short} as an immediate corollary of \cref{lem:k-1-factorisation-first-last}.

\shortlengths*

\begin{proof}
The shortest pair of distinct $k$-binomially equivalent factors necessarily begin and end with different letters due to $k$-binomial equivalence being cancellative (cf.~\cref{lem:cancel}). \cref{lem:k-1-factorisation-first-last}
thus shows that the pair of factors can be written in the form $\smn{k-1}(u)$ and $\smn{k-1}(v)$ with $u \sim_1 v$.
Therefore, $|u| = |v| \geqslant 2$ (since they must begin and end with different letters), giving the lower bound. The pair $\smn{k-1}(01)$ and
$\smn{k-1}(10)$, for example, gives
the desired pair of length $2m^{k-1}$.
\end{proof}

We can now prove \cref{prop:conclusion-final-generalization}

\begin{proof}[Proof of \cref{prop:conclusion-final-generalization}]
Let $k \geqslant 2$ be arbitrary.
If $U$ and $V$ have the $\smn{k-1}$-factorizations $U = p_{_{U}} \smn{k-1}(u) s_{_{U}}$ and $V = p_{_{V}} \smn{k-1}(v)s_{_{V}}$, where $p_{_{U}} = p_{_{V}}$, $s_{_{U}} = s_{_{V}}$, and $u\sim_1 v$, then $U \sim_k V$ follows by \cref{pro:both_dir} and the fact that $\sim_k$ is a congruence.

For the converse, assume  $U \sim_k V$. There is nothing to prove if $U = V$, as all factors have a $\smn{k-1}$-factorization by \cref{rem:all-factors-smk-factorization}.
So assume  $U \neq V$.
Write $U = pU's$ and $V = pV's$, where $U'$ and $V'$ begin and end with distinct letters.
By cancellativity (\cref{lem:cancel}), we have  $U' \sim_k V'$. By \cref{lem:k-1-factorisation-first-last},
there exist $\smn{k-1}$-factorizations $U' = \smn{k-1}(u')$ and $V' = \smn{k-1}(v')$, where
$u' \sim_1 v'$.
Note that \cref{thm:main1short} implies  $|U'|, |V'| \geqslant 2m^{k-1}$. By \cref{cor:unique-factorization-bound}, these
$\smn{k-1}$-factorizations are unique.
It follows that $U$ and $V$ have the desired (unique) $\smn{k-1}$-factorizations
$U = p\smn{k-1}(u')s$ and $V = p\smn{k-1}(v')s$, where $u' \sim_1 v'$.
\end{proof}

The proof of \cref{lem:k-1-factorisation-first-last} proceeds by induction on $k$. 
We divide the remainder of the section into two subsections: the base case $k=2$,  handled in the first subsection, and the induction step, covered in the second.
We observe that the base case $k=2$ is almost handled by \cite[Thm.~2]{ChenWen2024}, except
that the additional assumption $|u|$, $|v| \geqslant 3$ appearing there needs to be removed.
Although the cases where $|u|$, $|v| \leqslant 3$ could be treated separately, we provide a complete, independent, but similar, proof of the
case $k=2$, as it reveals our strategy for tackling the induction step.

\subsection{The base case}

We shall state the induction base case as a separate lemma:

\begin{lemma}\label{lem:generalised-k=2-distinct-extremal-letters}
Let $U$ and $V$ be factors of $\infw{t}_m$ that begin and end with distinct letters. Then $U \sim_2 V$ if and only if there exist
$\sm$-factorizations $U = \sm(u)$ and $V = \sm(v)$, such that $u \sim_1 v$.
\end{lemma}

\begin{proof}
If such $\sm$-factorizations exist for $U$ and $V$, then the two words are $2$-binomially equivalent by \cref{pro:both_dir}.

Assume  that $U$ and $V$ are $2$-binomially equivalent factors, beginning and ending with distinct letters.
Let $U$ and $V$ have the $\sm$-factorizations $p_{_{U}} \sm(u) s_{_{U}}$ and $p_{_{V}} \sm(v) s_{_V}$, respectively (such factorizations exist by \cref{rem:all-factors-smk-factorization}). Notice that $\left||u|-|v|\right| \leqslant 1$ due to length constraints. W.l.o.g., we assume that $|u| \leqslant |v|$.

\textbf{First, assume that $|u| = |v|$.}
If both $s_{_{U}}$ and $s_{_{V}}$ are empty, it follows that $p_{_{U}}\sm(u) \sim_1 p_{_{V}} \sm(v)$. Since $\sm(u) \sim_1 \sm(v)$,
we conclude that $p_{_{U}} \sim_1 p_{_{V}}$.
This further implies $p_{_{U}} = \varepsilon = p_{_{V}}$, as $U$ and $V$ start with distinct letters, and $p_{_{U}}$
and $p_{_{V}}$ are proper suffixes of images of letters. 
By
\cref{pro:both_dir}, it follows that $u \sim_1 v$, thereby establishing the claimed factorizations.

Thus, we proceed under the assumption that at least one of the words $s_{_{U}}$ and $s_{_{V}}$ is non-empty, intending to get a contradiction.
W.l.o.g.,
we assume that $s_{_{U}}$ is non-empty.
Now,
let  $\alpha-1$ denote the last letter of $s_{_{U}}$.
By assumption, we have 
 $\binom{U}{\alpha (\alpha-1)} = \binom{V}{\alpha (\alpha-1)}$; 
 applying 
 \cref{lem:binomial2} twice, we obtain
\[
\begin{aligned}
    \binom{p_{_{_U}} s_{_{_U}}}{\alpha (\alpha-1)} +& \binom{\sm(u)}{\alpha (\alpha-1)} + |u|\left(|p_{_{_U}}|_{\alpha} + |s_{_{_U}}|_{\alpha-1}\right) \\
    = &\binom{p_{_{_V}} s_{_{_V}}}{\alpha (\alpha-1)} + \binom{\sm(v)}{\alpha (\alpha-1)} + |v|\left(|p_{_{_V}}|_{\alpha} + |s_{_{_V}}|_{\alpha-1}\right).
\end{aligned}
\]

Observe that $|p_w|_{\alpha} = |p_ws_w|_{\alpha} - |s_w|_{\alpha}$, where $w$
is either
$u$ or $v$. Similarly, we have 
$|s_{_{U}}|_{\alpha-1} = 1$ and $|s_{_{U}}|_{\alpha} = 0$. 
Substituting these values into the previous equation yields
\begin{multline*}
\binom{p_{_{U}} s_{_{U}}}{\alpha (\alpha-1)} + \binom{\sm(u)}{\alpha (\alpha-1)} + |u|\left(|p_{_{U}}s_{_{U}}|_{\alpha} + 1\right)\\
= \binom{p_{_{V}} s_{_{V}}}{\alpha (\alpha-1)} + \binom{\sm(v)}{\alpha (\alpha-1)} + |v|\left(\left|p_{_{V}} s_{_{V}}\right|_{\alpha} - |s_{_{V}}|_{\alpha} + |s_{_{V}}|_{\alpha-1}\right).
\end{multline*}
The terms $|u||p_{_{U}} s_{_{U}}|_{\alpha}$ and $|v||p_{_{V}} s_{_{V}}|_{\alpha}$ cancel  because $|u| = |v|$, and the equivalence $U\sim_2 V$ implies  $p_{_{U}} s_{_{U}} \sim_1 p_{_{V}} s_{_{V}}$. 
By \cref{lem:subwords}, $\alpha(\alpha-1)$ appears exclusively in $\sm(\alpha)$, implying that $\binom{\sm(u)}{\alpha (\alpha-1)}=|u|_\alpha$. 
Rearranging this equation yields the following equality
\begin{equation}\label{eq:equal-length-2bin}
|u|_{\alpha} + |v|\left(|s_{_{V}}|_{\alpha}-|s_{_{V}}|_{\alpha - 1}\right)
= |v|_{\alpha} - |u| +  \binom{p_{_{V}} s_{_{V}}}{\alpha(\alpha-1)} - \binom{p_{_{U}} s_{_{U}}}{\alpha(\alpha-1)}.
\end{equation}

\begin{claim}
\begin{enumerate}
    \item[1)] The left-hand side of \eqref{eq:equal-length-2bin} is non-negative. Furthermore, it is equal to $0$ if and only if
either $u = v = \varepsilon$, or $|u|_{\alpha} = 0$ and $|s_{_{V}}|_{\alpha} = |s_{_{V}}|_{\alpha-1}$.

\item[2)] 
 The right-hand side of \eqref{eq:equal-length-2bin} is non-positive. Moreover, it equals $0$ if and only if $|v|_{\alpha} = |v|$
and $\alpha$ does not appear in $p_{_{U}}s_{_{U}}$.
\end{enumerate}
\end{claim}

\begin{claimproof}
Consider the first claim. Note that the left-hand side can only be negative if $|s_{_{V}}|_{\alpha-1} > |s_{_{V}}|_{\alpha}$.
However, this situation cannot occur: if $\alpha-1$ appears in $s_{_{V}}$, then as $s_{_{V}}$ does not end with $\alpha-1$;
instead, $\alpha-1$ must be followed by $\alpha$.
Consequently, the coefficient of $|v|$ is non-negative, showing the non-negativity of the left-hand side.
To attain a value of $0$, we must have that either $u = \varepsilon$, or
$|u|_{\alpha} = 0$ and $|s_{_{V}}|_{\alpha} = |s_{_{V}}|_{\alpha-1}$.

Let us consider the second claim.
If $\alpha$ does not appear in $p_{_U}s_{_{U}}$, then 
\[
\binom{p_{_V}s_{_{V}}}{\alpha(\alpha-1)} = 0 = \binom{p_{_U}s_{_{U}}}{\alpha(\alpha-1)}.
\]
Consequently, the right-hand side is equal to $|v|_{\alpha} - |v|$, which is clearly non-positive, and it is equal to $0$ if and only if $|v|_{\alpha} = |v|$.

If $\alpha$ appears in $p_{_U}s_{_{U}}$, it must occur in $p_{_U}$ and does so precisely once.
Since $\alpha-1$ does not appear in $p_{_U}$ after $\alpha$, we have
$\binom{p_{_U} s_{_U}}{\alpha(\alpha-1)} = 1$.
Next,
consider the occurrences of $\alpha-1$ and $\alpha$ in $p_{_V} s_{_V}$. Note that $\alpha$ cannot precede $\alpha-1$ in $p_{_V}$ or $s_{_V}$.
If $\alpha-1$ appears in $s_{_V}$  then, because $s_{_V}$ does not end with $\alpha-1$, it must be followed by $\alpha$ in $s_{_V}$.
Thus, we conclude that
$\binom{p_{_V} s_{_V}}{\alpha(\alpha-1)} = 0$. Hence, the right-hand side equals $|v|_{\alpha} - |v| - 1$, which is strictly negative.
The desired conclusion thereby follows.
\end{claimproof}

The above claim shows that \eqref{eq:equal-length-2bin} can only be satisfied when both the left-hand side and the right-hand side are equal to zero.
In other words,  $\alpha$ must not appear in $p_{_U}s_{_U}$ (and consequently not in $p_{_V}s_{_V}$) and either:
(a) $u=v=\varepsilon$; or (b) $|u|_{\alpha} = 0$, $|s_{_V}|_{\alpha-1} = |s_{_V}|_{\alpha} = 0$, and $|v|_{\alpha} = |v|$.
Note that $p_{_V}$ must contain $\alpha-1$, which corresponds to the occurrence of $\alpha-1$ as the last letter of $s_{_U}$, and thus $p_{_V}$ must end with $\alpha-1$; otherwise, it would contain $\alpha$ immediately following $\alpha-1$.
This situation is illustrated in
\cref{fig:facsm1}.
Since $|u|_\alpha=0$, the image of each letter of $u$ under $\sm$ contains the factor $\alpha(\alpha-1)$.
Since $v=\alpha^{|v|}$, the image of each letter of $v$ under $\sm$ begins with $\alpha$ and ends with $\alpha-1$.
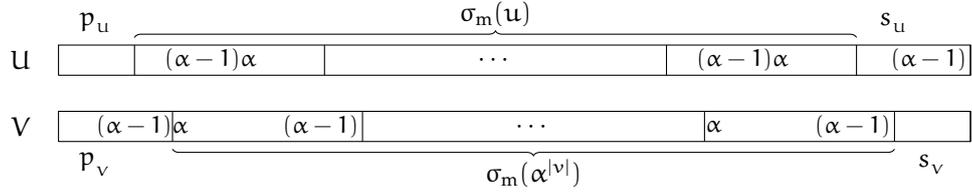
\begin{figure}[h!t]
\centering
\begin{tikzpicture}
\draw[yscale=.4] (0,0) -- (12,0) -- (12,1) -- (0,1) -- (0,0);
\draw[yscale=.4] (1,1) -- (1,0);
\draw[yscale=.4] (3.5,1) -- (3.5,0);
\draw[yscale=.4] (8,1) -- (8,0);
\draw[yscale=.4] (10.5,1) -- (10.5,0);
\draw[yscale=.4] (12,1) -- (12,0);

\node at (-.5,.2) {$U$};
\node at (.5,.7) {$p_{_U}$};
\node at (2,.2) {\small{$(\alpha-1)\alpha$}};
\draw[decoration={brace},decorate] (1,.5) -- (10.5,.5) node[above, midway] {$\sm(u)$};
\node at (5.75,.2) {$\cdots$};
\node at (9,.2) {\small{$(\alpha-1)\alpha$}};
\node at (11,.7) {$s_{_{U}}$};

\node at (11.45,.2) {\small{$(\alpha-1)$}};

\begin{scope}[yshift = -25]
\draw[yscale=.4] (0,0) -- (12,0) -- (12,1) -- (0,1) -- (0,0);
\draw[yscale=.4] (1.5,1) -- (1.5,0);
\draw[yscale=.4] (4,1) -- (4,0);
\draw[yscale=.4] (8.5,1) -- (8.5,0);
\draw[yscale=.4] (11,1) -- (11,0);
\draw[yscale=.4] (12,1) -- (12,0);

\node at (-.5,.2) {$V$};
\node at (.5,-.3) {$p_{_{_V}}$};
\node at (1.10,.2) {\small{$(\alpha-1)\alpha$}};
\node at (3.45,.2) {\small{$(\alpha-1)$}};
\draw[decoration={brace,mirror},decorate] (1.5,-.1) -- (11,-.1) node[below, midway] {$\sm(\alpha^{|v|})$};
\node at (6.25,.2) {$\cdots$};
\node at (8.62,.2) {\small{$\alpha$}};
\node at (10.45,.2) {\small{$(\alpha-1)$}};
\node at (11.5,-.3) {$s_{_{V}}$};
\end{scope}
\end{tikzpicture}
\caption{Illustrating the situation $|u|=|v|$ and $s_{_{_U}}$ or $s_{_{V}}$ non-empty.}\label{fig:facsm1}
\end{figure}

Consider the sum
\[
\sum_{x \in \am} \binom{U}{(\alpha-1) x} + \binom{U}{x \alpha} - \binom{V}{(\alpha-1) x} - \binom{V}{x \alpha},
\]
which equals zero, based on the assumption that $U \sim_2 V$. 
Observe that 
\(\sum_{x \in \am}\binom{U}{(\alpha-1)x}\)
counts, for each occurrence of $(\alpha-1)$ in $U$, the number of letters to its right.
Similarly,
\(
\sum_{x \in \am}\binom{U}{x\alpha}
\)
 counts,
for each occurrence of $\alpha$ in $U$, the number of letters to its left.
With this interpretation, the “positive” part of the sum is equal to  $|u|\cdot |U|$.
Each of the $|u|$ occurrences of the factor $(\alpha-1)\alpha$  contributes $|U|$ to the positive count, while the last occurrence of $\alpha-1$ contributes zero.
Similarly, the negative part of the sum is equal to
$-|v|\cdot |V| - |s_{_{V}}|$.
Each of the $|v|$ occurrences of the factor $(\alpha-1)\alpha$ contributes $-|v|\cdot |V|$ to the negative count, while the last occurrence of $\alpha-1$ contributes $-|s_{_{V}}|$. 
Since the sum must equal zero, we conclude that $s_{_{V}} = \varepsilon$.
However, now $V$ ends with $\alpha-1$: if $v \neq \varepsilon$, then $\sm(v)$ ends with $\alpha-1$, and if $v = \varepsilon$, then $p_{_{V}} = V$ ends with $\alpha - 1$.
This contradicts the assumption that $U$ and $V$ end with distinct letters, resulting in a contradiction when $|u| = |v|$ and at least one of the words $s_{_{U}}$ and $s_{_{V}}$ is non-empty.

\textbf{Second, assume that $|u|+1 = |v|$.} 
We will show that this case is impossible as it leads to a contradiction.
In this situation, $s_{_{U}}$ must be non-empty (as must $p_{_{U}}$), since $|p_{_{U}}s_{_{U}}| = |p_{_{V}}s_{_{V}}| + m$ and $|p_{_{U}}|$, $|s_{_{U}}| < m$.
Let $\alpha-1$ be the last letter of $s_{_{U}}$.
Let $\beta$ be the first letter of $v = \beta v'$, where $|v'| = |u|$, and let $p_{_{V}}' = p_{_{_V}}\sm(\beta)$. Note that $p_{_{U}}s_{_{U}}\sim_1 p_{_{v}}'s_{_{_V}}$. As before, we have
\[
\binom{U}{\alpha (\alpha-1)} = \binom{V}{\alpha (\alpha-1)}.
\]
 Using similar techniques as in the previous case, the equality can be expressed equivalently as
\[
|u|_{\alpha} + |v'|\left(|s_{_{V}}|_{\alpha} - |s_{_{V}}|_{\alpha-1}\right)
=  |v'|_{\alpha} - |u| + \binom{p_{_{V}}' s_{_{V}}}{\alpha (\alpha-1)} - \binom{p_{_{U}} s_{_{U}}}{\alpha (\alpha-1)}.
\]
We may proceed similarly as in the previous case. 
It is clear that the left-hand side is non-negative, and it equals zero if and only if either $u = v' = \varepsilon$ or $|s_{_{V}}|_{\alpha} = |s_{_{V}}|_{\alpha-1}$ and $|u|_{\alpha} = 0$.

\begin{claim}
The right-hand side is non-positive and, moreover, equals zero if and only if $v' = \alpha^i$ and $\beta = \alpha$.
\end{claim}

\begin{claimproof}
 To begin, we show that $\binom{p_{_{U}} s_{_{U}}}{\alpha(\alpha - 1)} = 1$.
Since $\alpha$ appears in $p_{_{V}}'s_{_{V}}$ (in $\sm(\beta)$), it must also appear in $p_{_{U}}s_{_{U}}$; 
 since 
it does not appear in $s_{_{U}}$,
it appears in $p_{_{U}}$.
Furthermore, there is exactly one occurrence of $\alpha$ in $p_{_{U}} s_{_{U}}$.
It should be noted that in $p_{_{U}}$, $\alpha - 1$  can precede  $\alpha$ (if it appears at all) since $|p_{_{U}}|<m$. Hence, there is only one occurrence of the subword $\alpha (\alpha-1)$, as desired.

Next, we consider $\binom{p_{_{V}}'s_{_{V}}}{\alpha (\alpha-1)}$.
Observe that $s_{_{V}}$ does not contain $\alpha-1$; if it did, then it would be followed by a second occurrence of $\alpha$ in $p_{_{V}}'s_{_{V}}$ since it cannot end with $\alpha-1$, resulting in a contradiction.

Since $\alpha$ appears in $\sm(\beta)$ within $p_{_{V}}'s_{_{V}}$ (and only once), we conclude that
$\binom{p_{_{V}}'s_{_{V}}}{\alpha (\alpha-1)} = 1$ if and only if $\beta = \alpha$. Otherwise, $\binom{p_{_{V}}'s_{_{V}}}{\alpha(\alpha-1)} = 0$. Consequently, the right-hand side is non-positive and equals $0$ if and only if $|v'| = |v'|_{\alpha}$ and $\beta = \alpha$.
\end{claimproof}

For the equation above to be satisfied,
we must have
$|u|_{\alpha} = 0$, $v' = \alpha^i$ for some $i \geqslant 0$, and $\beta = \alpha$.
Additionally, we have established that $|s_{_{V}}|_{\alpha} = |s_{_{V}}|_{\alpha -1} = 0$, regardless of whether $u=\varepsilon$ or not.
It should be noted that if $\alpha-1$ appears for a second time in $p_{_{U}}s_{_{U}}$, it must occur just before $\alpha$ in $p_{_{U}}$ and as the last letter of $p_{_{V}}$; otherwise $p_{_{V}}'s_{_{V}}$ would contain a second occurrence of $\alpha$. If $\alpha-1$ appears only once in $p_{_{U}}s_{_{U}}$,
then $p_{_{U}}$ begins with $\alpha$.
\cref{fig:facsm2} illustrates the situation (the possible occurrences of $\alpha-1$ in $p_{_{U}}$ and $p_{_{V}}$ are not shown).

\begin{figure}[h!t]
\centering
\begin{tikzpicture}
\draw[yscale=.4] (0,0) -- (12,0) -- (12,1) -- (0,1) -- (0,0);
\draw[yscale=.4] (2,1) -- (2,0);
\draw[yscale=.4] (4.5,1) -- (4.5,0);
\draw[yscale=.4] (7.5,1) -- (7.5,0);
\draw[yscale=.4] (10,1) -- (10,0);
\draw[yscale=.4] (12,1) -- (12,0);

\node at (-.5,.2) {$U$};
\node at (.5,.7) {$p_{_{U}}$};
\node at (1,.2) {\small{$\alpha$}};
\node at (3,.2) {\small{$(\alpha-1)\alpha$}};
\draw[decoration={brace},decorate] (2,.5) -- (10,.5) node[above, midway] {$\sm(u)$};
\node at (6,.2) {$\cdots$};
\node at (8.7,.2) {\small{$(\alpha-1)\alpha$}};
\node at (11,.7) {$s_{_{U}}$};

\node at (11.45,.2) {\small{$(\alpha-1)$}};

\begin{scope}[yshift = -35]
\draw[yscale=.4] (0,0) -- (12,0) -- (12,1) -- (0,1) -- (0,0);
\draw[yscale=.4] (1,1) -- (1,0);
\draw[yscale=.4] (3.5,1) -- (3.5,0);
\draw[yscale=.4] (9,1) -- (9,0);
\draw[yscale=.4] (11.5,1) -- (11.5,0);
\draw[yscale=.4] (12,1) -- (12,0);

\node at (-.5,.2) {$V$};
\node at (.5,-.3) {$p_{_{V}}$};
\node at (1.14,.2) {\small{$\alpha$}};
\node at (2.97,.2) {\small{$(\alpha-1)$}};
\node at (3.65,.2) {\small{$\alpha$}};
\draw[decoration={brace,mirror},decorate] (1,-.1) -- (3.5,-.1) node[below, midway] {$\sm(\beta)$};
\draw[decoration={brace,mirror},decorate] (3.5,-.1) -- (11.5,-.1) node[below, midway] {$\sm(v')$};
\draw[decoration={brace},decorate] (0,.5) -- (3.5,.5) node[above, midway] {$p_v'$};
\node at (7,.2) {$\cdots$};
\node at (9.12,.2) {\small{$\alpha$}};
\node at (11,.2) {\small{$(\alpha-1)$}};
\node at (11.8,-.3) {$s_{_{V}}$};
\end{scope}
\end{tikzpicture}
\caption{Illustrating the situation $|u|+1=|v|$.}\label{fig:facsm2}
\end{figure}
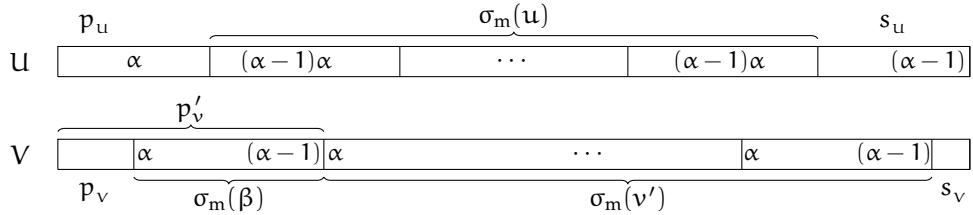

Consider now the sum 
\[
\sum_{x \in \am} \binom{U}{(\alpha-1) x} + \binom{U}{x \alpha} - \binom{V}{(\alpha-1) x} - \binom{V}{x \alpha},
\]
which is equal to $0$ due to the assumption that $U \sim_2 V$. If $\alpha-1$ does not appear in $p_{_{U}}$ (and thus not in $p_{_{V}}$),
then the positive side equals $|u||U|$; recall that $p_{_{U}}$ begins with $\alpha$ in this case. The negative side equals $-|v'||V|-|p_{_{V}}s_{_{V}}|$. This implies that
$p_{_{V}}s_{_{V}} = \varepsilon$. But then $V = \sm(\alpha^{i+1})$ ends with $\alpha-1$,  a contradiction.

If $\alpha-1$ does appear in $p_{_{U}}$ and $p_{_{V}}$, then the positive side is equal to $\left(|u|+1\right)|U|$ whereas the negative side is equal to
$-\left(|v'|+1\right)|V| - |s_{_{V}}|$. 
Hence, $s_{_{V}} = \varepsilon$, and again $V$ ends with $\alpha-1$. This shows that the case where $|u|+1 = |v|$
is impossible.

We have shown that the only possible way for $U \sim_2 V$ to hold is by having the claimed $\sm$-factorizations, thus completing the proof.
\end{proof}


\subsection{The induction step}

\begin{proof}[Proof of \cref{lem:k-1-factorisation-first-last}]
Suppose the two factors $U$ and $V$ possess the $\smn{k-1}$-factorizations $U = \smn{k-1}(u)$ and $V = \smn{k-1}(v)$, where $u\sim_1 v$. In that case, they are $k$-binomially equivalent, as stated in \cref{prop:phik}.

We consider the converse claim by induction on $k$, starting with the base case $k=2$ which is addressed by
\cref{lem:generalised-k=2-distinct-extremal-letters}.
Assume that the claim holds for some $k \geqslant 2$, and consider $U \sim_{k+1} V$ with $U \neq V$, beginning and ending with distinct letters.
Suppose $U$ and $V$ have $\smn{k}$-factorizations of the form $p_{_{U}} \smn{k}(u) s_{_{U}}$ and $p_{_{V}} \smn{k}(v) s_{_{V}}$, respectively, where $|u|$,
$|v| \geqslant 0$ (note that such factorizations are guaranteed by \cref{rem:all-factors-smk-factorization}). By factoring out full $\smn{k-1}$-images from $p_{_{U}}$, $s_{_{U}}$, $p_{_{V}}$, and $s_{_{V}}$, we obtain the corresponding $\smn{k-1}$-factorizations
of the form
\[
U = p_{_{U}}' \smn{k-1}\left(\gamma_u\sm(u)\delta_u\right) s_{_{U}}' \quad \text{and} \quad V = p_{_{V}}' \smn{k-1}\left(\gamma_v\sm(u)\delta_v\right) s_{_{V}}',
\] where $p_w = p'_w\smn{k-1}(\gamma_w)$ and
$s_w = \smn{k-1}(\delta_w)s'_w$ for $w \in \{u,v\}$.
Under this assumption, it follows that $U \sim_k V$,  and by the induction hypothesis, we have
$p_w's_w' = \varepsilon$ for $w \in \{u,v\}$. Furthermore, $\gamma_u \sm(u) \delta_u \sim_1 \gamma_v \sm(v) \delta_v$,
where the words $U$ and $V$ begin and end with distinct letters.

\textbf{First, assume that $|u| = |v|$.}
Then $\gamma_u \delta_u \sim_1 \gamma_v \delta_v$.
If both $\delta_u$ and $\delta_v$ are empty, it follows that $\gamma_u \sim_1 \gamma_v$. Since $\gamma_u$ and $\gamma_v$ are suffixes of
$\sm$-images of letters, they must be equal.
Moreover, since $U$ and $V$ begin with distinct letters, this implies that $\gamma_u = \gamma_v = \varepsilon$.
Thus, we have $U = \smn{k}(u)$ and $V = \smn{k}(v)$, confirming the claimed $\smn{k}$-factorizations by 
\cref{pro:both_dir}.

We now proceed to the case where either $\delta_u$ or $\delta_v$ is non-empty. W.l.o.g, we may assume
that that $\delta_u \neq \varepsilon$, and let $\alpha-1$  denote its final letter.
In particular, 
$\alpha$ does not occur in $\delta_u$.
We can apply \cref{lem:bigdiff} to $\smn{k-1}(\gamma_u \sm(u)\delta_u)$ and $\smn{k-1}(\gamma_v \sm(v) \delta_v)$,
using  $\alpha(\alpha-1)\cdots$ in place of $0\overline{1} \cdots$.
Since these two words are assumed to be $(k+1)$-binomially equivalent, we obtain, by dividing by $m^{\binom{k}{2}-1}$
\begin{multline*}
0
=
m\biggl[ |u|_\alpha-|v|_\alpha + |u|\,  (|\gamma_u|_\alpha-|\gamma_v|_\alpha +|\delta_u|_{\alpha-1}  - |\delta_v|_{\alpha-1} )\biggr]
      \\+ \sum_{b\in\am}\left( \binom{\gamma_u\delta_u}{b(\alpha-1)}-\binom{\gamma_v\delta_v}{b(\alpha-1)}+ \binom{\gamma_u\delta_u}{\alpha b}-\binom{\gamma_v\delta_v}{\alpha b}\right).
\end{multline*}
By observing that $|\gamma_w|_\alpha = |\gamma_w\delta_w|_\alpha - |\delta_w|_\alpha$,  where $w \in \{u,v\}$,  we can simplify the first term as follows
\[
\begin{aligned}
    m \bigl[ |u|_{\alpha} - |v|_{\alpha} + |u|(|\delta_{u}|_{\alpha - 1} - |\delta_u|_{\alpha} -|\delta_v|_{\alpha-1} + |\delta_v|_{\alpha})\bigr] \\
    = m\bigl[|u|_{\alpha} + |u|(|\delta_v|_{\alpha} - |\delta_v|_{\alpha-1})\bigr] + m(|v|-|v|_{\alpha}).
\end{aligned}
\]
Let us define
$\Delta_{\alpha}$
as 
\[\Delta_{\alpha}:= \sum_{b\in\am}\left( \binom{\gamma_v\delta_v}{b(\alpha-1)}+\binom{\gamma_v\delta_v}{\alpha b}-\binom{\gamma_u\delta_u}{b(\alpha-1)}- \binom{\gamma_u\delta_u}{\alpha b}\right).\] Rearranging the previous equation, we obtain
\begin{equation}\label{eq:simplified-bigdiff-equal-length}
m\bigl[|u|_{\alpha} + |u|\left(|\delta_v|_{\alpha} - |\delta_v|_{\alpha-1}\right)\bigr]
=
m(|v|_{\alpha}-|v|) + \Delta_{\alpha}.
\end{equation}
Recall that $|\delta_u|_{\alpha-1}=1$ and $|\delta_u|_\alpha=0$. 
Notice that the left-hand side is non-negative, and the only way where it could become negative is if $\alpha-1$ appears in $\delta_v$.
However, since $\delta_v$ does not end with $\alpha-1$ (as $\delta_u$ ends with it), this occurrence of $\alpha-1$ must be followed by $\alpha$.
Furthermore, the left-hand side is equal to zero if and only if $|u|_{\alpha} = 0$ and $|\delta_v|_{\alpha} = |\delta_v|_{\alpha-1}$.

Next, we show that the right-hand side is non-positive. Indeed, since $m\left(|v|_\alpha - |v|\right)$  is non-positive, it is sufficient to show that the sum $\Delta_{\alpha}$ is also non-positive.

\begin{claim}
The value of $\Delta_{\alpha}$ is $-|\delta_v|$ if and only if $|\gamma_u\delta_u|_{\alpha-1} = |\gamma_u\delta_u|_{\alpha} + 1$. Otherwise, $\Delta_\alpha = -|\gamma_u\delta_u|$. Moreover, in the former case, $\alpha - 1$ is the last letter of $\gamma_v$.
\end{claim}

\begin{claimproof}
 We first observe that
 \[
 \sum_{b \in \am} \binom{x}{\alpha b}
 \]
counts, for each occurrence of the letter $\alpha$ in the word $x$, the number of letters that occur to its right. Similarly,
\[
\sum_{b \in \am} \binom{x}{b(\alpha-1)},
\]
counts, for each occurrence of $\alpha-1$, the number of letters occurring to its left.

We then consider the occurrences of $\alpha$ and $\alpha-1$ in the two words $\gamma_u\delta_u$ and $\gamma_v\delta_v$,
as well as their contributions to the sum 
$\Delta_\alpha$. 
Notice that since $\delta_u$ does not contain $\alpha$, there is at most one occurrence of $\alpha$ in $\gamma_u\delta_u$. Furthermore, there can be at most two occurrences of $\alpha-1$.

First of all, note that $\alpha$ does not appear in $\delta_u$. 
The contribution from
$\alpha-1$,
 as the last letter of
$\delta_u$,
to the term
$-\binom{\gamma_u\delta_u}{b(\alpha-1)}$
results in a value of
$-|\gamma_u\delta_u|+1$
to
$\Delta_\alpha$.

\begin{itemize}
\item First, assume that $|\gamma_u\delta_u|_{\alpha-1} = 1$.
We proceed by dividing this into additional subcases, considering whether $\alpha-1$
appears in  $\delta_v$ or not.

\begin{itemize}
\item If $\delta_v$ contains $\alpha-1$, then this occurrence must be followed by
  by $\alpha$. 
  These two occurrences provide $|\gamma_v\delta_v|-2$ towards $\Delta_\alpha$. Now, $\alpha$ must appear in $\gamma_u$, whereas $\alpha-1$  should not.
  This situation occurs only if $\gamma_u$ starts with $\alpha$, as it is a suffix of the image of a letter (as depicted in \cref{fig:facsit1}).
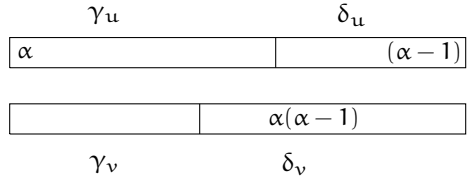
\begin{figure}[h!t]
\centering
\begin{tikzpicture}
\draw[yscale=.4] (0,0) -- (6,0) -- (6,1) -- (0,1) -- (0,0);
\draw[yscale=.4] (3.5,1) -- (3.5,0);

\node at (1.25,.7) {$\gamma_u$};
\node at (0.2,.2) {\small{$\alpha$}};
\node at (5.45,.2) {\small{$(\alpha-1)$}};
\node at (4.5,.7) {$\delta_u$};

\begin{scope}[yshift = -25]
\draw[yscale=.4] (0,0) -- (6,0) -- (6,1) -- (0,1) -- (0,0);
\draw[yscale=.4] (2.5,1) -- (2.5,0);

\node at (1.25,-.4) {$\gamma_v$};
\node at (4,.2) {\small{$\alpha(\alpha-1)$}};
\node at (3.75,-.4) {$\delta_v$};
\end{scope}
\end{tikzpicture}
\caption{Illustrating the situation $|\gamma_u\delta_u|_{\alpha-1}=|\gamma_v|_{\alpha-1}=1$.}\label{fig:facsit1}
\end{figure}

  This occurrence contributes $-|\gamma_u\delta_u|+1$ towards $\Delta_\alpha$.
Consequently, in this case, we have: 
\[
\Delta_\alpha = -|\gamma_u\delta_u|+1 + |\gamma_v\delta_v|-2 -|\gamma_u\delta_u|+1 = -|\gamma_u\delta_u|.
\]
Moreover, in this situation, we also have 
\[
|\gamma_u\delta_u|_{\alpha} = |\gamma_u\delta_u|_{\alpha-1}.
\]

\item If $\delta_v$ does not contain $\alpha-1$, then $\gamma_v$ contains $\alpha-1$.
We then further split this case based on whether
$\alpha$ appears in $\gamma_u\delta_u$ or not.

\begin{itemize}
\item Assume that $\alpha$ appears in $\gamma_v\delta_v$.
In this case, either  $\gamma_v$ contains $\alpha$ as the letter directly following $\alpha-1$ with $|\delta_v|_\alpha = 0$, or $\alpha-1$ is the last letter of $\gamma_v$ and $\delta_v$ begins with $\alpha$ (because, in this case,
$\alpha-1$ does not appear in $\delta_v$). In both cases, we have $\alpha-1$ followed by $\alpha$ in $\gamma_v\delta_v$, resulting in a contribution of $|\gamma_v\delta_v|-2$.
Now, $\alpha$ appears in $\gamma_u$, while $\alpha-1$ does not. This is possible only when $\alpha$ is the first letter of $\gamma_u$, thus contributing $-|\gamma_u\delta_u|+1$ towards $\Delta_\alpha$. 
Hence, we find
\[
\Delta_\alpha = -|\gamma_u\delta_u|+1 + |\gamma_v\delta_v|-2-|\gamma_u\delta_u|+1 = -|\gamma_u\delta_u|.
\]
Note also that in this case 
\[
|\gamma_u\delta_u|_{\alpha} = |\gamma_u\delta_u|_{\alpha-1}.
\]

\item Assume that $\alpha$ does not occur in $\gamma_v\delta_v$.
Consequently, the occurrence $\alpha-1$ in $\gamma_v$ must be its last letter. Therefore, in this case, we have
\[
\Delta_\alpha = -|\gamma_u\delta_u|+1 + |\gamma_v|-1 = -|\delta_v|.
\]
Moreover, in this case, we have 
\[
|\gamma_u\delta_u|_{\alpha} + 1 = |\gamma_u\delta_u|_{\alpha-1}
\]
and 
$\alpha-1$
is the last letter of
 $\gamma_u$.
\end{itemize}
\end{itemize}

\item Assume secondly that$|\gamma_u\delta_u|_{\alpha-1} = 2$. Therefore, $\delta_v$ must contain $\alpha-1$, and this is followed by
$\alpha$ since $\delta_v$ cannot end with $\alpha-1$.
These occurrences contribute $|\gamma_v \delta_v|-2$ to $\Delta_\alpha$.
Now, $\gamma_u$ also contains $\alpha-1$. Since $\alpha$ appears in $\delta_v$,  it must also occur in $\gamma_u$ causing the two letters to appear consecutively.
These occurrences contribute $-|\gamma_u\delta_u|+2$ to $\Delta_\alpha$.
Finally, we consider the contribution of $\alpha-1$ in $\gamma_v$.
Since $\alpha$ is already present in $\delta_v$ it cannot occur in
$\gamma_v$, thus  $\gamma_v$ ends with $\alpha-1$. 
This provides $\Delta_\alpha$ with $|\gamma_v|-1$.  \cref{fig:facsit2} illustrates this situation.
\begin{figure}[h!t]
\centering
\begin{tikzpicture}
\draw[yscale=.4] (0,0) -- (6,0) -- (6,1) -- (0,1) -- (0,0);
\draw[yscale=.4] (3.5,1) -- (3.5,0);

\node at (1.25,.7) {$\gamma_u$};
\node at (1.4,.2) {\small{$(\alpha-1)\alpha$}};
\node at (5.45,.2) {\small{$(\alpha-1)$}};
\node at (4.5,.7) {$\delta_u$};

\begin{scope}[yshift = -25]
\draw[yscale=.4] (0,0) -- (6,0) -- (6,1) -- (0,1) -- (0,0);
\draw[yscale=.4] (2.5,1) -- (2.5,0);

\node at (1.25,-.4) {$\gamma_v$};
\node at (2,.2) {\small{$(\alpha-1)$}};
\node at (4,.2) {\small{$(\alpha-1)\alpha$}};
\node at (3.75,-.4) {$\delta_v$};
\end{scope}
\end{tikzpicture}
\caption{Illustrating the situation $|\gamma_u\delta_u|_{\alpha-1}=2$.}\label{fig:facsit2}
\end{figure}
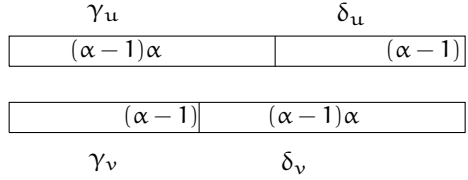

Thus, in this case, we have
\[
\begin{aligned}
    \Delta_\alpha &= -|\gamma_u \delta_u| + 1 + |\gamma_v \delta_v| - 2 - |\gamma_u \delta_u| + 2 
    + |\gamma_v| - 1 \\ &= -|\gamma_u \delta_u| + |\gamma_v| = -|\delta_v|.
\end{aligned}
\]
Observe once more that 
\[
|\gamma_u\delta_u|_{\alpha} + 1 = |\gamma_u\delta_u|_{\alpha-1},
\]
and furthermore,  $\alpha-1$ is the last letter of
$\gamma_v$.

\end{itemize}

All cases have been considered, and each one leads to the desired conclusion.
\end{claimproof}

The preceding claim indicates that $\Delta_\alpha$ in \eqref{eq:simplified-bigdiff-equal-length} is non-positive.
For \eqref{eq:simplified-bigdiff-equal-length} to hold true, it must be the case that $\delta_v = \varepsilon$
and $|v|_{\alpha} = |v|$. Moreover, $\gamma_v$ ends with $\alpha-1$ as stated in the above claim. However, $\gamma_v\sm(v)$ ends with $\alpha-1$: 
either
$\gamma_v$ ends with $\alpha-1$ when $v\neq \varepsilon$, or $\gamma_v$ ends with $\alpha-1$ if $v = \varepsilon$.  This conclusion contradicts the initial assumption that the words end with distinct letters. Therefore, we have shown that the case $|u| = |v|$ is impossible if either of the words $s_{_{U}}$ or $s_{_{V}}$ is empty.

\textbf{Second, assume that $|u| \neq |v|$.}
Due to the length constraints, it follows that
$\left||u|-|v|\right| = 1$. 
W.l.o.g, let us assume that $|v| = |u|+1$ and express $v$ in the form $v = \beta v'$, where $\beta \in \am$.
Consequently, we have $\gamma_u\delta_u \sim_1 \gamma_v\sm(\beta)\delta_v$ implying that both $\gamma_u$ and $\delta_u$ are non-empty.
Let $\alpha-1$ denote the last letter of $\delta_u$.

We may now apply \cref{lem:bigdiff}, since we have $|u| = |v'|$ and $\gamma_u\delta_u \sim_1 \gamma_v\sm(\beta)\delta_v$ (with $\alpha$ in place of $0$).
Rewriting
$\gamma_v' $  as  $ \gamma_v\sm(\beta)$, we obtain (after dividing both sides by $m^{\binom{k}{2}-1}$)
  \begin{multline*} 
    0=  m \biggl[ |u|_{\alpha}-|v'|_\alpha + |u|\,  (|\gamma_u|_\alpha - |\gamma_v'|_\alpha +|\delta_u|_{\alpha-1}  - |\delta_v|_{\alpha-1} )\biggr]\\
      +\sum_{b\in\am}\left( \binom{\gamma_u\delta_u}{b(\alpha-1)}-\binom{\gamma_v'\delta_v}{b(\alpha-1)}+ \binom{\gamma_u\delta_u}{\alpha b}-\binom{\gamma_v'\delta_v}{\alpha b}\right).
    \end{multline*}
    Write again 
    \[
    |\gamma_u|_\alpha -|\gamma_v'|_\alpha +|\delta_u|_{\alpha-1}  - |\delta_v|_{\alpha-1} = |\delta_u|_{\alpha-1} - |\delta_u|_\alpha - |\delta_v|_{\alpha-1} + |\delta_v|_\alpha.    \]
    Furthermore, defining
    \[
    \Delta_\alpha = \sum_{b\in\am}\left( \binom{\gamma_v'\delta_v}{b(\alpha-1)} + \binom{\gamma_v'\delta_v}{\alpha b} - \binom{\gamma_u\delta_u}{b(\alpha-1)}- \binom{\gamma_u\delta_u}{\alpha b}\right),
    \]
    and recalling that $|\delta_u|_{\alpha-1} = 1$ and $|\delta_u|_{\alpha} = 0$,
    the preceding equation simplifies to
\begin{equation}\label{eq:S_h-equation-diff-lengths}
m \left(|u|_\alpha + |u|(|\delta_v|_{\alpha} - |\delta_v|_{\alpha-1})\right) = m\left(|v'|_{\alpha} - |v'|\right) +  \Delta_\alpha.
\end{equation}

Using arguments analogous to those in \textbf{Case 1}, the left-hand side is shown to be non-negative. Moreover, it equals zero if and only if $|u|_{\alpha} = 0$ and $|\delta_v|_{\alpha} = |\delta_v|_{\alpha-1}$.
Additionally, we compute the right-hand side in an analogous manner, showing that it is non-positive.

\begin{claim}\label{cl:sign-Sh-diff}
We have $\Delta_{\alpha} = -|\delta_v|$, or $\Delta_{\alpha} = -|\gamma_v\delta_v|$ if and only if $\beta = \alpha$. In all other cases,
$\Delta_{\alpha} = -|\gamma_u\delta_u|$ or $\Delta_{\alpha} = -m -|\delta_v|$.
\end{claim}

\begin{claimproof}
We once again consider the occurrences of $\alpha$ and $\alpha-1$ in the two words $\gamma_u\delta_u$ and $\gamma_v'\delta_v$,
and examine their contributions to the sum $\Delta_\alpha$.

Recall that $\alpha-1$ is the last letter of $\delta_u$.
Therefore, $\alpha$ can appear at most once in  $\gamma_u \delta_u$.
Since $\gamma_u\delta_u \sim_1 \gamma_v'\delta_v = \gamma_v\sm(\beta)\delta_v$, and $\alpha$ appears in
$\sm(\beta)$, we conclude that $\alpha$ appears precisely once in $\gamma_u\delta_u$, and therefore must appear in  $\gamma_u$.

$\bullet$
\textbf{Occurrences in $\delta_u$ and $\sm(\beta)$:}
The occurrence of $\alpha-1$ as the last letter of  $\delta_u$ contributes $\Delta_\alpha$ to
$-|\gamma_u\delta_u|+1$.

Since $\sm(\beta)$ contains both $\alpha-1$ and $\alpha$,
there are two possible cases:
\begin{enumerate}
    \item[1)]
    if $\alpha-1$ is the last letter of $\sm(\beta)$ (which is equivalent to $\beta=\alpha$), the contribution is
    \(
    |\gamma_v'|-1 + |\delta_v| + m-1 = |\gamma_v'\delta_v| + m-2.
    \)

    \item[2)] 
    Otherwise, if the two letters appear consecutively, the contribution is $|\gamma_v'\delta_v|-2$.
\end{enumerate}

$\bullet$ \textbf{Other occurrences:} 
 We consider two cases based on the number of the occurrences of $\alpha-1$.
 
\begin{itemize}
\item Suppose first that $\alpha-1$ appears exactly once in $\gamma_u\delta_u$. 
Consequently, $\alpha$ must be the first letter of $\gamma_u$,
contributing $-|\gamma_u\delta_u| + 1$ to $\Delta_\alpha$. 
Thus, in this case, $\Delta_\alpha = m-|\gamma_u\delta_u| = -|\gamma_v\delta_v|$ if $\beta = \alpha$, and
$\Delta_\alpha = -|\gamma_u\delta_u|$ otherwise.

\item Now, assume that $\alpha-1$ occurs for a second time in $\gamma_u\delta_u$.
Since $\alpha$ must appear in $\gamma_u$ 
with
$\alpha-1$, the letters must appear consecutively, with $\alpha-1$ preceding $\alpha$. These occurrences give the contribution
$-|\gamma_u\delta_u|+2$. It remains to consider the second occurrence of $\alpha-1$ in $\gamma_v'\delta_v$. Notice that $\alpha-1$
cannot appear in $\delta_v$; since it cannot be the last letter of $\delta_v$, it would be followed by a second $\alpha$. Thus
$\alpha-1$ appears in $\gamma_v$. Since $\gamma_v$ does not contain $\alpha$, we must have that $\alpha-1$ is the last letter of
$\gamma_v$. This gives the contribution $|\gamma_v|-1 = |\gamma_v'\delta_v| - |\delta_v| - m - 1$.

In total, we have
\begin{align*}
    \Delta_\alpha &= |\gamma_v'\delta_v|-2 -|\gamma_u\delta_u|+1 -|\gamma_u\delta_u|+2 + |\gamma_v'\delta_v| - |\delta_v| - m - 1 = -|\delta_v|-m
\end{align*}
if $\beta \neq \alpha$, and $\Delta_\alpha = -|\delta_v|$ if $\beta = \alpha$.\qedhere
\end{itemize}
\end{claimproof}

We are now ready to conclude with the proof. The claim above asserts that the only way \eqref{eq:S_h-equation-diff-lengths} holds is if both sides are equal to zero. In particular, this implies that 
$|v'|_{\alpha} = |v'|$, $\beta = \alpha$, and $\delta_v = \varepsilon$. Consequently, the last letter of $\sm(v')$ is $\alpha-1$, leading us to conclude that the words 
$U = \smn{k-1}\left(\gamma_u \sm(u) \delta_u\right)$ and
$V = \smn{k-1}(\gamma_v\sm(\beta v'))$  both end with the last letter of $\smn{k-1}(\alpha-1)$. This is a contradiction, in the case where $|u| \neq |v|$.

Thus, we conclude that the only possible way for $U \sim_{k+1} V$ is when $\delta_u = \delta_v = \gamma_u = \gamma_v = \varepsilon$ and $u \sim_1 v$. Hence, the proof is complete.
\end{proof}


\section{Abelian Complexity for Short Factors}\label{sec:abco}

The initial values of the abelian complexity  $\bc{\infw{t}_m}{1}(\ell)$ of $\mathbf{t}_m$, for $1\leqslant\ell<m$ are presented in \cref{tab:aml}.
For lengths $\ell\geqslant m$, the function is periodic with period~$m$, i.e., $\bc{\infw{t}_m}{1}(\ell+m)=\bc{\infw{t}_m}{1}(\ell)$, and its behavior is fully described by \cref{thm:abelian_complexity} from \cite{ChenWen2019}.
Thus, the following proposition complements the findings of Chen et al.

\begin{table}[h!t]
  {\small
    \[
    \begin{array}{ccccccccccccc}
      2 & \text{} & \text{} & \text{} & \text{} & \text{} & \text{} &
      \text{} & \text{} & \text{} & \text{} \\
      3 & 6 & \text{} & \text{} & \text{} & \text{} & \text{} &
      \text{} & \text{} & \text{} & \text{} \\
      4 & 10 & 12 & \text{} & \text{} & \text{} & \text{} & \text{} &
      \text{} & \text{} & \text{} \\
      5 & 15 & 20 & 25 & \text{} & \text{} & \text{} & \text{} &
      \text{} & \text{} & \text{} \\
      6 & 21 & 30 & 39 & 42 & \text{} & \text{} & \text{} & \text{} &
      \text{} & \text{} \\
      7 & 28 & 42 & 56 & 63 & 70 & \text{} & \text{} & \text{} &
      \text{} & \text{} \\
      8 & 36 & 56 & 76 & 88 & 100 & 104 & \text{} & \text{} & \text{}
      & \text{} \\
      9 & 45 & 72 & 99 & 117 & 135 & 144 & 153 & \text{} & \text{} &
      \text{} \\
      10 & 55 & 90 & 125 & 150 & 175 & 190 & 205 & 210 & \text{} &
      \text{} \\
      11 & 66 & 110 & 154 & 187 & 220 & 242 & 264 & 275 & 286 &
      \text{} \\
      12 & 78 & 132 & 186 & 228 & 270 & 300 & 330 & 348 & 366 & 372 \\
      13 & 91 & 156 & 221 & 273 & 325 & 364 & 403 & 429 & 455 & 468 & 481 & \text{} \\
      14 & 105 & 182 & 259 & 322 & 385 & 434 & 483 & 518 & 553 & 574 & 595 & 602 \\
    \end{array}
    \]
  }
  \caption{Values of \texorpdfstring{$\bc{\infw{t}_m}{1}(\ell)$}{bc(t_m,1)(ell)} for $1 \leqslant \ell < m \leqslant 14$.}
  \label{tab:aml}
\end{table}

\begin{proposition}\label{pro:abco_small}
The initial values of the abelian complexity  $\bc{\infw{t}_m}{1}(\ell)$ of the generalized Thue--Morse word over $m$ letters are given as follows.

\begin{itemize}
    \item For odd $\ell<m$, say $\ell=2\ell'+1$, where $\ell'\ge 0$, we have
\[\bc{\infw{t}_m}{1}(\ell)=m \left(1 - \ell' - \ell'^2 + \ell' m\right).\]

\item For even $\ell<m$, we have
\[\bc{\infw{t}_m}{1}(\ell)  =  \frac{m}{4} \left(6 - \ell^2 - 2 m + 2 \ell m\right).\]
\end{itemize}

\end{proposition}

\begin{proof}
 By \cref{lem:boundaryseq}, every pair $(i,j)\in\mathcal{A}_m^2$ appears in $\mathbf{t}_m$. 
 Thus,
  any factor $w$ of length $\ell<m$ can be written as $w=ps$, where $p$ is a suffix of some $\sm(i)$ and $s$ is a prefix of some $\sm(j)$. Our aim is to count the possible Parikh vectors for such $w$. Since we are dealing with abelian equivalence, and the images of a letter under $\sm$ are cyclic permutations of $01\cdots (m-1)$, we can limit ourselves to $|p|=\ell,\ell-1,\ldots, \lceil\ell/2\rceil$. When $p$ is shorter than $s$, we obtain exactly the same Parikh vectors.

If $|p|=\ell$, then $p$ is of the form $t\, (t+1) \cdots (t+\ell-1)$, which is a factor of some $\sm(i)$. This corresponds to the $m$ cyclic permutations of the Parikh vector $1^\ell 0^{m-\ell}$ (expressed as a word of length~$m$).

If $|p|=\ell-1$ and $|s|=1$, there are $m$ possible suffixes of $\sm(i)$ of the form $t\, (t+1)\cdots (t+\ell-2)$, where $t\in\am$.
We need to determine which $j\in\am$ provides Parikh vectors that have not already been listed. Here, $s$ is the first letter of $\sm(j)$, which is $j$. If $j=t-1$ or $j=t+\ell-1$, then we get a Parikh vector from the first case.
Thus, for $j$, we can choose any elements in $\am$ except these two,
resulting in $m-2$ possibilities and a total of $m(m-2)$ new Parikh vectors. 
Note that we obtain Parikh vectors (along with their cyclic permutations) of the form $1^{\ell-1}0^r10^s$ with some isolated $1$,  where $r,s>0$ and $r+s=m-\ell$, or of the form $1^r 2 1^{\ell-r-1} 0^{m-\ell}$, with one $2$ in any position within the block of size $\ell-1$.

If $|p|=\ell-2$ and $|s|=2$, this case is similar. We have $m$ possible suffixes of the form $t\, (t+1)\cdots (t+\ell-3)$, where $t\in\am$.
We need to determine which $j\in\am$ provides new Parikh vectors. Here, $s$ is the first two letters of $\sm(j)$, which are $j (j+1)$.
If $j\in\{t-2,t-1,t+\ell-3,t+\ell-2\}$, the Parikh vectors are already described in the first two cases. Otherwise, we obtain new vectors either with a block $1^r 22 1^{\ell-r-2}$, or with two isolated block $1^2$ and $1^{\ell-2}$. This results in $m.(m-4)$ new Parikh vectors.

In general, if $|p|=\ell-u$ and $|s|=u$ with $\ell/2>u$, then $p$ is of the form $t\, (t+1)\cdots (t+\ell-u-1)$.
To obtain new Parikh vectors, either with a block $1^s 2^u 1^{\ell-s-u}$, or with two isolated blocks $1^u$ and $1^{\ell-u}$, from $s=j(j+1)\cdots (j+u-1)$, then $j$ cannot be in $\{t-u,\ldots,t-1,t+\ell-2u,\ldots,t+\ell-u-1\}$. Therefore,  $j$ can take $m-2u$ values.

In conclusion, if $\ell$ is odd of the form $\ell=2\ell'+1$, we obtain a total of 
\[\bc{\infw{t}_m}{1}(\ell)=m+\sum_{u=1}^{\ell'}m(m-2u)
  =m \left(1 - \ell' - \ell'^2 + \ell' m\right).\]
Now, if $\ell$ is even, we still have to consider the situation where $|p|=|s|=\ell/2$. In this case, $p$ and $s$ have  symmetric roles, and we should avoid double counting.
We need to select two elements $i,j\in\am$  that are at distance greater than $\ell/2$ from each other (over $\mathbb{Z}/(m\mathbb{Z})$) in order to obtain Parikh vectors that are a cyclic permutation of $1^{\ell/2}0^r1^{\ell/2}0^s$, where $r,s>0$. The number of such pairs $\{i,j\}$, where $j\not\in\{i-\ell/2,\ldots,i-1,i,i+1,\ldots,i+\ell/2\}$, is given by $m (m-\ell-1)/2$. There are also  $m$ permutations of $2^{\ell/2}0^{m-\ell/2}$ when $p=s$. Hence, for even $\ell$, we obtain
\[\bc{\infw{t}_m}{1}(\ell)=m+\sum_{u=1}^{\frac{\ell}{2}-1}m(m-2u)+\frac{m (m-\ell-1)}{2}+m
  =  \frac{m}{4} \left(6 - \ell^2 - 2 m + 2 \ell m\right).\]
\end{proof}

\begin{remark}
Interestingly, the infinite triangular array whose initial elements are given in \cref{tab:aml}, exhibits several intriguing combinatorial properties and identities.

\begin{itemize}
\item Regarding the rows of the triangle, the following relation holds for $1\leqslant \ell <m-4$
  \[ \bc{\infw{t}_{m}}{1}(\ell+4)= 2\bc{\infw{t}_{m}}{1}(\ell+3)-2 \bc{\infw{t}_{m}}{1}(\ell+1)+ \bc{\infw{t}_{m}}{1}(\ell).\]
 This relation can be easily deduced from the previous proposition.
  For $m\geqslant 5$, the initial conditions are given by
  \[ \left(\bc{\infw{t}_{m}}{1}(1),\ldots,\bc{\infw{t}_{m}}{1}(4)\right)=\left(m, m(m+1)/2,m(m-1),m(3m-5)/2\right).\]
\item Similarly, for each column,  the following holds for all $m\geqslant 2$ and all $\ell<m$, 
  \[ \bc{\infw{t}_{m+3}}{1}(\ell)= 3\bc{\infw{t}_{m+2}}{1}(\ell)-3 \bc{\infw{t}_{m+1}}{1}(\ell)+ \bc{\infw{t}_{m}}{1}(\ell).\]
\item Furthermore, the diagonal and parallels to the diagonal  $\left(\bc{\infw{t}_{\ell+2+i}}{1}(\ell+1)\right)_{\ell\geqslant 0}$ for all $i\geqslant 0$ satisfy the same recurrence relation of order $6$
\[
x_{n+6}=2x_{n+5}+x_{n+4}-4x_{n+3}+x_{n+2}+2x_{n+1}-x_n.
\]

\item The sequence \( \left(\bc{\infw{t}_{m}}{1}(m-2)\right)_{m\geqslant 3}=3, 10, 20, 39, 63, 100, 144,\ldots \) appears in several entries  of the OEIS, as {\tt A005997} (number of paraffins) and {\tt A272764} (number of positive roots in reflection group~$E_n$), among others.

\item The sequence \( \left(\bc{\infw{t}_{2m+1}}{1}(2m)\right)_{m\geqslant 1}\)
is given by
\( 2m^3+m^2+2m+1\).
  
\item The sequence \( \left((\bc{\infw{t}_{2m}}{1}(2m-1))/2\right)_{m\geqslant 1}=1, 6, 21, 52, 105, 186, 301,\ldots\) is the sequence of $q$-factorial numbers $([3]!_q)_{q\geqslant 0}$ where 
\[
[3]!_q=
\frac{(1-q)(1-q^2)(1-q^3)}{(1-q)^3}=
(1 + q) (1 + q + q^2).
\]
 It appears as {\tt A069778} in the OEIS.
\end{itemize}
\end{remark}



\section{Description of the Abelian Rauzy Graphs}\label{sec:arg}

The abelian Rauzy graph is defined in \cref{def:abr}.
Refer to \cref{sec:strategy} for the definitions of the sets $Y_{m,L},Y_{m,R}$, and $Y_m$.
The aim of this section is to count the number of edges in the abelian Rauzy graph $G_{m,\ell}$ of order $\ell$ for $\infw{t}_m$, where $\ell<2m$, as well as determine the size of the corresponding set $Y_m(\ell)$. These expressions, together with \cref{pro:5.5}, lead to \cref{thm:inter}.

The structure of these graphs depends on the value of the parameter $\ell$. Specifically, the behavior varies significantly depending on whether $\ell<m$ or $\ell\geqslant m$.

\begin{example}\label{exa:abr64}
 \cref{fig:abr64}   depicts the graph~$G_{6,4}$.  To keep clarity in the figure, we have omitted the edge labels.
 The color of each edge is determined by the second component of its label.
 Thus, two edges originating from the same vertex and sharing the same color correspond to the same element of $Y_{m,R}$. The vertices are labeled with Parikh vectors.
 According to  \cref{pro:abco_small}, $\bc{\infw{t}_{6}}{1}(4)=39$, which implies that the graph~$G_{6,4}$ has $39$ vertices.
 The symmetry of the graph results from \cref{lem:permut}.
\begin{figure}[h!t]
  \centering
  \includegraphics[width=11cm]{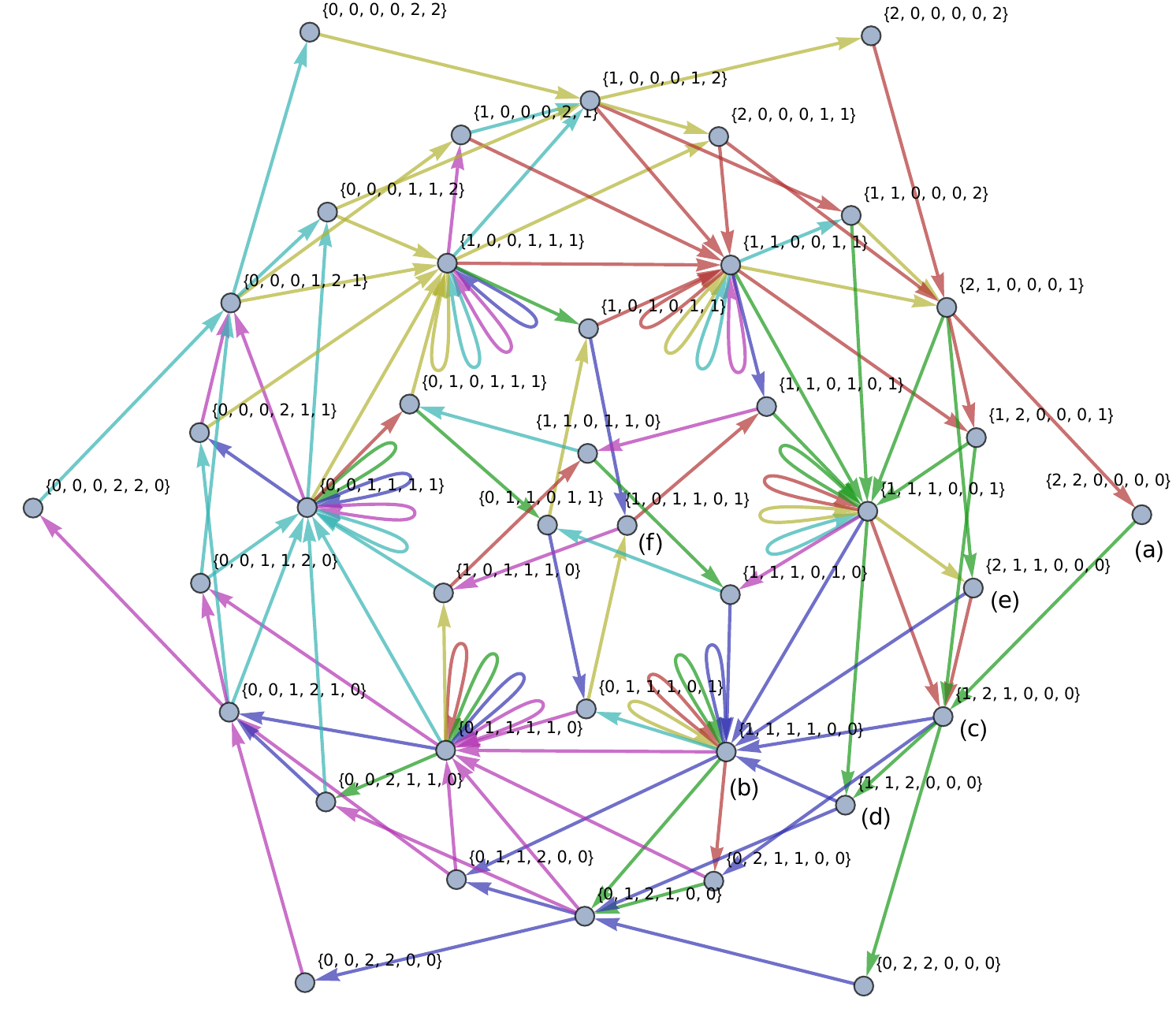}
  \caption{Abelian Rauzy graph $G_{6,4}$ of order $4$ for $\infw{t}_6$.}
  \label{fig:abr64}
\end{figure}
\end{example}

\begin{example}\label{exa:abr54}
   \cref{fig:abr54}  depicts the graph~$G_{5,4}$, which  has $\bc{\infw{t}_{5}}{1}(4)=25$ vertices.
   This example may help the reader follow the developments presented in the proof below, where the case of odd 
$m$ and even 
$\ell$ is discussed.
   Providing two distinct examples is insightful.
    \cref{fig:abr54} exhibits a $5$-fold symmetry in the graph. 
    However,  \cref{fig:abr64}  shows that the three central vertices exhibit a different behavior, specifically a $3$-fold symmetry, instead of the $6$-fold symmetry present in the rest of the graph.

  \begin{figure}[h!t]
  \centering
  \includegraphics[width=11cm]{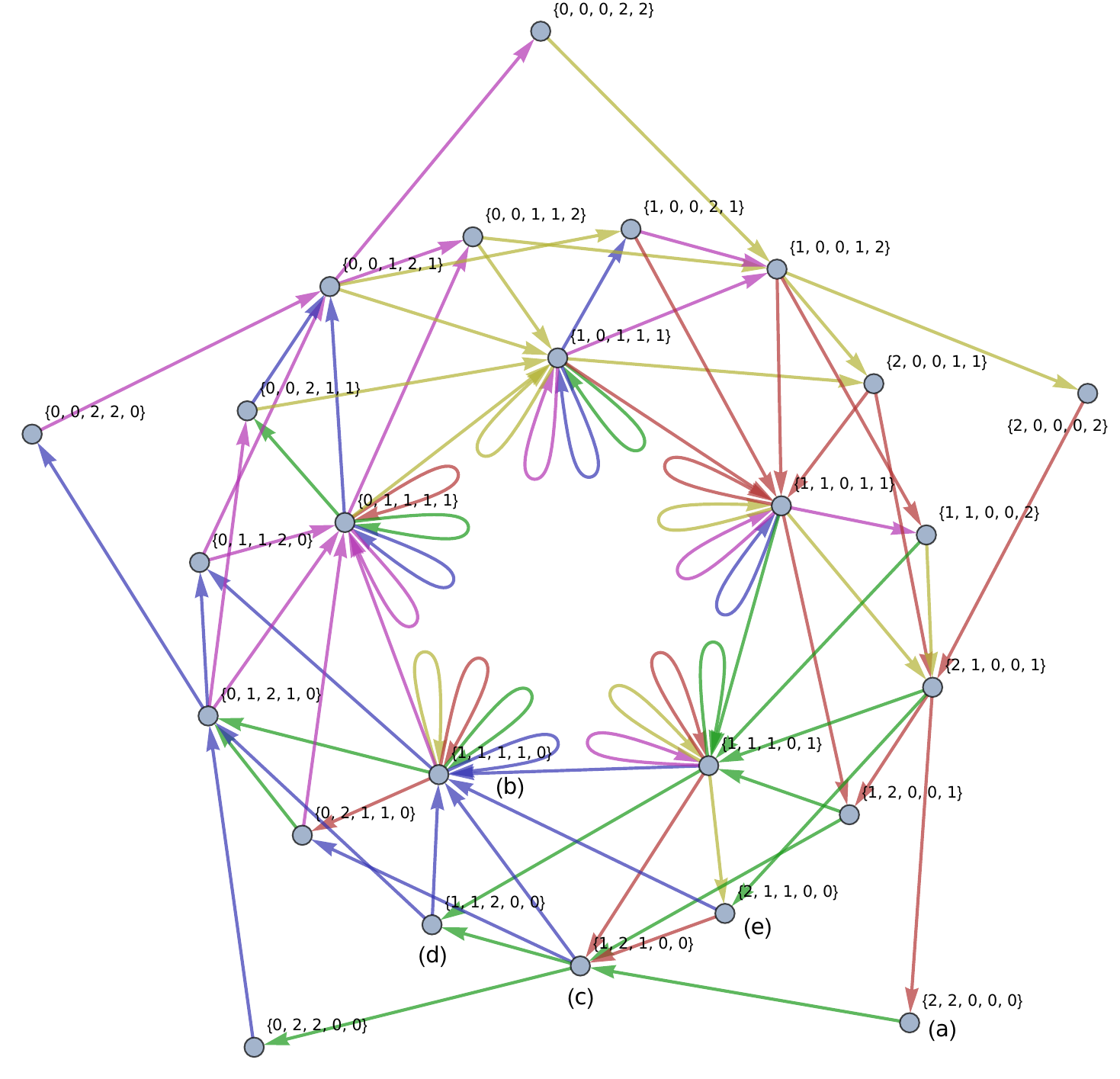}
  \caption{Abelian Rauzy graph $G_{5,4}$ of order $4$ for $\infw{t}_5$.}
  \label{fig:abr54}
\end{figure}
\end{example}

\subsection{When \texorpdfstring{$\ell < m$}{ell < m}}

\begin{proposition}\label{pro:edgml1}
  For $1\leqslant \ell<m$, the number of edges in the abelian Rauzy graph $G_{m,\ell}$ is given by \[m(1+\ell m-\ell).\]
\end{proposition}

\begin{proof}
  For $\ell=1$, all length-$2$ factors of the form $ab$ appear in $\infw{t}_m$.
  Thus, $G_{m,1}$ is a complete directed graph with $m^2$ edges.

  Now, assume $\ell\geqslant 2$.
   As a first case, let $\ell$ be even, in the form $2\ell'$, where $\ell'>0$, and $m$ is odd (as in \cref{exa:abr54}).
    \cref{tab:different_vertices}
    lists the possible Parikh vectors
     $v$ and their corresponding out-degree $d^+(v)$.
    Note that we must also consider the cyclic permutations of these vectors, which correspond to other vertices in the graph.
  \begin{table}[h!t]
   \[
    \begin{array}{c||l|lr|l}
      \text{type} & \Psi(u) & d^+ & \text{choices} & \text{total when }\ell\text{ even}\\
      \hline
      a)&2^{\ell'}0^{m-\ell} & 1 & & 1\\
      b)&1^\ell0^{m-\ell} & \ell+m-1 & & 1\\
      c)&1^i2^j1^{\ell-i-2j}0^{m-\ell+j} & 4& i,j,\ell-i-2j>0 & (\ell'-1)^2\\
      d)&1^{\ell-2i}2^i0^{m-\ell+i} & 2 & i,\ell-2i>0 & \ell'-1\\
      e)&2^i1^{\ell-2i}0^{m-\ell+i}&2 & i,\ell-2i>0 & \ell'-1 \\
      f)&1^i0^j1^{\ell-i}0^{m-\ell-j} & 2 & i,j,\ell-i,m-\ell-j>0 & (\ell'-\frac12)(m-\ell-1)\\
    \end{array}
  \]
    \caption{The different types of vertices (not counting permutations).}
    \label{tab:different_vertices}
  \end{table}
 
  We proceed similarly to the proof of \cref{pro:abco_small},  describing the Parikh vectors represented succinctly as words.

  \begin{enumerate}
      \item[(a)]
       The factor $0 1 \cdots \ell' 0 1 \cdots \ell'$ has a unique successor in $\infw{t}_m$, which is $1 \cdots \ell' 0 1 \cdots \ell' (\ell'+1)$.
       Thus, there is an edge $2^{\ell'}0^{m-\ell} \to 1 2^{\ell'-1}10^{m-\ell-1}$. The reader may refer to \cref{exa:abr54} to observe the different types of vertices described in this proof. 
       For the first type, these vertices are located on the outermost part of \cref{fig:abr54}.

\item[(b)] 
       The Parikh vector $1^\ell0^{m-\ell}$ can be associated with the factor $0\cdots (\ell-1)$. Since all pairs of letters occur in $\infw{t}_m$, the factor $0\cdots (\ell-1) a$ occurs in $\infw{t}_m$ for all $a\in\am$.
       Thus, there are 
$m$ edges with the label
       $(0,a)$; in particular, one of them is a loop with label 
 $(0,0)$. 
  This Parikh vector is also associated with a factor of the form $vu$, where $u=0\cdots (i-1)$ and $v=i\cdots (\ell-1)$, with $i=1,\ldots,\ell-1$. 
  Thus, there are $\ell-1$ loops labeled $(i,i)$. 
For the second type, these vertices are located on the innermost part of \cref{fig:abr54}.

\item[(c)] 
The Parikh vector $1^i2^j1^{\ell-i-2j}0^{m-\ell+j}$ is associated with a factor of the form $uv$ or $vu$, where $u=0\cdots (i+j-1)$ and $v=i\cdots (\ell-j-1)$.
  It can also be associated with a factor $uv$ or $vu$, where $u=0\cdots (\ell-j-1)$ and $v=i\cdots (i+j-1)$. 
  This results in four edges towards the following vertices:
 \begin{eqnarray*}
      &01^{i-1}2^j1^{\ell-i-2j+1}0^{m-\ell+j-1},
      &1^{i+1}2^{j}1^{\ell-i-2j-1}0^{m-\ell+j},\\
      &01^{i-1}2^{j+1}1^{\ell-i-2j-1}0^{m-\ell+j-1},
      &1^{i+1}2^{j-1}1^{\ell-i-2j+1}0^{m-\ell+j}.
  \end{eqnarray*}
  
  \item[(d) $\&$ (e)] 
These cases are similar.
   The Parikh vector $2^i1^{\ell-2i}0^{m-\ell+i}$ is associated with a factor of the form $uv$ or $vu$, where $u=0\cdots (i-1)$ and $v=0\cdots (\ell-i-1)$.
   This results in two edges labeled $(0,\ell-i)$ and $(0,i)$, which are distinct because $\ell-2i>0$. 

 \item[(f)]  
 We have factors of the form $uv$ or $vu$, where $u=0\cdots (i-1)$ and $v=(i+j)\cdots (\ell+j-1)$. This results in two edges labeled $(0,\ell+j)$ and $(i+j,i)$.
  \end{enumerate}

  Next, we count the total number of edges. To do so, we need to determine the number of vertices of each type. 
  There are $m$ pairwise distinct cyclic permutations of the vector of type (a).
  The same observation applies for type (b). This results in $m+m(\ell+m-1)=m(\ell+m)$ edges in $G_{m,\ell}$.

  For a vector of type (c), for each valid $j\leqslant\ell'-1$, there are $\ell-2j-1$ ways to arrange  $\ell-2j$ ones on both sides of $2^j$. 
  This results in
  \begin{equation}
    \label{eq:totl'}
    \sum_{j=1}^{\ell'-1} (\ell-2j-1)=(\ell'-1)^2.
\end{equation}
  Taking into account the cyclic permutations, we obtain $4m(\ell'-1)^2$ edges.

  For a vector of type (d) or (e), there are $\ell'-1$ choices for $i$. This resulting in a total of $4m(\ell'-1)$ edges.

  The type (f) requires extra caution: since $\ell$ is even, not all cyclic permutations are distinct, so we must avoid double counting.
  We have to limit ourselves to $i\leqslant\ell'$. Indeed, the $m$~cyclic permutations of $1^i0^j1^{\ell-i}0^{m-\ell-j}$ and  those of $1^{\ell-i}0^{m-\ell-j}1^i0^j$ are identical.
  For each $i<\ell'$, there are $m-\ell-1$ choices for~$j$. This results in  $2m(\ell'-1)(m-\ell-1)$ edges.
  When $i=\ell'$, there are two blocks of ones of the same size, giving only $(m-\ell-1)/2$ choices for $j$.
  This is the only place where the fact that $m$ is odd plays a role. This provides $2m(m-\ell-1)/2=m(m-\ell-1)$ edges.
  Summing up all contributions yields the expected value
 \[
\begin{aligned}
m (\ell + m) + 4 m (\ell' - 1)^2 &+ 4 m (\ell' - 1) 
+ 2 m (\ell' - 1) (m - \ell - 1)\\ &+ m (m - \ell - 1) = m(1+\ell m-\ell).
\end{aligned}
\]
  If $m$ is even and $i=\ell'$, we must consider $j<(m-\ell)/2$ and $j=(m-\ell)/2$ separately because in the latter case, there are also two blocks of zeroes of the same size.
  Thus, we must again avoid double counting. This results in 
  $2m\left(({m-\ell}/{2})-1\right)+m/2$
  edges. The last term corresponds to the permutations of $1^{\ell'}0^{(m-\ell)/2}1^{\ell'}0^{(m-\ell)/2}$, which can be observed in \cref{fig:abr64} with the three innermost vertices. The summation yields the same expression.

  The case where $\ell$ is odd treated similarly. Note that there are no Parikh vectors of type (a).
\end{proof}

\begin{remark}
  For $1\leqslant\ell<m$, the graph $G_{m,\ell}$ is an Eulerian graph.
  The previous proof can be reproduced by focusing on the in-degree of the vertices and show that for all vertices $v$, $d^+(v)=d^-(v)$. Since $\infw{t}_m$ is recurrent, the graph $G_{m,\ell}$ is strongly connected. This suffices to conclude.
\end{remark}

\begin{proposition}\label{pro:Ym1}
  For $1\leqslant \ell<m$, the following holds
  \[\# Y_{m,R}(\ell)=\# Y_{m,L}(\ell)=m(1+\ell m-\ell)-\frac{m}{2}\ell(\ell-1).\]
  In particular, the value of
  $\# Y_{m}(\ell)$
  is given by
  \[
  \# Y_{m}(\ell)=2m(1+\ell m-\ell)-m\ell(\ell-1).
  \]
\end{proposition}

\begin{proof}
  Assume $\ell$  is even, of the form $2\ell'$. To compute $\# Y_{m,R}(\ell)$, we must identify the edges in $G_{m,\ell}$ that are outgoing from a vertex with labels sharing the same second component.
 If such edges exist, they are counted once in $\#Y_{m,R}(\ell)$. 
 Our strategy is to subtract, from the total number of edges given by \cref{pro:edgml1}, those that do not contribute a new element to the set $\# Y_{m,R}(\ell)$.
 In \cref{exa:abr64}, to compute $\# Y_{6,R}(4)$, one must sum, for each vertex, the number of outgoing edges, counting only one edge per distinct color.
  
  Using the same notation as in the proof of  \cref{pro:edgml1}, only vertices of type (b), (c), or (d) will contribute.
  We now identify the edges whose labels share the same second component.
  The vertex $1^\ell0^{m-\ell}$ has $\ell-1$ outgoing edges  labeled $(0,j)$ and $\ell-1$ loops  labeled $(j,j)$, for $j=1,\ldots,\ell-1$. (Refer to \cref{fig:abr54,fig:abr64} to observe the vertices having loops.) 
  Considering the cyclic permutations of the Parikh vector, we must subtract $m(\ell-1)$ from the total number of edges.
  A vertex of type (c) has $2$ has two outgoing edges with a second component of $\ell-j$, and two outgoing edges with a second component of $i+j$.
  (Refer to \cref{fig:abr54,fig:abr64} to observe the vertices with an out-degree of $4$.)
  Moreover, $\ell-j\neq i+j$ since $\ell-i-2j>0$. 
  From \eqref{eq:totl'}, we must subtract $2m(\ell'-1)^2$.
  Finally a vertex of type (d) has two outgoing edges with a second component of $\ell-i$. Hence, we subtract $m(\ell'-1)$. 
  The total amount to subtract is:
  \[m \left[ \ell-1+2(\ell'-1)^2+\ell'-1\right]=
  \frac{m \ell (\ell-1)}{2}.
  \]
  The remaining cases are treated similarly.

  To determine $\# Y_{m,L}(\ell)$, we need to identify the edges in $G_{m,\ell}$ that are incoming to a vertex with labels sharing the same first component. If such edges exist, they are counted once in 
 $Y_{m,L}(\ell)$.
 Only vertices of type (b), (c), or (e) contribute.
 Refer to \cref{exa:abr54b}
 for further clarification.
  The reasoning is similar in this case.
\end{proof}

\begin{example}\label{exa:abr54b}
 \cref{fig:abr54first} depicts the graph $G_{5,4}$. 
 Compared to \cref{exa:abr54,exa:abr64},  the color of each edge is determined by the first component of its label.
Vertices are labeled with their corresponding Parikh vectors. 
\begin{figure}[h!t]
  \centering
  \includegraphics[width=11cm]{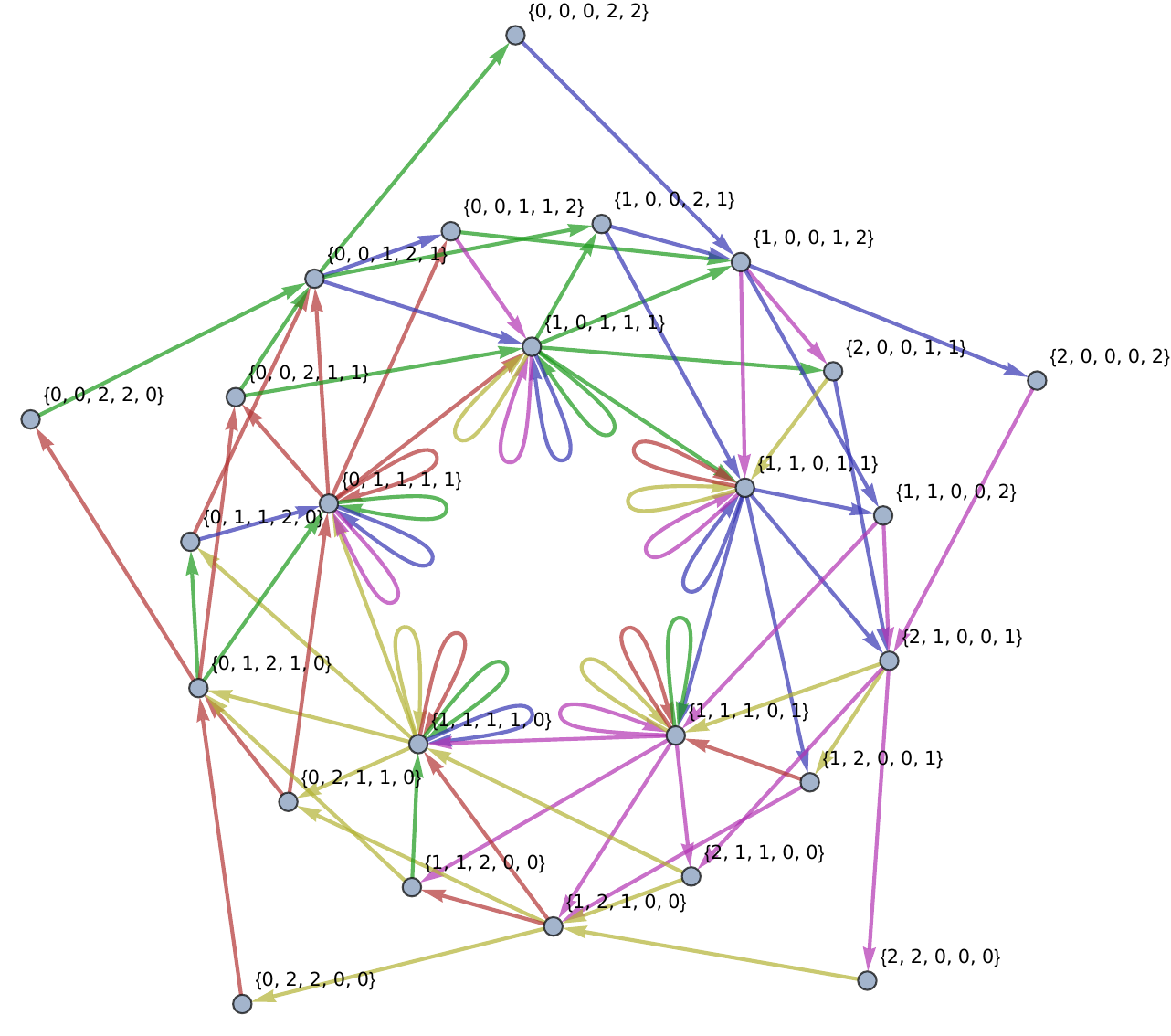}
  \caption{Abelian Rauzy graph $G_{5,4}$ of order $4$ for $\infw{t}_5$ ; edges colored by the first component of the label.}
  \label{fig:abr54first}
\end{figure}
\end{example}

\subsection{When \texorpdfstring{$\ell \geqslant m$}{ell >= m}}

\begin{proposition}\label{pro:abrl2}
  For $m\leqslant \ell<2m$, the number of edges in the abelian Rauzy graph $G_{m,\ell}$ is given by \[m(m^2-m+1).\]
\end{proposition}

\begin{proof}
  Let $b\in\am$.
  Due to the symmetry of $\sm$, we  count the number of edges labeled $(0,b)$ and then multiply the result by $m$. 
  So, we  focus on factors of length~$\ell+1$ that start with $0$ and end with $b$. These factors can be of one of the following two forms
  
  \begin{itemize}
  \item $uvb$, where $u$ starts with $0$, $|u|=t\leqslant m$, 
  and
  $|v|=\ell-t<m$, i.e.,  $\ell-m<t\leqslant m$ ; or
  
  \item  $u\sm(a)vb$, for some letter $a$,  and where $u$ starts with $0$, $|u|=t\leqslant \ell-m$, and $|v|=\ell-m-t$, i.e., $1\leqslant t\leqslant\ell-m$.
  \end{itemize}
  In both cases, $u$ (respectively, $vb$) is a suffix (respectively, prefix) of the image of a letter under $\sm$.
  In particular, all letters of $u$ are determined by the first letter~$0$, and all letters of $v$ are determined by $b$. Note that the first letter of $v$ is congruent to $b-\ell+t$ modulo~$m$.

  Consider the first case, where $\ell-b=m$. There is a single edge labeled $(0,b)$ from $2^b1^{m-b}$ to $12^b1^{m-b-1}$.
  Since $|u|=t$, the last letter of $u$ is $t-1$. Under the assumption $\ell-b=m$, the first letter of $v$ is $t$.
  Therefore, all the previously described factors have the same Parikh vector.

  Next, assume that $\ell-b\neq m$.
  We will prove that there are $m$ pairwise distinct Parikh vectors, each with an outgoing edge labeled $(0,b)$.
  Since there are $m-1$ possible values for $b$, we obtain the expected value of $m(m-1)=m^2-m$. In this case, the last letter of $u$ is $t-1$, and the first letter of $v$ is $b-\ell-t$ which is not congruent to $t$ modulo~$m$.

  First, assume that we have two factors $uvb$ and $u'v'b$ of the first form, where $|u'|=t'<|u|=t\leqslant m$. 
  Then, $\Psi(uvb)-\Psi(u'v'b)$, and also contains $1$'s in positions corresponding to $t',t'+1,\ldots,t-1$ and contains $-1$'s in positions corresponding to $b-\ell+t',\ldots,b-\ell+t-1$ (modulo~$m$). Since $\ell-b\neq m$, the two intervals of length $t-t'$, made of these positions are not equal over $\mathbb{Z}/(m\mathbb{Z})$. Therefore, $\Psi(uvb)-\Psi(u'v'b)\neq 0$.

  A similar reasoning applies to the two factors $u\sm(a)vb$ and $u'\sm(a')v'b$ of the second form.

  Finally, we  compare a factor $x=uvb$ of the first form with a factor $y=u'\sm(a)v'b$ of the second form. Let $t=|u|$ and $t'=|u'|$,
  with $\ell-m<t\leqslant m$ and $0<t'\leqslant \ell-m$. Then, $x$ and $y$ have the same prefix (respectively, suffix) of length $t'$ (respectively, $\ell-m-t'$).
  Thus,
  \[\Psi(x)-\Psi(y)=\Psi\left(t' (t'+1)\cdots (t-1) (b-\ell+t)\cdots (b-t') \right)-\Psi\left(\sm(a)\right).\]
This difference is non-zero, as $\ell-b\neq m$. Consequently, the length-$m$ word 
\[t' (t'+1)\cdots (t-1) (b-\ell+t)\cdots (b-t')\]
contains at least one repeated letter. 
\end{proof}

\begin{example}\label{exa:abr45b}
In \cref{fig:abr45first}, we have depicted the graph $G_{4,5}$. 
The color of each edge is determined by the first component of its label, as the next proof focuses on the set $Y_{m,L}$.
The vertices are labeled with their corresponding Parikh vectors.
\begin{figure}[h!t]
  \centering
  \includegraphics[width=9cm]{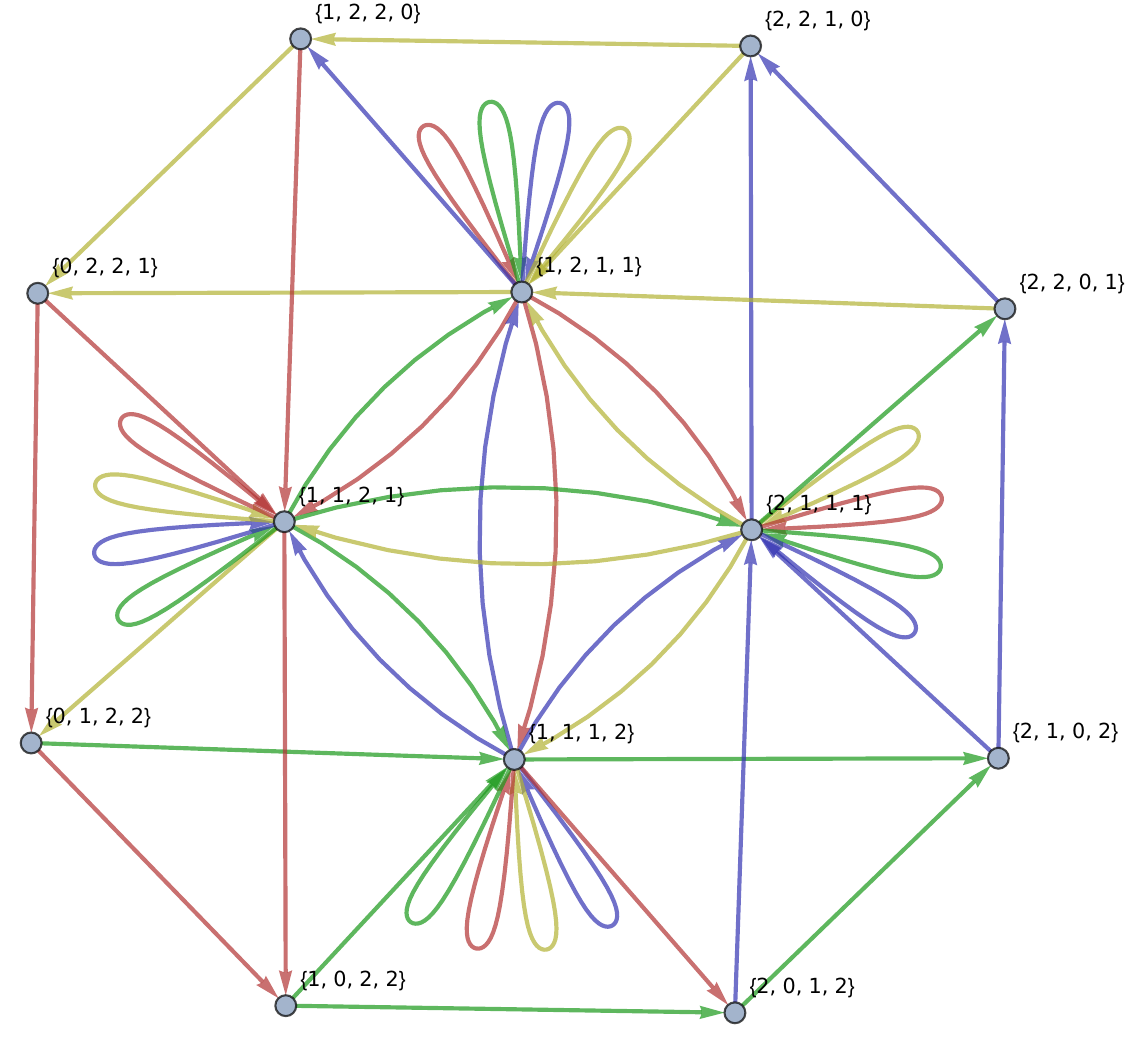}
  \caption{Abelian Rauzy graph $G_{4,5}$ of order $5$ for $\infw{t}_4$ ; edges colored by the first component of the label.}
  \label{fig:abr45first}
\end{figure}
\end{example}

\begin{proposition}
  For $m\leqslant \ell<2m$, the following holds
  \[\# Y_{m,R}=
  \# Y_{m,L}=
  \frac{m + m^2 (m - 1)}{2}.
  \]
  In particular,
   $\# Y_{m}(\ell)$
  is given by
  \[\# Y_{m}(\ell)=2m + m^2 (m - 1).\]
\end{proposition}

\begin{proof}
  We focus on $Y_{m,L}$, using the same notation as in the proof of  \cref{pro:abrl2}.
  The strategy is similar to that used in \cref{pro:Ym1}: subtracting, from the total number of edges given by \cref{pro:abrl2}, those that do not contribute a new element to the set $\# Y_{m,L}(\ell)$.

  If $\ell-b=m$, there are $m$ incoming edges labeled as $(0,i)$ for all $i\in\am$, directed to $x=12^b1^{m-b-1}$.
  The initial vertices $\Psi(0)+x-\Psi(i)$ are pairwise distinct. So we have to subtract $m-1$ from the total number of edges in $G_{m,\ell}$.
  For example, observe the four yellow vertices leading to vertex $1211$ in \cref{fig:abr45first}.
  For distinct $b,b'\neq\ell-m$, there exists a unique Parikh vector $x_{\{b,b'\}}$ with two incoming edges labeled as $(0,b)$ and $(0,b')$.
  For two such pairs $\{b,b'\}$ and $\{c,c'\}$, the corresponding vertices are such that $x_{\{b,b'\}}\neq x_{\{c,c'\}}$. Note that the number of these pairs is $\binom{m-1}{2}$. 
  In \cref{fig:abr45first},  three vertices --- namely, $0221$, $1121$ and $1112$,  each have two yellow incoming edges. So we also have to subtract $(m-1)(m-2)/2$. Thus,
  \[\# Y_{m,L}=m(m^2-m+1)-m\left[m-1+
  \frac{(m-1)(m-2)}{2}\right]=m\left(\frac{(m^2-m)}{2}+1\right).\]
  To obtain the result for $Y_{m,R}$, the reasoning remains identical; however, one has to consider edges labeled as $(b,0)$. 
\end{proof}

\begin{remark}
  For all $j\geqslant 1$ and $m\leqslant \ell<2m$, the abelian Rauzy graph $G_{m,\ell+j\cdot m}$ is isomorphic to $G_{m,\ell}$.
  Refer to the proof of \cref{pro:abrl2}. We may have factors of length~$\ell$ in one of the following two forms
  
  \begin{itemize}
  \item $uv$ where $u$ starts with $|u|=t\leqslant m$, $|v|=\ell-t<m$, i.e.,  $\ell-m<t\leqslant m$; or
  
  \item  $u'\sm(c)v'$ for some letter $c$,   where $|u'|=t\leqslant \ell-m$ and $|v'|=\ell-m-t$, i.e., $1\leqslant t\leqslant\ell-m$.
  \end{itemize}
  In both cases, $u,u'$ (respectively, $v,v'$) is a suffix of $\sm(a)$ for some letter $a$ (respectively, prefix of $\sm(b)$ for some letter $b$). By \cref{lem:boundaryseq}, there exists a factor $x$ (respectively, $y$) of length $j$ (respectively, $j+1$) such that $u\sm(x)v$ and $u'\sm(y)v'$ are factors of $\infw{t}_m$. Note that 
  \[
  \Psi\left(u\sm(x)v\right)=\Psi(uv)+j\cdot (1,\ldots,1)
  \]
   and 
   \[
   \Psi\left(u'\sm(y)v'\right)=\Psi\left(u'\sm(c)v'\right)+j\cdot (1,\ldots,1).
   \]
    These two observations show that  $G_{m,\ell+t\cdot m}$ and $G_{m,\ell}$ are the same graph up to a renaming of the vertices.
\end{remark}

The careful reader may observe that this remark provides an alternative proof of our main result, \cref{thm:main}.
Once the structure of the abelian Rauzy graphs is well understood, the formula given by \cref{pro:5.5} also provides a characterization of the $k$-binomial complexity.
The two approaches developed in this paper are, in our view, complementary.
Each approach provides its own set of combinatorial perspectives.
With this article, we have reconciled several approaches.
First, we simplified Lejeune's arguments in \cite{LLR} and considered the same type of equivalence relation for larger alphabets. Next, we applied abelian Rauzy graphs in a different context from that in \cite{RSW}.



\section{Proof of \texorpdfstring{\cref{lem:bigdiff}}{Lemma bigdiff}}\label{sec:appbigfiff}

Recall from \cref{sec:discern} that  $\overline{a}$ denotes $-a$ for $a\in\mathbb{Z}$. \cref{lem:bigdiff} is crucial for proving  \cref{lem:k-1-factorisation-first-last}.
\begin{proof}
Let $e=0 \mi{1}\cdots \mi{k}$. The subword $e$ may appear entirely in $\smn{k}(u)$, entirely in $\smn{k-1}(\gamma\delta)$, or intersects both parts. So we have
\[
\begin{aligned}
  \binom{\smn{k-1}(\gamma \sm(u) \delta)}{e} &=
      \binom{\smn{k}(u)}{e}+
      \binom{\smn{k-1}(\gamma \delta)}{e} \\
      &\quad + \sum_{\substack{e=xyz\\ 0<|y|<k+1}}
      \binom{\smn{k-1}(\gamma)}{x}\binom{\smn{k}(u)}{y} \binom{\smn{k-1}(\delta)}{z}.
\end{aligned}
\]
Since $|u|=|u'|$, by \cref{prop:phik}, $\smn{k}(u)\sim_{k}\smn{k}(u')$ and
\begin{eqnarray}
  &&\binom{\smn{k-1}(\gamma \sm(u) \delta)}{e}-
     \binom{\smn{k-1}(\gamma' \sm(u') \delta')}{e} \label{bigsumeq1} \\
  &=&
      \binom{\smn{k}(u)}{e}-\binom{\smn{k}(u')}{e}+
      \binom{\smn{k-1}(\gamma \delta)}{e}-\binom{\smn{k-1}(\gamma' \delta')}{e}\nonumber \\
  &&+
     \sum_{\substack{e=xyz\\ 0<|y|<k+1}}\binom{\smn{k}(u)}{y}
  \left[
  \binom{\smn{k-1}(\gamma)}{x} \binom{\smn{k-1}(\delta)}{z}
  -\binom{\smn{k-1}(\gamma')}{x} \binom{\smn{k-1}(\delta')}{z}\right]. \label{bigsumeq3} 
\end{eqnarray}
Observing that the factors $x$, $y$, and $z$ in the above sum are respectively of the form $x=0\mi{1}\cdots \mi{j-1}$; $y=\mi{j} \cdots \mi{j+\ell-1}$; $z=\mi{j+\ell}\cdots \mi{k}$ for $1\leqslant \ell\leqslant k$, let us rewrite term~\eqref{bigsumeq3} of the latter expression as
\[
  \sum_{\ell=1}^k
      \sum_{j=0}^{k-\ell+1}\binom{\smn{k}(u)}{\mi{j} \cdots \mi{j+\ell-1}}\left[
      \binom{\smn{k-1}(\gamma)}{0\mi{1}\cdots \mi{j-1}} \binom{\smn{k-1}(\delta)}{\mi{j+\ell}\cdots \mi{k}}
      -\binom{\smn{k-1}(\gamma')}{0\mi{1}\cdots \mi{j-1}} \binom{\smn{k-1}(\delta')}{\mi{j+\ell}\cdots \mi{k}}\right].
\]
By \cref{lem:smart}, the coefficient $\binom{\smn{k}(u)}{\mi{j} \cdots \mi{j+\ell-1}}$ equals $\binom{\smn{k}(u)}{0\mi{1} \cdots \mi{\ell-1}}$ for each $j$ since $\ell \leqslant k$; thus, the sum simplifies to
\[
  \sum_{\ell=1}^k \binom{\smn{k}(u)}{0\mi{1} \cdots \mi{\ell-1}}
  \sum_{j=0}^{k-\ell+1}\left[ \binom{\smn{k-1}(\gamma)}{0\mi{1}\cdots
      \mi{j-1}} \binom{\smn{k-1}(\delta)}{\mi{j+\ell}\cdots \mi{k}}
    -\binom{\smn{k-1}(\gamma')}{0\mi{1}\cdots \mi{j-1}}
    \binom{\smn{k-1}(\delta')}{\mi{j+\ell}\cdots \mi{k}}\right].
 \]
By \cref{lem:smart} again, we may replace $\binom{\smn{k-1}(\delta)}{\mi{j+\ell}\cdots \mi{k}}$
 with $\binom{\smn{k-1}(\delta)}{\mi{j}\cdots \mi{k-\ell}}$ and $\binom{\smn{k-1}(\delta')}{\mi{j+\ell}\cdots \mi{k}}$
 with $\binom{\smn{k-1}(\delta')}{\mi{j}\cdots \mi{k-\ell}}$, as long as $|\mi{j}\cdots \mi{k-\ell}| < k$, i.e., when $\ell \geqslant 2$ or $\ell = 1$ and $j\geqslant 1$. We decompose the sum accordingly (for convenience, we also add and subtract the same extra term)
\begin{align*}
  &\sum_{\ell=2}^k \binom{\smn{k}(u)}{0\mi{1} \cdots \mi{\ell-1}}
     \sum_{j=0}^{k-\ell+1}\left[
     \binom{\smn{k-1}(\gamma)}{0\mi{1}\cdots \mi{j-1}} \binom{\smn{k-1}(\delta)}{\mi{j}\cdots \mi{k-\ell}}
     -\binom{\smn{k-1}(\gamma')}{0\mi{1}\cdots \mi{j-1}} \binom{\smn{k-1}(\delta')}{\mi{j}\cdots \mi{k-\ell}}\right]\\
  &+ \binom{\smn{k}(u)}{0} \biggl(
      \sum_{j=1}^{k}\left[
      \binom{\smn{k-1}(\gamma)}{0\mi{1}\cdots \mi{j-1}} \binom{\smn{k-1}(\delta)}{\mi{j}\cdots \mi{k-1}}
      -\binom{\smn{k-1}(\gamma')}{0\mi{1}\cdots \mi{j-1}} \binom{\smn{k-1}(\delta')}{\mi{j}\cdots \mi{k-1}}\right]\\
      &+ \binom{\smn{k-1}(\delta)}{0\mi{1}\cdots \mi{k-1}} -\binom{\smn{k-1}(\delta')}{0\mi{1}\cdots \mi{k-1}}\\
  &+  \binom{\smn{k-1}(\delta)}{\mi{1}\cdots \mi{k}}
      - \binom{\smn{k-1}(\delta')}{\mi{1}\cdots \mi{k}} - \left[ \binom{\smn{k-1}(\delta)}{0\mi{1}\cdots \mi{k-1}} -\binom{\smn{k-1}(\delta')}{0\mi{1}\cdots \mi{k-1}} \right] \biggr).
\end{align*}
Since
 \[
   \sum_{j=0}^{k-\ell+1}\binom{\smn{k-1}(x)}{0\mi{1}\cdots \mi{j-1}} \binom{\smn{k-1}(y)}{\mi{j}\cdots \mi{k-\ell}} = \binom{\smn{k-1}(xy)}{0\mi{1}\cdots \mi{k-\ell}}
 \]
for any words $x$, $y$, we further simplify to
\begin{multline}
  \sum_{\ell=1}^k\binom{\smn{k}(u)}{0\mi{1} \cdots \mi{\ell-1}} \left[
    \binom{\smn{k-1}(\gamma\delta)}{0\mi{1}\cdots \mi{k-\ell}}
    - \binom{\smn{k-1}(\gamma'\delta')}{0\mi{1}\cdots \mi{k-\ell}}\right]\\
  + m^{k-1}|u|\left(
    \binom{\smn{k-1}(\delta)}{\mi{1}\cdots\mi{k}} -
    \binom{\smn{k-1}(\delta)}{0\cdots\mi{k-1}} -
    \binom{\smn{k-1}(\delta')}{\mi{1}\cdots\mi{k}} +
    \binom{\smn{k-1}(\delta')}{0\cdots\mi{k-1}}
  \right). \label{eq:almost_simplified}
\end{multline}
Now $\binom{\smn{k-1}(\delta)}{\mi{1}\cdots\mi{k}} = \binom{\smn{k-1}\tau_m(\delta)}{0\cdots\mi{k-1}}$, where we recall that $\tau_m$ is the morphism defined by $\tau_m(i) = i+1$. Thus, by \cref{cor:notequiv}, the second term in \eqref{eq:almost_simplified} simplifies to:
\begin{equation}\label{eq:missing}
\begin{aligned}
   m^{k-1}|u|\left(m^{\binom{k-1}{2}}\left(|\delta+1|_0 - |\delta|_0 - |\delta'+1|_0 + |\delta'|_0 \right)\right) \\
   = m^{\binom{k}{2}}|u|\left(|\delta|_{\mi{1}} - |\delta|_0 - |\delta'|_{\mi{1}} + |\delta'|_0 \right).
\end{aligned}
\end{equation}
      
Consider the sum appearing in \eqref{eq:almost_simplified}. Since $|\delta\gamma|=|\delta'\gamma'|$, by \cref{prop:phik}, $\smn{k-1}(\gamma\delta)\sim_{k-1}\smn{k-1}(\gamma'\delta')$, and the sum reduces to a single term (corresponding to $\ell=1$)
  \[
    \binom{\smn{k}(u)}{0}
    \left[ \binom{\smn{k-1}(\gamma\delta)}{0\mi{1}\cdots \mi{k-1}}
      - \binom{\smn{k-1}(\gamma'\delta')}{0\mi{1}\cdots \mi{k-1}} \right]
    = m^{k-1} |u|\, \left(|\gamma\delta|_0-|\gamma'\delta'|_0\right) m^{\binom{k-1}{2}}
  \]
(where we have used \cref{cor:notequiv}) and is equal to
  \[
    |u|\, m^{\binom{k}{2}} \left(|\gamma\delta|_0-|\gamma'\delta'|_0\right).
  \]
  
 We can now return to the initial difference~\eqref{bigsumeq1} of interest. By applying \cref{cor:notequiv} again, we get that ~\eqref{bigsumeq1} is equal to
\[
\begin{aligned}
  &\binom{\smn{k-1}(\gamma \delta)}{e}-\binom{\smn{k-1}(\gamma' \delta')}{e} + \\
  &m^{\binom{k}{2}} \biggl[ |u|_0 - |u'|_0 + |u| \, \bigl( |\gamma\delta|_0 - |\gamma'\delta'|_0 + |\delta|_{\mi{1}} - |\delta|_0 - |\delta'|_{\mi{1}} + |\delta'|_0 \bigr) \biggr].
\end{aligned}
\]

To conclude the proof, we develop the difference between the first two terms.
Let $\gamma\delta=x_1\cdots x_t$ and $\gamma'\delta'=x_1'\cdots x_t'$. We use the same argument as in the proof of  \cref{prop:-1}.
We need to count occurrences of the subword $e$. 
If an occurrence is split across multiple $m^{k-1}$-blocks and at most $k-1$ letters appear in any block, then these occurrences will cancel because $\smn{k-1}(x_i)\sim_{k-1}\smn{k-1}(x_i')$. We only have to consider occurrences where at least $k$ letters (out of $k+1$) appear in the same $m^{k-1}$-block. Then, we look at $e$ occurring entirely within one $m^{k-1}$-block, given by the following expression
\[\sum_{i=1}^t \left( \binom{\smn{k-1}(x_i)}{e} - \binom{\smn{k-1}(x_i')}{e} \right)\]
and this sum vanishes because $\gamma\delta\sim_1\gamma'\delta'$. 
Alternatively, $e$ is split with $k$ letters in one $m^{k-1}$-block and one (the first or the last) in another $m^{k-1}$-block, we obtain
\begin{eqnarray*}
  &&
     \sum_{i=1}^{t-1}\sum_{j=i+1}^t  \left( \binom{\smn{k-1}(x_i)}{0} \binom{\smn{k-1}(x_j)}{\mi{1}\cdots \mi{k}}- \binom{\smn{k-1}(x_i')}{0}  \binom{\smn{k-1}(x_j')}{\mi{1}\cdots \mi{k}}\right)\\
  &+& \sum_{i=1}^{t-1}\sum_{j=i+1}^t  \left( \binom{\smn{k-1}(x_i)}{0\,\mi{1}\cdots \mi{k-1}} \binom{\smn{k-1}(x_j)}{\mi{k}}- \binom{\smn{k-1}(x_i')}{0\,\mi{1}\cdots \mi{k-1}}  \binom{\smn{k-1}(x_j')}{\mi{k}}\right).
\end{eqnarray*}
We get
\begin{eqnarray*}
  &&
     \sum_{i=1}^{t-1}\sum_{j=i+1}^t  m^{k-2} \left(  \binom{\smn{k-1}(x_j+1)}{0\mi{1}\cdots \mi{k-1}}-   \binom{\smn{k-1}(x_j'+1)}{0\mi{1}\cdots \mi{k-1}}\right)\\ 
  &+& \sum_{i=1}^{t-1}\sum_{j=i+1}^t  m^{k-2} \left( \binom{\smn{k-1}(x_i)}{0\,\mi{1}\cdots \mi{k-1}}- \binom{\smn{k-1}(x_i')}{0\,\mi{1}\cdots \mi{k-1}} \right).
\end{eqnarray*}
By \cref{cor:notequiv}, it is equal to 
\begin{eqnarray*}
  &&
     \sum_{i=1}^{t-1}\sum_{j=i+1}^t  m^{k-2} m^{\binom{k-1}{2}}(|x_j|_{\mi{1}}-|x_j'|_{\mi{1}})          + \sum_{i=1}^{t-1}\sum_{j=i+1}^t  m^{k-2}  m^{\binom{k-1}{2}}(|x_i|_0-|x_i'|_0)
\end{eqnarray*}
which can be rewritten as
\[
  m^{k-2} m^{\binom{k-1}{2}} \sum_{j=2}^t (j-1) \left(|x_j|_{\mi{1}}-|x_j'|_{\mi{1}}\right)          + m^{k-2}  m^{\binom{k-1}{2}} \sum_{i=1}^{t-1} (t-i) \left(|x_i|_0-|x_i'|_0\right).
\]
If $x_j=\mi{1}$, the factor $j-1$ represents the number of letters to the left of  $x_j$ and if $x_i=0$, the factor $t-i$ represents the number of letters to the right of $x_i$.
Therefore, we can write
\[
  m^{k-2} m^{\binom{k-1}{2}} \sum_{b\in\am}\left( \binom{\gamma\delta}{b\mi{1}}-\binom{\gamma'\delta'}{b\mi{1}}+ \binom{\gamma\delta}{0b}-\binom{\gamma'\delta'}{0b}\right).
\]
\end{proof}

\bibliographystyle{plainurl}
\bibliography{./bibliography}

\end{document}